\documentclass[english]{smfart}
\title[Comparison of types]{Comparing Bushnell-Kutzko and S\'echerre's constructions of types for $\mathrm{GL}_{N}$ and its inner forms with Yu's construction}

\alttitle{Comparer les constructions de Bushnell-Kutzko et Sécherre de types pour $GL_N$ et ses formes intérieurs à la construction de Yu}

\author{Arnaud Mayeux, Yuki Yamamoto}
\usepackage{mathrsfs}
\usepackage{amssymb}
\usepackage{color}
\usepackage[all]{xy}
\newtheorem{defn}{Definition}[section]
\newtheorem{rem}[defn]{Remark}
\newtheorem{thm}[defn]{Theorem}
\newtheorem{prop}[defn]{Proposition}
\newtheorem{lem}[defn]{Lemma}
\newtheorem{cor}[defn]{Corollary}

\newenvironment{prf}{Proof. }{\hfill$\Box$}

\DeclareMathOperator{\Hom}{Hom}
\DeclareMathOperator{\GL}{GL}
\DeclareMathOperator{\id}{id}
\DeclareMathOperator{\End}{End}
\DeclareMathOperator{\Aut}{Aut}
\DeclareMathOperator{\M}{M}
\DeclareMathOperator{\Res}{Res}
\DeclareMathOperator{\Ind}{Ind}
\DeclareMathOperator{\cInd}{c-Ind}
\DeclareMathOperator{\ord}{ord}
\DeclareMathOperator{\Lie}{Lie}
\DeclareMathOperator{\Tr}{Tr}
\DeclareMathOperator{\Trd}{Trd}
\DeclareMathOperator{\Nrm}{N}
\DeclareMathOperator{\Nrd}{Nrd}
\DeclareMathOperator{\Gal}{Gal}
\DeclareMathOperator{\Cent}{Cent}
\DeclareMathOperator{\Stab}{Stab}
\DeclareMathOperator{\Latt}{Latt}
\DeclareMathOperator{\sr}{sr}
\newcommand{\Afra}{\mathfrak{A}}
\newcommand{\Bfra}{\mathfrak{B}}
\newcommand{\Bscr}{\mathscr{B}}
\newcommand{\gfra}{\mathfrak{g}}
\newcommand{\Kfra}{\mathfrak{K}}
\newcommand{\Lcal}{\mathcal{L}}

\newcommand{\ofra}{\mathfrak{o}}
\newcommand{\Pfra}{\mathfrak{P}}
\newcommand{\pfra}{\mathfrak{p}}
\newcommand{\Qfra}{\mathfrak{Q}}
\newcommand{\sfra}{\mathfrak{s}}
\newcommand{\Phibf}{\mathbf{\Phi}}
\newcommand{\dr}{\mathrm{d}}
\newcommand{\rbf}{\mathbf{r}}
\newcommand{\U}{\mathbf{U}}
\newcommand{\N}{\mathbb{N}}
\newcommand{\Z}{\mathbb{Z}}
\newcommand{\R}{\mathbb{R}}
\newcommand{\C}{\mathbb{C}}
\newcommand{\nr}{\mathrm{nr}}

\begin{document}
\maketitle

\begin{center}

\textit{Dedicated to Colin J. Bushnell. }

\end{center}

\bigskip

\begin{abstract}
Let $F$ be a non-archimedean local field with odd residual characteristic, $A$ be a central simple $F$-algebra, and $G$ be the multiplicative group of $A$.  
To construct types for complex supercuspidal representations of $G$, simple types by S\'echerre and Yu's construction are already known.  
In this paper, we compare these constructions.  
In particular, we show essentially tame supercuspidal representations of $G$ defined by Bushnell--Henniart are nothing but tame supercuspidal representations defined by Yu.  
\end{abstract}

\tableofcontents

\section{Introduction}
Let $F$ be a non-archimedean local field such that the residual characteristic $p$ is odd, and let $G$ be the group of $F$-points of a connected reductive group defined over $F$.  
The aim of type theory is to classify, up to some natural equivalence, the smooth irreducible complex representations of $G$ in terms of representations of compact open subgroups.
For complex supercuspidal representations of $G$, some constructions of types are known.  

For example, Bushnell--Kutzko \cite{BK1} constructed types, called simple types,  for any irreducible supercuspidal representations when $G=\GL_{N}(F)$. 
S\'echerre \cite{S1}, \cite{S2}, \cite{S3}, and S\'echerre--Stevens \cite{SS} extended the construction of simple types to any irreducible supercuspidal representations of any inner form $G$ of $\GL_{N}(F)$.  

For an arbitrary reductive group $G$, Yu's construction \cite{Yu} (cf. also \cite{Ad} for a similar pioneering method) provides some supercuspidal representations. In his paper \cite[p580]{Yu}, Yu wrote "\textit{In particular, our method should yield all supercuspidal
representations when $p$ is large enough relative to the type of $G.$}", Yu also wrote "\textit{it is possible that our method
yields all supercuspidal representations that deserve to be called tame}". Yu's expectations are now theorems by works of Kim \cite{Ki} and Fintzen \cite{Fi}. More precisely, Yu's construction yields all supercuspidal representations if $p$ does not divide the order of the Weyl group of $G$ by \cite{Fi}, a condition that guarantees that all tori of $G$ split over a tamely ramified field extension of $F$.

It is a natural question whether there exists some relationship between these constructions of types. A natural motivation being to unify or generalize these constructions by taking advantage of each theory. In his paper \cite[p581]{Yu}, Yu wrote "\textit{However, the real difficulty in the wild case is that considerably
different} (authors note: than his) \textit{ constructions should be involved as revealed in the $GL_n$ case by the work of Bushnell, Corwin, and Kutzko.}"  The goal of the present article is to compare carefully Bushnell-Kutzko and Sécherre's constructions to Yu's one.   

From now on, let $A$ be a finite dimensional central simple $F$-algebra, and let $V$ be a simple left $A$-module.  
Then $\End_{A}(V)$ is a central division $F$-algebra.  
Let $D$ be the opposite algebra of $\End_{A}(V)$.  
Then $V$ is also a right $D$-module and we have $A \cong \End_{D}(V)$.  
Let $G$ be the multiplicative group of $A$.  
Then we have $G \cong \GL_{m}(D)$, which is the group of $F$-points of an inner form of $\GL_{N}$.    

We introduce our main theorem.  
In Yu's construction, from a tuple $\Psi=(x, (G^{i})_{i=0}^{d}, (\rbf_{i})_{i=0}^{d}, (\Phibf_{d})_{i=0}^{d}, \rho)$, which is called a Yu datum of $G$, one constructs some open subgroups ${}^{\circ}K^{d}(\Psi), \, K^{d}(\Psi)$ in $G$ and an irreducible representation $\rho_{d}(\Psi)$ of $K^{d}(\Psi)$, which are precisely defined in \S \ref{Yu's_types}.  

\begin{thm}[Theorem \ref{Main1}]
\label{main1}
Let $(J, \lambda)$ be a simple type for an essentially tame supercuspidal representation $\pi$, and let $(\tilde{J}, \Lambda)$ be an extension of $(J, \lambda)$ such that $\pi \cong \cInd_{\tilde{J}}^{G} \Lambda$.  
Then there exists a Yu datum $\Psi=(x, (G^{i})_{i=0}^{d}, (\rbf_{i})_{i=0}^{d}, (\Phibf_{d})_{i=0}^{d}, \rho)$ satisfying the following conditions:  
\begin{enumerate}
\item $J={}^{\circ}K^{d}(\Psi)$, 
\item $\tilde{J} \subset K^{d}(\Psi)$, and
\item $\rho_{d}(\Psi) \cong \cInd_{\tilde{J}}^{K^{d}(\Psi)} \Lambda$.  
\end{enumerate}
\end{thm}

\begin{thm}[Theorem \ref{Main2}]
\label{main2}
Conversely, let $\Psi=\left( x, (G^{i}), (\rbf_{i}), (\Phibf_{i}), \rho \right)$ be a Yu datum of $G$.  
Then there exists a tame simple type $(J, \lambda)$ and a maximal extension $(\tilde{J}, \Lambda)$ of $(J, \lambda)$ such that
\begin{enumerate}
\item ${}^{\circ}K^{d}(\Psi) = J$, 
\item $K^{d}(\Psi) \supset \tilde{J}$, and
\item $\rho_{d}(\Psi) \cong \cInd_{\tilde{J}}^{K^{d}(\Psi)} \Lambda$.  
\end{enumerate}
\end{thm}

By these theorems, we obtain the following corollary.  

\begin{cor}[Corollary \ref{tame_and_esstame}]
For any inner form $G$ of $\GL_{N}(F)$, the set of essentially tame supercuspidal representations of $G$ is equal to the set of tame supercuspidal representations of $G$ defined by Yu \cite{Yu}.  
\end{cor}

In particular, for $G=\GL_{N}(F)$, the statements of the above theorems are as follows: 

\begin{thm}
\label{main_for_split_case}
Let $G=\GL_{N}(F)$.  
Then $\tilde{J} = K^{d}(\Psi)$ in the statement of (2) in Theorem \ref{main1} and Theorem \ref{main2}, and $\rho_{d}(\Psi) \cong \Lambda$ in the statement of (3) in these theorems.  
\end{thm}

We sketch the outline of this paper and the main steps to prove Theorems \ref{main1} and \ref{main2}.  
First, in \S \ref{Secherre} and \ref{Yu's_types}, we recall constructions of types.  
We explain simple types of $G$ by S\'echerre in \S \ref{Secherre} and Yu's construction of tame supercuspidal representations in \S \ref{Yu's_types}.  
Next, in \S \ref{TameSimple}-\ref{Factorization}, we prepare ingredients to compare these two constructions.  
A class of simple types corresponding to Yu's type is defined in \S \ref{TameSimple}.  
In \S \ref{TTLS}, we determine tame twisted Levi subgroups in $G$.  
For some tame twisted Levi subgroup $G'$ in $G$ and some ``nice" $x $ in the enlarged Bruhat-Tits building $ \Bscr^{E}(G', F)$, we obtain another description of Moy--Prasad filtration on $G'(F)$ attached to $x$, using hereditary orders, in \S \ref{BroussousLemaire}.  
Then we can compare the groups which the types are defined as representations of.  
In \S \ref{Generic}, we discuss generic elements and generic characters.  
We relate them to some defining sequence of some simple stratum.  
In \S \ref{Depth0}, we show some lemmas on simple types of depth zero.  
These lemmas are used to take ``depth-zero" parts of types.  
In \S \ref{Factorization}, we represent a simple character with a tame simple stratum as a product of characters.  
This factorization is needed to construct generic characters.  
Finally, in \S \ref{SStoYu} and \ref{YutoSS}, we prove the main theorem.  
In \S 9, from tame S\'echerre data, which are used to construct tame simple types, we construct Yu data.  
By comparing these types constructed by these data and comfirming some kind of match between the two, we show that tame simple types can be constructed from Yu's types.  
Conversely, we also show that Yu's types are constructed from tame simple types in \S \ref{YutoSS}.  In \S 12, we briefly discuss the wild case.

\begin{rem}
This paper is the result of a project whose first traces appeared publicly in 2017.  
This brings together and extends previous works of the authors that are not intented to be published in journals.  
These works are precisely:\begin{enumerate} \item[-] 
 \textnormal{A. Mayeux}: \textit{Représentations supercuspidales: comparaison des construction de Bushnell-Kutzko et Yu},\textnormal{ arXiv:1706.05920, 2017, unpublished.}
 \item[-] 
\textnormal{A. Mayeux}: \textit{Comparison of Bushnell-Kutzko and Yu's constructions of supercuspidal representations},\textnormal{ arXiv:2001.06259, 2020, unpublished.}
\item[-]  \textnormal{Y. Yamamoto}: {\it 
Comparison of types for inner forms of $GL_N$},\textnormal{ arXiv:2005.02622, 2020, unpublished. }
\item[-]  The parts about the comparison of our PhD thesis defended at Paris in 2019 and Tokyo in 2022. 
\end{enumerate}
Our present paper is mathematically self-contained in the sense that it does not rely on the previously mentioned unpublished works.  
\end{rem}
\bigbreak

\noindent{\bfseries Acknowledgment}\quad Arnaud Mayeux was supported
by ISF grant 1577/23.
This project has received funding from the European Research Council (ERC) under the European Union’s Horizon 2020 research and innovation programme (A.M. grant agreement No 101002592 (2021-2023)).
The first author was previously funded by a Boya postdoctoral fellowship of BICMR and Peking University (2019-2021) and a doctoral fellowship of Université Paris Cité (2015-2019).  
The second author is supported by the FMSP program at Graduate School of Mathematical Sciences, the University of Tokyo.  
He was also supported by JSPS KAKENHI Grant Number JP21J13751.  
We thank Anne-Marie Aubert, Colin J. Bushnell, Naoki Imai, Guy Henniart, Kazuki Tokimoto, Vincent S\'echerre, Jeffrey Adler and Paul Broussous for many help, comments and suggestions.  
The first author thank Jessica Fintzen, Wee Teck Gan and Shuichiro Takeda for the organization of the workshop \textit{New Developments in Representation Theory of p-adic Groups, Oberwolfach, 2019} and the opportunity to talk about supercuspidal representations \cite{FGT}. Finally, we are grateful to the referee for his careful reading and many helpful suggestions and improvements.
\bigbreak

\noindent{\bfseries Notation}\quad

If $X$ is a scheme over a base scheme $S$ and if $T\to S$ is a morphism of schemes, then $X_T $ denotes $X\times _S T$ and is seen as a scheme over $T$. 

Let $G \to S$ be a group scheme acting on a scheme $X/S$. The functor  of fixed points, by definition, sends a scheme $T$ over $S$ with an action of $G$ to $\{x \in \Hom _T ( T , X_T ) | x \text{ is } G_T\text{-equivariant} \}$ where $T$ is endowed with the trivial action of $G_T$. This $S$-functor is denoted  by $X^G$. Note that for any scheme $T$ over $S$, we have $X^G (T) \subset X (T)$. It is known that this functor is representable by a scheme in many cases (cf. e.g. \cite[Exp. 12 Prop. 9.2]{SGA3}). 

In this paper, we consider smooth representations over $\C$.  
We fix a non-archimedean local field $F$ such that the residual characteristic $p$ is odd.    
For a finite-dimensional central division algebra $D$ over $F$, let $\ofra_{D}$ be the ring of integers, $\pfra_{D}$ be the maximal ideal of $\ofra_{D}$, and let $k_{D}$ be the residual field of $D$.  
We fix a smooth, additive character $\psi:F \to \C^{\times}$ of conductor $\pfra_{F}$.  
For a finite field extension $E/F$, let $v_{E}$ be the unique surjective valuation $E \to \Z \cup \{ \infty \}$.  
Moreover, for any element $\beta$ in some algebraic extension field of $F$, we put $\ord(\beta) = e(F[\beta]/F)^{-1}v_{F[\beta]}(\beta)$.  

If $K$ is a field and $G$ is a $K$-group scheme, then $\underline{\Lie}(G)$ denotes the Lie algebra functor and we put $\Lie(G)=\underline{\Lie}(G)(K)$.  
If a $K$-group scheme is denoted by a capital letter $G$, the Lie algebra functor of $G$ is denoted by the same small Gothic letter $\gfra$.  
We also denote by $\Lie^{*}(G)$ or $\gfra^{*}(K)$ the dual of $\Lie(G)=\gfra(K)$.  
For connected reductive group $G$ over $F$, we denote by $\Bscr^{E}(G, F)$ (resp. $\Bscr^{R}(G, F)$) the enlarged Bruhat--Tits building (resp. the reduced Bruhat--Tis building) of $G$ over $F$ defined in \cite{BT}, \cite{BT2}.  
If $x \in \Bscr^{E}(G, F)$, we denote by $[x]$ the image of $x$ via the canonical surjection $\Bscr^{E}(G, F) \to \Bscr^{R}(G, F)$.  
The group $G(F)$ acts on $\Bscr^{E}(G, F)$ and $\Bscr^{R}(G, F)$.  
For $x \in \Bscr^{E}(G, F)$, let $G(F)_{x}$ (reps. $G(F)_{[x]}$) denote the stabilizer of $x \in \Bscr^{E}(G, F)$ (resp. $[x] \in \Bscr^{R}(G, F)$).  
We denote by $\tilde{\R}$ the totally ordered commutative monoid $\R \cup \{ r+ \mid r \in \R \}$.  
When $G$ splits over some tamely ramified extension of $F$, for $x \in \Bscr^{E}(G, F)$ let $\{ G(F)_{x, r} \}_{r \in \tilde{\R}_{\geq 0}}$, $\{ \gfra(F)_{x,r} \}_{r \in \tilde{\R}}$ and $\{ \gfra^{*}(F)_{x,r} \}_{r \in \tilde{\R}}$ be the Moy--Prasad filtration \cite{MP1}, \cite{MP2} on $G(F)$, $\gfra(F)$ and $\gfra^{*}(F)$, respectively.  
Here, we have $\gfra^{*}(F)_{x,r} = \{ X^{*} \in \gfra^{*}(F) \mid X^{*}(\gfra(F)_{x, (-r)+}) \subset \pfra_{F} \}$ for $r \in \R$.  
If $G$ is a torus, Moy--Prasad filtrations are independent of $x$, and then we omit $x$.  

Let $G$ be a group, $H$ be a subgroup in $G$ and $\lambda$ be a representation of $H$.  
Then we put ${}^{g}H=gHg^{-1}$ for $g \in G$, and we define a ${}^{g}H$-representation ${}^{g} \lambda$ as ${}^{g}\lambda(h) = \lambda(g^{-1}hg)$ for $h \in {}^{g}H$.  
Moreover, we also put
\[
	I_{G}( \lambda ) = \{ g \in G \mid \Hom_{H \cap {}^{g}H}(\lambda, {}^{g}\lambda) \neq 0 \}.  
\]

\section{Simple types by S\'echerre}
\label{Secherre}

We recall the theory of simple types from \cite{S1}, \cite{S2}, \cite{S3}, \cite{SS}.  
In this section, we can omit the assumption that the residual characteristic of $F$ is odd.  

\subsection{Lattices, hereditary orders}

Let $D$ be a finite-dimensional central division $F$-algebra.  
Let $V$ be a right $D$-module with $\dim_{F}V < \infty$.  
We put $A=\End_{D}(V)$, and then $A$ is a central simple $F$-algebra.  
Moreover, there exists $m \in \Z_{>0}$ such that $A \cong \M_{m}(D)$.  
Let $G$ be the multiplicative group of $A$, and then $G$ is isomorphic to $\GL_{m}(D)$.  
We also put $d = (\dim _{F} D)^{1/2}$ and $N=md$.  

An $\ofra_{D}$-submodule $\Lcal$ in $V$ is called a lattice if and only if $\Lcal$ is a compact open submodule.  

\begin{defn}[{{\cite[D\'efinition 1.1]{S1}}}]
For $i \in \Z$, let $\Lcal_{i}$ be a lattice in $V$.  
We say that $\Lcal = (\Lcal_{i})_{i \in \Z}$ is an $\ofra_{D}$-sequence if
\begin{enumerate}
\item $\Lcal_{i} \supset \Lcal_{j}$ for any $i<j$, and
\item there exists $e \in \Z_{>0}$ that $\Lcal_{i+e} = \Lcal_{i}\pfra_{D}$ for any $i$.  
\end{enumerate}
The number $e=e(\Lcal)$ is called the $\ofra_{D}$-period of $\Lcal$.  
An $\ofra_{D}$-sequence $\Lcal$ is called an $\ofra_{D}$-chain if $\Lcal_{i} \supsetneq \Lcal_{i+1}$ for every $i$.  
An $\ofra_{D}$-chain $\Lcal$ is called uniform if $[\Lcal_{i} : \Lcal_{i+1}]$ is constant for any $i$.  
\end{defn}

An $\ofra_{F}$-subalgebra $\Afra$ in $A$ is called a hereditary $\ofra_{F}$-order if every left and right ideal in $\Afra$ is $\Afra$-projective.  

We explain the relationship between $\ofra_{D}$-sequences in $V$ and hereditary $\ofra_{F}$-orders in $A$ from \cite[1.2]{S1}.  
Let $\Lcal = (\Lcal_{i})$ be an $\ofra_{D}$-sequence in $V$.  
For $i \in \Z$, we put
\[
	\Pfra_{i}(\Lcal) = \{ a \in A \mid a \Lcal_{j} \subset \Lcal_{i+j}, \, j \in \Z \}.  
\]
Then $\Afra = \Pfra_{0}(\Lcal)$ is a hereditary $\ofra_{F}$-order with the radical $\Pfra(\Afra) = \Pfra_{1}(\Lcal)$.  
For every hereditary $\ofra_{F}$-order $\Afra$ in $V$, there exists an $\ofra_{D}$-chain $\Lcal$ in $V$ such that $\Afra = \Afra(\Lcal)$.  
If $\Lcal$ is a uniform $\ofra_{D}$-chain, $\Afra = \Afra(\Lcal)$ is called principal.  

For any $\ofra_{D}$-chain $\Lcal = (\Lcal_{i})$, let $\Kfra(\Lcal)$ be the set of $g \in G$ such that there exists $n \in \Z$  satisfying $g \Lcal_{i} = \Lcal_{i+n}$ for any $i$.  
For the hereditary $\ofra_{F}$-order $\Afra=\Afra(\Lcal)$, let $\Kfra(\Afra)$ be the set of $g \in G$ such that $g\Afra g^{-1}=\Afra$.  
Then we have $\Kfra(\Afra) = \Kfra(\Lcal)$ and $\Kfra(\Afra)$ is compact modulo center.  

For $g \in \Kfra(\Afra)$, there exists a unique element $n \in \Z$ such that $g\Afra = \Pfra(\Afra)^{n}$.  
The map $g \mapsto n$ induces a group homomorphism $v_{\Afra} : \Kfra(\Afra) \to \Z$.  
Let $\U(\Afra)$ be the kernel of $v_{\Afra}$.  
Then we have $\U(\Afra) = \Afra^{\times}$ and $\U(\Afra)$ is the unique maximal compact open subgroup in $\Kfra(\Afra)$.  
We put $\U^{0}(\Afra) = \U(\Afra)$ and $\U^{n}(\Afra) = 1 + \Pfra(\Afra)^{n}$ for $n \in \Z_{>0}$.  
We also put $e(\Afra|\ofra_{F}) = v_{\Afra}(\varpi_{F})$, and then we have $e(\Afra|\ofra_{F})=de(\Lcal)$ for an $\ofra_{D}$-chain $\Lcal$ in $V$ such that $\Afra = \Afra(\Lcal)$.  

Let $E$ be an extension field of $F$ in $A$.  
Since $A$ is a central simple $F$-algebra, the centralizer $B=\Cent_{A}(E)$ of $E$ in $A$ is a central simple $E$-algebra.  
On the other hand, $V$ is equipped with an $E$-vector space structure via $E \subset A$.  
Since the actions of $E$ and $D$ in $V$ are compatible, $V$ is also equipped with a right $D \otimes_{F} E$-module structure, and then we have $B = \Cent_{A}(E) = \End_{D \otimes_{F} E}(V)$.  

Let $\Afra$ be a hereditary $\ofra_{F}$-order in $A$.  
The order $\Afra$ is called $E$-pure if we have $E^{\times} \subset \Kfra(\Afra)$.  

\begin{prop}[{{\cite[Theorem 1.3]{Br}}}]
For an $E$-pure hereditary $\ofra_{F}$-order $\Afra$ in $A$, the subring $\Bfra = \Afra \cap B$ in $B$ is a hereditary $\ofra_{E}$-order in $B$ with the radical $\Qfra = \Pfra(\Afra) \cap B$.  
\end{prop}

For any finite extension field $E$ of $F$, we put $A(E)=\End_{F}(E)$, and then $A(E)$ is a central simple $F$-algebra.  
The field $E$ is canonically embedded in $A(E)$ as a maximal subfield.  
By \cite[1.2]{BK1}, there exists a unique $E$-pure hereditary $\ofra_{F}$-order $\Afra(E)=\{ x \in A(E) \mid x(\pfra_{E}^{i}) \subset \pfra_{E}^{i}, \, i \in \Z \}$ in $A(E)$, which is associated with the $\ofra_{F}$-chain $(\pfra_{E}^{i})_{i \in \Z}$.  
Then we have $v_{\Afra(E)}(\beta) = v_{E}(\beta)$ for $\beta \in E^{\times}$.  

For $\beta \in \bar{F}$, we put $n_{F}(\beta) = -v_{F[\beta]}(\beta) = -v_{\Afra(F[\beta])}(\beta)$ as in \cite[2.3.3]{S1}.  

Let $\Afra$ be a hereditary $\ofra_{F}$-order in $A$ with the radical $\Pfra$.  
For non-negative integer $i, j$ with $\lfloor j/2 \rfloor \leq i \leq j$, the map $1+x \mapsto x$ induces the group isomorphism
\[
	\U^{i+1}(\Afra)/\U^{j+1}(\Afra) \cong \Pfra^{i+1}/\Pfra^{j+1}.  
\]
If $i$ and $j$ are as above and $c \in \Pfra^{-j}$, we can define a character $\psi_{c}$ of $\U^{i+1}(\Afra)$ as
\[
	\psi_{c}(1+x) = \psi \circ \Trd_{A/F}(cx)
\]
for $1+x \in \U^{i+1}(\Afra)$.  
We have $\psi_{c} = \psi_{c'}$ if and only if $c-c' \in \Pfra^{-i}$.  

\subsection{Strata, defining sequences of simple strata}

\begin{defn}[{{\cite[\S 2.1, Remarque 4.1]{S3}}}]
\begin{enumerate}
\item A 4-tuple $[\Afra, n, r, \beta]$ is called a stratum in $A$ if $\Afra$ is a hereditary $\ofra_{F}$-order in $A$, $n$ and $r$ are non-negative integer with $n \geq r$, and $\beta \in \Pfra(\Afra)^{-n}$.  
\item A stratum $[\Afra, n, r, \beta]$ is called pure if the followings hold:  
	\begin{enumerate}
	\item $E=F[\beta]$ is a field.  
	\item $\Afra$ is $E$-pure.  
	\item $n>r$.  
	\item $v_{\Afra}(\beta) = -n$.  
	\end{enumerate}
\item A stratum $[\Afra, n, r, \beta]$ is called simple if one of the followings holds:  
	\begin{enumerate}
	\item $n=r=0$ and $\beta \in \ofra_{F}$.  
	\item $[\Afra, n, r, \beta]$ is pure, and $r < -k_{0}(\beta, \Afra)$, where $k_{0}(\beta, \Afra) \in \Z \cup \{-\infty\}$ is defined as in \cite[\S 2.1]{S3} such that $k_{0}(\beta, \Afra)=-\infty$ if and only if $\beta \in F$, and $v_{\Afra}(\beta) \leq k_{0}(\beta, \Afra)$ for $\beta \notin F$.  
	\end{enumerate}
\end{enumerate}
\end{defn}

\begin{rem}
In \cite[\S 2.1]{S3}, simple strata are assumed be pure.  
By adding strata satisfying (3)(a) to simple strata, we can regard simple types of depth zero as coming from simple strata.  
\end{rem}

\begin{defn}
Strata $[\Afra, n, r, \beta]$ and $[\Afra, n, r, \beta']$ in $A$ are called equivalent if $\beta-\beta' \in \Pfra(\Afra)^{-r}$.  
\end{defn}

\begin{thm}[{{\cite[Th\'eor\`eme 2.2]{S3}}}]
\label{existapp}
Let $[\Afra, n, r, \beta]$ be a pure stratum.  
Then there exists $\gamma \in A$ such that $[\Afra, n, r, \gamma]$ is simple and equivalent to $[\Afra, n, r, \beta]$.  
\end{thm}

For $\beta \in \bar{F}$, we put $k_{F}(\beta) = k_{0}(\beta, \Afra(F[\beta]))$ as in \cite[2.3.3]{S1}.  

\begin{prop}[{{\cite[Proposition 2.25]{S1}}}]
\label{relofk0}
Suppose $E=F[\beta]$ can be embedded in $A$.  
We fix an embedding $E \hookrightarrow A$.  
Let $\Afra$ be an $E$-pure hereditary $\ofra_{F}$-order in $A$.  
Then we have $k_{0}(\beta, \Afra) = e(\Afra|\ofra_{F})e(E/F)^{-1}k_{F}(\beta)$.  
\end{prop}

The following lemma is used later.  

\begin{lem}
\label{compofk0}
Let $E/F$ be a field extension in $A$, and let $\Afra$ be an $E$-pure hereditary $\ofra_{F}$-order in $A$.  
Then, we have $k_{0}(\gamma, \Afra) = e(\Afra|\ofra_{F})e(E/F)^{-1}k_{0}(\gamma, \Afra(E))$ for any $\gamma \in E$.  
\end{lem}

\begin{prf}
First, by Proposition \ref{relofk0} we have 
\[
	k_{0}(\gamma, \Afra(E)) = e(\Afra(E)|\ofra_{F})e(F[\gamma]/F)^{-1}k_{F}(\gamma).  
\]
On the other hand, we also have $e(\Afra(E)|\ofra_{F})=e(E/F)$ by definition of $\Afra(E)$.  
Then we obtain
\begin{eqnarray*}
	e(\Afra|\ofra_{F})e(E/F)^{-1}k_{0}(\gamma, \Afra(E)) & = & e(\Afra|\ofra_{F})e(E/F)^{-1}e(E/F)e(F[\gamma]/F)^{-1}k_{F}(\gamma) \\
	& = & e(\Afra|\ofra_{F})e(F[\gamma]/F)^{-1}k_{F}(\gamma) \\
	& = & k_{0}(\gamma, \Afra), 
\end{eqnarray*}
where the last equality also follows from Proposition \ref{relofk0}.  
\end{prf}

\begin{defn}
An element $\beta \in \bar{F}$ is called minimal if $\beta \in F$ or $k_{F}(\beta) = -v_{F}(\beta)$.  
\end{defn}

\begin{defn}
Let $[\Afra, n, r, \beta]$ be a simple stratum in $A$.  
A sequence $\left( [\Afra, n, r_i, \beta_i] \right)_{i=0}^{s}$ is called a defining sequence of $[\Afra, n, r, \beta]$ if
\begin{enumerate}
\item $\beta_0=\beta, r_0=r$, 
\item $r_{i+1}=-k_{0}(\beta_{i}, \Afra)$ for $i=0, 1, \ldots, s-1$, 
\item $[\Afra, n, r_{i+1}, \beta_{i+1}]$ is simple and equivalent to $[\Afra, n, r_{i+1}, \beta_{i}]$ for $i=0, 1, \ldots, s-1$, 
\item $\beta_{s}$ is minimal over $F$.  
\end{enumerate}
\end{defn}

By Theorem \ref{existapp}, for any simple stratum $[\Afra, n, r, \beta]$ there exists a defining sequence of $[\Afra, n, r, \beta]$, as in the case $A$ is split over $F$.  

\subsection{Simple characters}

Let $[\Afra, n, 0, \beta]$ be a simple stratum in $A$.  
Then we can define compact open subgroups $J(\beta, \Afra)$ and $H(\beta, \Afra)$ in $\U(\Afra)$ as in \cite[\S3]{S1}.  
The subgroup $H(\beta, \Afra)$ in $\U(\Afra)$ is also contained in $J(\beta, \Afra)$.  
For $i \in \Z_{\geq 0}$, we put $J^{i}(\beta, \Afra) = J(\beta, \Afra) \cap \U^{i}(\Afra)$ and $H^{i}(\beta, \Afra) = H(\beta, \Afra) \cap \U^{i}(\Afra)$.  

\begin{lem}[{{\cite[\S3.3]{S1}}}]
Let $[\Afra, n, 0, \beta]$ be a simple stratum in $A$.  
If $\beta$ is not minimal, let $([\Afra, n, r_{i}, \beta_{i}])_{i=0}^{s}$ be a defining sequence of $[\Afra, n, 0, \beta]$.  
\begin{enumerate}
\item $J^{i}(\beta, \Afra)$ is normalized by $\Kfra(\Afra) \cap B^{\times}$ for any $i \in \Z_{\geq 0}$.  
\item We have $J(\beta, \Afra) =\U(\Bfra) J^{1}(\beta, \Afra)$, where $\Bfra = \Afra \cap B$.  
\item If $\beta$ is minimal, we have 
\[
	J^{1}(\beta, \Afra) = \U^{1}(\Bfra)\U^{\lfloor (n+1)/2 \rfloor}(\Afra), \, H^{1}(\beta, \Afra) = \U^{1}(\Bfra) \U^{\lfloor n/2 \rfloor +1}(\Afra).
\]  
\item If $\beta$ is not minimal, we have 
\[
	J^{t}(\beta, \Afra) = J^{t}(\beta_{1}, \Afra), \, H^{t'+1}(\beta, \Afra) = H^{t'+1}(\beta_{1}, \Afra), 
\]
where $t = \lfloor (-k_{0}(\beta, \Afra)+1)/2 \rfloor$ and $t' = \lfloor -k_{0}(\beta, \Afra) /2 \rfloor$.  
Moreover, we also have 
\[
	J^{1}(\beta, \Afra) = \U^{1}(\Bfra) J^{t}(\beta_{1}, \Afra), \, H^{1}(\beta, \Afra) = \U^{1}(\Bfra) H^{t'+1}(\beta_{1}, \Afra).  
\]
\end{enumerate}
\end{lem}

\begin{defn}[{{\cite[Proposition 3.47]{S1}}}]
\label{defofsimpch}
Let $[\Afra, n, 0, \beta]$ be a simple stratum.  
We put $q= -k_{0}(\beta, \Afra)$.  
Let $0 \leq t < q$ and we put $t'= \max \{ t, \lfloor q/2 \rfloor \}$.  
If $\beta$ is not minimal over $F$, we fix a defining sequence $\left( [\Afra, n, r_i, \beta_i] \right)_{i=0}^{s}$ of $[\Afra, n, 0, \beta]$.  
The set of simple characters $\mathscr{C}(\beta, t, \Afra)$ consists of characters $\theta$ of $H^{t+1}(\beta, \Afra)$ satisfying the following conditions:  
\begin{enumerate}
\item $\Kfra(\Afra) \cap B^{\times}$ normalizes $\theta$.  
\item $\theta | _{H^{t+1}(\beta, \Afra) \cap \U(\Bfra)}$ factors through $\Nrd_{B/E}$.  
\item If $\beta$ is minimal over $F$, we have $\theta |_{H^{t+1}(\beta, \Afra) \cap \U^{\lfloor n/2 \rfloor +1}(\Afra)} = \psi_{\beta}$.  
\item If $\beta$ is not minimal over $F$, there exists $\theta' \in \mathscr{C}(\beta_{1}, t', \Afra)$ such that $\theta |_{H^{t'+1}(\beta, \Afra)} = \psi_{\beta-\beta_{1}} \theta'$.  
\end{enumerate}
\end{defn}

\begin{rem}
\label{existsimpch}
This definition is well-defined and independent of the choice of a defining sequence by \cite[D\'efinition 3.45, Proposition 3.47]{S1}.  
Moreover, for any simple stratum $[\Afra, n, 0, \beta]$ the set $\mathscr{C}(\beta, 0, \Afra)$ is nonempty by \cite[Corollaire 3.35, D\'efinition 3.45]{S1}.  
\end{rem}

We recall the properties of $\mathscr{C}(\beta, 0, \Afra)$ from \cite{S2}.  
For $\theta \in \mathscr{C}(\beta, 0, \Afra)$, there exists an irreducible $J^{1}(\beta, \Afra)$-representation $\eta_{\theta}$ containing $\theta$, unique up to isomorphism.  
We call $\eta_{\theta}$ the Heisenberg representation of $\theta$.  
We have $\dim \eta_{\theta} = \left( J^{1}(\beta, \Afra) : H^{1}(\beta, \Afra) \right) ^{1/2}$.  
Moreover, there exists an extension $\kappa$ of $\eta_{\theta}$ to $J(\beta, \Afra)$ such that $I_{G}(\kappa)=J^{1}B^{\times}J^{1}$.  
We call $\kappa$ a $\beta$-extension of $\eta_{\theta}$.  
If $\kappa$ is a $\beta$-extension of $\eta_{\theta}$, then any $\beta$-extension of $\eta_{\theta}$ is the form $\kappa \otimes (\chi \circ \Nrd_{B/E})$, where $\chi$ is trivial on $1+\pfra_{E}$ and $\chi \circ \Nrd_{B/E}$ is regarded a character of $J(\beta, \Afra)$ via the isomorphism $J(\beta, \Afra)/J^{1}(\beta, \Afra) \cong \U(\Bfra)/\U^{1}(\Bfra)$.  

\subsection{Maximal simple types}

We state the definition of maximal simple types.  
Recall that for a simple stratum $[\Afra, n, 0, \beta]$ we put $E=F[\beta]$, $B=\Cent_{A}(E)$ and $\Bfra = \Afra \cap B$.  
Since $B$ is a central simple $E$-algebra, there exist $m_{E} \in \Z$ and a division $E$-algebra $D_{E}$ such that $B \cong \M_{m_{E}}(D_{E})$.  

\begin{defn}[{{\cite[\S2.5, \S4.1]{S3}}}]
\label{defofsimpletype}
A pair $(J, \lambda)$ consisting a compact open subgroup $J$ in $G$ and an irreducible $J$-representation $\lambda$ is called a maximal simple type if there exists a simple stratum $[\Afra, n, 0, \beta]$ and irreducible $J$-representations $\kappa$ and $\sigma$ satisfying the following assertions:  
\begin{enumerate}
\item $\Bfra$ is a maximal hereditary $\ofra_{E}$-order in $A$, that is, $\Bfra \cong \M_{m_{E}}(\ofra_{D_{E}})$.  
\item $J = J(\beta, \Afra)$.  
\item $\kappa$ is a $\beta$-extension of $\eta_{\theta}$ for some $\theta \in \mathscr{C}(\beta, 0, \Afra)$.  
\item $\sigma$ is trivial on $J^{1}(\beta, \Afra)$, and when we regard $\sigma$ as a $\GL_{m_{E}}(k_{D_{E}})$-representation via the isomorphism
\[
	J(\beta, \Afra)/J^{1}(\beta, \Afra) \cong \U(\Bfra)/\U^{1}(\Bfra) \cong \GL_{m_{E}}(k_{D_{E}}), 
\]
$\sigma$ is a cuspidal representation of $\GL_{m_{E}}(k_{D_{E}})$.  
\item $\lambda \cong \kappa \otimes \sigma$.  
\end{enumerate}
\end{defn}

\begin{rem}
Let $(J, \lambda)$ be a maximal simple type associated with a simple stratum $[\Afra, 0, 0, \beta]$.  
Then we have $E=F[\beta]=F$, $B = \Cent_{A}(F) = A$ and $\Bfra = \Afra \cap A = \Afra$.  
Since $(J, \lambda)$ is a maximal simple type, $\Afra$ is a maximal hereditary $\ofra_{F}$-order in $A$.  
Moreover, we have $J(\beta, \Afra)=\U(\Afra)$ and $H^{1}(\beta, \Afra) = \U^{1}(\Afra)$.  
Let $\kappa$ and $\sigma$ be as in Definition \ref{defofsimpletype}.  
Since we have $\mathscr{C}(\beta, 0, \Afra) = \{ 1 \}$, there exists a character $\chi$ of $F^{\times}$ trivial on $1+\pfra_{F}$ such that  $\kappa = \chi \circ \Nrd_{A/F}$.  
Then $\kappa \otimes \sigma$ is trivial on $\U^{1}(\Afra)$ and cuspidal as a $\GL_{m}(k_{D})$-representation.  
Therefore $(J, \lambda) = (\U(\Afra), \kappa \otimes \sigma)$ is nothing but the maximal simple type of level 0, defined in \cite[\S2.5]{S3}.  
\end{rem}

\begin{thm}[{{\cite[Theorem 5.5(ii)]{GSZ} and \cite[Th\'eor\`eme 5.21]{S3}}}]
Let $\pi$ be an irreducible representation of $G$.  
Then $\pi$ is supercuspidal if and only if there exists a maximal simple type $(J, \lambda)$ such that $\lambda \subset \pi|_{J}$.  
\end{thm}

We recall the construction of irreducible supercuspidal representations of $G$ from maximal simple types.  
Let $(J, \lambda)$ be a maximal simple type associated with a simple stratum $[\Afra, n, 0, \beta]$.  
Let $\kappa$ and $\sigma$ be as in Definition \ref{defofsimpletype}.  
Since $\Bfra$ is maximal, we have $\Kfra(\Bfra) = \Kfra(\Afra) \cap B^{\times}$ by \cite[Lemme 1.6]{S2}, and then $\Kfra(\Bfra)$ normalizes $J(\beta, \Afra)$.  

We fix $g \in \Kfra(\Bfra)$ with $v_{\Bfra}(g) = 1$.  
Since $g$ normalizes $J(\beta, \Afra)$, we can consider the twist ${}^{g} \sigma$ of $\sigma$ by $g$.  
Let $l_{0}$ be the smallest positive integer such that ${}^{g^{l_{0}}}\sigma \cong \sigma$.  
Then $\tilde{J}(\lambda) = I_{G}(\lambda)$ is the subgroup in $G$ generated by $J$ and $g^{l_{0}}$.  

\begin{thm}[{{\cite[Th\'eor\`eme 5.2]{S3}, \cite[Corollary 5.22]{SS}}}]
\begin{enumerate}
\item For any maximal simple type $(J, \lambda)$, there exists an extension $\Lambda$ of $\lambda$ to $\tilde{J}(\lambda)$.  
\item Let $(\tilde{J}(\lambda), \Lambda)$ be as above.  
Then $\cInd_{\tilde{J}(\lambda)}^{G} \Lambda$ is irreducible and supercuspidal.  
\item For any irreducible supercuspidal representation $\pi$ of $G$, there exists an extension $(\tilde{J}(\lambda), \Lambda)$ of a maximal simple type $(J, \lambda)$ such that $\pi = \cInd_{\tilde{J}(\lambda)}^{G} \Lambda$.  
\end{enumerate}
\end{thm}

\subsection{Concrete presentation of open subgroups}

Above we defined open subgroups $H^{1}(\beta, \Afra), J(\beta, \Afra)$ and $\tilde{J}(\lambda)$.  
In this subsection, we define another subgroup $\hat{J}(\beta, \Afra)$ and obtain the concrete presentation of some groups, which is used later.  

\begin{defn}
\label{defofJhat}
Let $[\Afra, n, 0, \beta]$ be a simple stratum with $\Bfra$ is maximal.  
Then we put $\hat{J}(\beta, \Afra) = \Kfra(\Bfra)J(\beta, \Afra)$.  
\end{defn}

\begin{rem}
\label{rem_for_Jhat}
\begin{enumerate}
\item Since $\Kfra(\Bfra)$ normalizes $J(\beta, \Afra)$, the set $\hat{J}(\beta, \Afra)$ is also a subgroup in $G$.  
We have $\Kfra(\Bfra) \cap J(\beta, \Afra) = \U(\Bfra)$, and then 
\[
	\hat{J}(\beta, \Afra)/J(\beta, \Afra) \cong \Kfra(\Bfra)/\U(\Bfra) \cong \Z.  
\]
\item Let $(J, \lambda)$ be a maximal simple type associated with $[\Afra, n, 0, \beta]$.  
Then we have $\tilde{J}(\lambda) \subset \hat{J}(\beta, \Afra)$.  
The group $\hat{J}(\beta, \Afra)$ only depends on $[\Afra, n, 0, \beta]$, while $\tilde{J}(\lambda)$ also depends on $\lambda$ in general.  
\item In the condition in (2), furthermore suppose $G=\GL_{N}(F)$.  
In this case, the group $\Kfra(\Bfra)$ is generated by $\U(\Bfra)$ and $E^{\times}$ which are contained in $I_{G}(\lambda)=\tilde{J}(\lambda)$.  
Then we have $\Kfra(\Bfra) \subset \tilde{J}(\lambda)$ and $\hat{J}(\lambda) = \Kfra(\Bfra) J(\beta, \Afra) = \tilde{J}(\lambda) \subset \hat{J}(\beta, \Afra)$, which implies that $\tilde{J}(\lambda) = \hat{J}(\beta, \Afra)$ is independent from the choice of $\lambda$ for $G=\GL_{N}(F)$ case.  
\end{enumerate}
\end{rem}

We describe $H^{1}(\beta, \Afra), J(\beta, \Afra)$ and $\hat{J}(\beta, \Afra)$ concretely, using a defining sequence $([\Afra, n, r_{i}, \beta_{i}])_{i=0}^{s}$ of $[\Afra, n, 0, \beta]$.  
We put $B_{\beta_{i}} = \Cent_{A}(F[\beta_{i}])$ for $i=0, \ldots, s$.  

\begin{lem}
\label{presenofHJ}
Let $[\Afra, n, 0, \beta]$ be a maximal simple stratum of $A$ and $\left( [\Afra, n, r_i, \beta_i] \right) _{i=0}^{s}$ be a defining sequence of $[\Afra, n, 0, \beta]$.  
Then we have following concrete presentations of groups:  
\begin{enumerate}
\item $H^{1}(\beta, \Afra)=\left( B_{\beta_{0}}^{\times} \cap \U^{ \lfloor \frac{r_0}{2} \rfloor +1}(\Afra) \right) \cdots \left( B_{\beta_{s}} ^{\times} \cap \U^{ \lfloor \frac{r_{s}}{2} \rfloor +1}(\Afra) \right) \U^{\lfloor \frac{n}{2} \rfloor +1}(\Afra)$.  
\item $J(\beta, \Afra)=\U(\Bfra) \left( B_{\beta_{1}}^{\times} \cap \U^{ \lfloor \frac{r_1+1}{2} \rfloor}(\Afra) \right) \cdots \left( B_{\beta_{s}} ^{\times} \cap \U^{ \lfloor \frac{r_{s}+1}{2} \rfloor}(\Afra) \right) \U^{ \lfloor \frac{n+1}{2} \rfloor}(\Afra)$.  
\item $\hat{J}(\beta, \Afra)=\Kfra(\Bfra) \left( B_{\beta_{1}}^{\times} \cap \U^{ \lfloor \frac{r_1+1}{2} \rfloor}(\Afra) \right) \cdots \left( B_{\beta_{s}} ^{\times} \cap \U^{ \lfloor \frac{r_{s}+1}{2} \rfloor}(\Afra) \right) \U^{ \lfloor \frac{n+1}{2} \rfloor}(\Afra)$.  
\end{enumerate}
\end{lem}

\begin{prf}
We show (1) by induction on the length $s$ of a defining sequence.  
When $s=0$, that is, $\beta$ is minimal over $F$, then $H^{1}(\beta, \Afra) = \U^{1}(\Bfra) \U^{\lfloor n/2 \rfloor+1}(\Afra)$.  
Since we have $\U^{1}(\Bfra)=1+(B \cap \Pfra)=B \cap (1+\Pfra) = B \cap \U^{1}(\Afra)$ and $r_0=0$, the equality in (1) for minimal $\beta$ holds.  
Suppose $s > 0$, that is, $\beta$ is not minimal over $F$.  
Then $H^{1}(\beta, \Afra) = \U^{1}(\Bfra) H^{\lfloor r_1/2 \rfloor +1}(\beta_1, \Afra)$.  
By induction hypothesis, we have 
\[
	H^{1}(\beta_1, \Afra)=\U^1(\Bfra_{\beta_1}) \left( B_{\beta_{2}}^{\times} \cap \U^{ \lfloor \frac{r_2}{2} \rfloor +1}(\Afra) \right) \cdots \left( B_{\beta_{s}} ^{\times} \cap \U^{ \lfloor \frac{r_{s}}{2} \rfloor +1}(\Afra) \right) \U^{\lfloor \frac{n}{2} \rfloor +1}(\Afra).  
\]
Since $r_1 < r_2 < \ldots < r_s < n$, we have $\lfloor r_1/2 \rfloor + 1 \leq \lfloor r_2 /2 \rfloor + 1 \leq \ldots \leq \lfloor r_s /2 \rfloor + 1 \leq \lfloor n/2 \rfloor + 1$ and
\[
	 B_{\beta_{2}}^{\times} \cap \U^{ \lfloor \frac{r_2}{2} \rfloor +1}(\Afra) , \cdots , B_{\beta_{s}} ^{\times} \cap \U^{ \lfloor \frac{r_{s}}{2} \rfloor +1}(\Afra) , \U^{\lfloor \frac{n}{2} \rfloor +1}(\Afra) \subset \U^{\lfloor \frac{ r_{1}}{2} \rfloor +1} (\Afra).   
\]
Therefore we obtain
\begin{eqnarray*}
	H^{\lfloor \frac{r_{1}}{2} \rfloor +1}(\beta_1, \Afra) & = & \left( \U^1(\Bfra_{\beta_1}) \cap \U^{\lfloor \frac{r_{1}}{2} \rfloor +1} (\Afra) \right) \left( B_{\beta_{2}}^{\times} \cap \U^{ \lfloor \frac{r_2}{2} \rfloor +1}(\Afra) \right) \cdots \\
	 & & \hspace{80pt} \cdots \left( B_{\beta_{s}} ^{\times} \cap \U^{ \lfloor \frac{r_{s}}{2} \rfloor +1}(\Afra) \right) \U^{\lfloor \frac{n}{2} \rfloor +1}(\Afra) \\
	 & = & \left( B_{\beta_{1}}^{\times} \cap \U^{ \lfloor \frac{r_1}{2} \rfloor +1}(\Afra) \right) \cdots \left( B_{\beta_{s}} ^{\times} \cap \U^{ \lfloor \frac{r_{s}}{2} \rfloor +1}(\Afra) \right) \U^{\lfloor \frac{n}{2} \rfloor +1}(\Afra), 
\end{eqnarray*}
and the equality in (1) for non-minimal $\beta$ also holds.  

Similarly, we can show that 
\[
	J^{1}(\beta, \Afra)=\U^{1}(\Bfra) \left( B_{\beta_{1}}^{\times} \cap \U^{ \lfloor \frac{r_1+1}{2} \rfloor}(\Afra) \right) \cdots \left( B_{\beta_{s}} ^{\times} \cap \U^{ \lfloor \frac{r_{s}+1}{2} \rfloor}(\Afra) \right) \U^{ \lfloor \frac{n+1}{2} \rfloor}(\Afra).  
\]
Then (2) and (3) are deduced from the fact $J(\beta, \Afra)=\U(\Bfra) J^{1}(\beta, \Afra)$ and $\hat{J}(\beta, \Afra) = \Kfra(\Bfra)J(\beta, \Afra)$.  
\end{prf}

\section{Yu's construction of types for tame supercuspidal representations}
\label{Yu's_types}

In this section, we recall how to construct Yu's types from \cite{Yu}.  
Let $G$ be a connected reductive group over $F$.  

\subsection{Admissible sequences}

\begin{defn}
Let $(G^{i})=(G^{0}, \ldots, G^{d})$ be a sequence of subgroup schemes in $G$ over $F$.  
We call $(G^{i})$ is a tame twisted Levi sequence if $G^{0} \subset G^{1} \subset \cdots \subset G^{d}=G$ and there exists a tamely ramified extension $E$ of $F$ such that $G^{i} \times _{F} E$ is a split Levi subgroup in $G \times _{F} E$ for $i=0, \ldots, d$.  
\end{defn}

Let $\vec{G}=(G^{0}, \ldots, G^{d})$ be a tame twisted Levi sequence in $G$.  
Then there exist a maximal torus $T$ in $G^{0}$ over $F$ and a tamely ramified, finite Galois extension $E$ over $F$ such that $T \times_{F} E$ is split.  
For $i=0, \ldots, d$, we put $\Phi_{i}=\Phi(G^{i}, T; E) \cup \{0\}$.  
For $\alpha \in \Phi_{d} \setminus \{ 0\} = \Phi(G, T; E)$, we denote by $G_{\alpha}$ the root subgroup in $G_{E}$ defined by $\alpha$.  
Let $G_{\alpha}=T$ if $\alpha = 0$.  
Let $\gfra_{\alpha}$ be the Lie algebra of $G_{\alpha}$, which is a Lie subalgebra in $\gfra_{E}$, and let $\gfra_{\alpha}^{*}$ be its dual.  

Let $\vec{\rbf}=(\rbf_{0}, \rbf_{1}, \ldots, \rbf_{d}) \in \tilde{\R}^{d+1}$.  
Then we can define a map $f_{\vec{\rbf}} : \Phi_{d} \to \tilde{\R}$ by $f_{\vec{\rbf}}(\alpha) = \rbf_{i}$ if $i = \min \{ j \mid \alpha \in \Phi_{j} \}$.  

A sequence $\vec{\rbf}=(\rbf_{0}, \ldots, \rbf_{d}) \in \tilde{\R}^{d+1}$ is called an admissible sequence if and only if there exists $\nu \in \{ 0, \ldots, d \}$ such that 
\[
	0 \leq \rbf_{0} = \ldots = \rbf_{\nu}, \, \frac{1}{2}\rbf_{\nu} \leq \rbf_{\nu+1} \leq \ldots \leq \rbf_{d}.  
\]

Let $x$ be in the apartment $A(G, T, E) \subset \Bscr^{E}(G, E)$.  
Then we can determine the filtrations $\{ G_{\alpha}(E)_{x,r} \}_{r \in \tilde{\R}_{\geq 0}}$ on $G_{\alpha}(E)$, $\{ \gfra_{\alpha}(E)_{x,r} \}_{r \in \tilde{\R}}$ on $\gfra_{\alpha}(E)$, and $\{ \gfra_{\alpha}^{*}(E)_{x,r} \}_{r \in \tilde{\R}}$ on $\gfra_{\alpha}^{*}(E)$.  

We denote by $\vec{G}(E)_{x, \vec{\rbf}}$ the subgroup in $G(E)$ generated by $G_{\alpha}(E)_{x, f_{\vec{\rbf}}(\alpha)}$ ($\alpha \in \Phi_{d}$).  

By taking $x \in A(G, T, E) \cap \Bscr^{E}(G, E)$, we can determine a valuation on the root datum of $(G, T, E)$ in the sense of \cite{BT}.  
By restricting this valuation,  we can also define a valuation on the root datum of $(G^{i}, T, E)$.  
Then we can determine $x_{i} \in \Bscr^{E}(G^{i}, E)$ by the valuation, uniquely up to $X^{*}(G^{i}) \otimes \R$.  
When we take $x_{i}$ in such a way, we can determine an affine, $G^{i}(E)$-equivalent embedding  $j_{i} : \Bscr^{E}(G^{i}, E) \to \Bscr^{E}(G, E)$ such that $j_{i}(x_{i})=x$. This embedding depends on the choice of $x$.  
We identify $x_{i}$ with $x$ via $j_{i}$.  

Now $E|F$ is a tamely ramified Galois extension. To consider subgroups in $G(F)$, we also assume $x \in \Bscr^{E}(G, E)^{\Gal(E/F)}$, that is, $x \in \Bscr^{E}(G, F)$.  
Then we can determine the Moy--Prasad filtration \cite{MP1}, \cite{MP2} on $G^{i}(F), \gfra^{i}(F)$ and $(\gfra^{i})^{*}(F)$ by $x$.  
We put $\vec{G}(F)_{x,\vec{\rbf}} = \vec{G}(E)_{x, \vec{\rbf}} \cap G(F)$.  

\begin{prop}[{{\cite[2.10]{Yu}}}]
The group $\vec{G}(F)_{x, \vec{\rbf}}$ is independent of the choice of $T$.  
If $\vec{\rbf}$ is increasing with $\rbf_{0} > 0$, then we have 
\[
	\vec{G}(F)_{x, \vec{\rbf}} = G^{0}(F)_{x, \rbf_{0}} G^{1}(F)_{x, \rbf_{1}} \cdots G^{d}(F)_{x, \rbf_{d}}.  
\]
\end{prop}

\subsection{Generic elements, generic characters}
Let $r$ and $ r'  $ be two elements in $ \tilde{\R}_{>0}$ with $  r \leq  r' \leq 2r$.  
We put $G(F)_{x,r:r'}=G(F)_{x,r}/G(F)_{x,r'}$ and $\gfra(F)_{x,r:r'}=\gfra(F)_{x,r}/\gfra(F)_{x,r'}$.  
Then we have a group isomorphism $G(F)_{x,r:r'} \cong \gfra(F)_{x,r:r'}$, cf.\cite[Corollary 2.4]{Yu}. 

\begin{rem}
The above isomorphism is often called 'Moy-Prasad isomorphism'. Let us mention that in \cite[Theorem 4.3]{MRR} and \cite{Ma23d}, Moy-Prasad-like isomorphisms, called 'congruent isomorphisms', are proved for group schemes using dilatations of schemes. In \cite[Theorem
13.5.1]{KP}, the authors prove Moy-Prasad isomorphisms using congruent isomorphisms.  We refer to \cite{DMdS} for a survey on the theory of algebraic dilatations, including references to pioneering works such as \cite{Yu2}.
\end{rem}

Let $S$ be a subgroup of $G(F)$ between $G(F)_{x,r/2+}$ and $G(F)_{x,r+}$, and let $\sfra$ be the sublattice of $\Lie(G)$ between $\gfra(F)_{x,r/2+}$ and $\gfra(F)_{x,r+}$ such that $\sfra/\gfra(F)_{x,r+} \cong S/G(F)_{x,r+}$.  

\begin{defn}
A character $\Phi$ of $S/G(F)_{x,r+}$, with respect to $\psi$, is realized by $X^{*} \in \Lie^{*}(G)_{x,-r}$ if $\Phi$ is equal to 
\[
	\xymatrix{
		S/G(F)_{x,r+} \cong \sfra/\gfra(F)_{x,r+} \ar[r]^-{X^{*}} & F \ar[r]^-{\psi} & \C^{\times} \\
	}.  
\]
\end{defn}

Let $G'$ be a tame twisted Levi subgroup in $G$.   
The Lie algebra $\underline{\Lie}(G')$ and its dual $\underline{\Lie^{*}}(G')$ are equipped with canonical adjoint actions of the group scheme $G'$.  
Then the functor of fixed point $(\underline{\Lie^{*}}(G'))^{G'}$ is representable by a scheme (cf. Notation). We now consider $(\underline{\Lie^{*}}(G'))^{G'}(F)$ as a subset of $\underline{\Lie^{*}}(G')(F)=\Lie ^* (G')$.
  
To define $G$-generic characters of depth $r$ of $G'$, we define $G$-generic elements of depth $r$ in $(\underline{\Lie^{*}}(G'))^{G'}(F)$. 
For this, following \cite[§8]{Yu}, as corrected in \cite[Rem.4.1.3]{FKS} and \cite[Def. 2.1]{Fide21}, we consider the conditions \textbf{GE0}, \textbf{GE1} and \textbf{GE2}.  

We start with \textbf{GE0}.

\begin{defn} 
Let $X^{*} \in (\underline{\Lie^{*}}(G'))^{G'}(F)$.  
We say $X^{*}$
 satisfies \textbf{GE0} with depth $r$ if for some (equivalently, every by \cite[Lemma 2.3]{Fide21}) point $x \in \Bscr^{E} (G', F)$ we have $X^* \in \Lie ^* (G' ) _{x, -r}$.
\end{defn}

Let $E$ be a finite, tamely ramified extension of $F$ and $T$ be an $F$-torus in $G'$ such that $T \times_{F} E$ is maximal and split. 
Let $\alpha \in \Phi(G, T; \bar{F})$.   
Then the derivation $\dr \check{\alpha}$ is an $\bar{F}$-linear map from $\underline{\Lie} (\mathbb{G}_{m})(\bar{F}) \cong \bar{F}$ to $\Lie(T \times_{F} \bar{F})$.  
We obtain $H_{\alpha}=\dr \check{\alpha}(1)$ as an element in $\Lie(T \times_{F} \bar{F})$.  

Here, we recall the condition \textbf{GE1}.  
Let $X^{*} \in (\underline{\Lie^{*}}(G'))^{G'}(F)$.  
Then we can regard $X^{*} \in \Lie^{*}(G')$ as above.  
We put $X_{\bar{F}}^{*}=X^{*} \otimes_{F} 1 \in \Lie^{*}(G') \otimes_{F} \bar{F} = \Lie^{*}(G' \times_{F} \bar{F})$.  
Since $T \subset G'$, we have $H_{\alpha} \in \Lie(G' \times_{F}\bar{F}) = \Lie(G') \otimes_{F} \bar{F}$.  
Therefore we obtain $X_{\bar{F}}^{*} (H_{\alpha}) \in \bar{F}$.  

\begin{defn}
Let $X^{*} \in (\underline{\Lie^{*}}(G'))^{G'}(F)$.  
We say $X^{*}$ satisfies \textbf{GE1} with depth $r$ if $\ord \left( X_{\bar{F}}^{*} (H_{\alpha}) \right) = -r$ for all $\alpha \in \Phi(G, T; \bar{F}) \setminus \Phi(G', T; \bar{F})$.  
\end{defn}

We also have to consider the condition \textbf{GE2} defined in \cite[\S 8]{Yu}.  
However, in our case if \textbf{GE1} holds, then \textbf{GE2} automatically holds. We use the notion of torsion prime of a root datum as defined in \cite[\S 7]{Yu}

\begin{prop}[{{\cite[Lemma 8.1]{Yu}}}]
If the residual characteristic of $F$ is not a torsion prime for the root datum of $G$, then \textbf{GE1} implies \textbf{GE2}. 
\end{prop}

\begin{prop}[{{\cite[Corollary 1.13]{Ro}}}]
If a root datum is type A, then the set of torsion primes for the datum is empty.  
\end{prop}

From these propositions, we obtain the following corollary.  

\begin{cor}
\label{GE1toGE2}
If the root datum of $G$ is type A, then \textbf{GE1} implies \textbf{GE2}.    
\end{cor}

\begin{defn}
Let $X^{*} \in  (\underline{\Lie^{*}}(G'))^{G'}(F) \subset \Lie ^* (G')$.
The linear form $X^{*}$ is called $G$-generic of depth $r$ if and only if conditions \textbf{GE0}, \textbf{GE1} and \textbf{GE2} hold.  
\end{defn}

Eventually, we can define generic characters.  

\begin{defn}
Let $r \in \R_{>0}$. A character $\Phi$ of $G'(F)$ is called $G$-generic of  depth $r$ relative to $x$  if $\Phi |_{G'(F)_{x,r+}}$ is trivial, $\Phi |_{G'(F)_{x,r}}$ is non-trivial, and there exists a $G$-generic element of depth $r$ $X^{*} \in (\underline{\Lie^{*}}(G'))^{G'}(F)$  such that $\Phi$ is realized by $X^{*}$ when $\Phi$ is regarded as a character of $G'(F)_{x,r:r+}$.  
\end{defn}

\subsection{Yu data}

Let $d \in \Z_{\geq 0}$.  

A 5-tuple $\Psi = \left( x, (G^{i})_{i=0}^{d}, (\rbf_{i})_{i=0}^{d}, (\Phibf_{i})_{i=0}^{d}, \rho \right)$ is called a Yu datum if $\Psi$ satisfies the following conditions:  

\begin{itemize}
\item The sequence $(G^{i})_{i=0}^{d}$ is a tame twisted Levi sequence such that $Z(G^{i})/Z(G)$ is anisotropic for $i=0, \ldots, d$ and 
\[
	G^{0} \subsetneq G^{1} \subsetneq \cdots \subsetneq G^{d} = G.  
\]
\item We have $x \in \Bscr^{E}(G^{0}, F) \cap A(G, T, E)$, where $T$ is a maximal $F$-torus in $G$ which splits over some tamely ramified extension $E$ of $F$.  
\item For $i=0, \ldots, d$, the number $\rbf_{i} \in \R$ such that 
\[
	0=\rbf_{-1} < \rbf_{0} < \ldots < \rbf_{d-1} \leq \rbf_{d}.  
\]
\item For $i=0, \ldots, d-1$, the character $\Phibf_{i}$ of $G^{i}(F)$ is $G^{i+1}$-generic relative to $x$ of depth $\rbf_{i}$.  
If $\rbf_{d-1} \neq \rbf_{d}$, the character $\Phibf_{d}$ of $G^{d}(F)$ is of depth $\rbf_{d}$.  
If $\rbf_{d-1} = \rbf_{d}$, the character $\Phibf_{d}$ of $G^{d}(F)$ is trivial.  
\item The irreducible representation $\rho$ of $G^{0}(F)_{[x]}$ is trivial on $G^{0}(F)_{x,0+}$ but nontrivial on $G^{0}(F)_{x}$, and $\cInd_{G^{0}(F)_{[x]}}^{G^{0}(F)} \rho$ is irreducible and supercuspidal.  
\end{itemize}

\subsection{Yu's construction}
\label{Yuconst}
In this subsection, we construct Yu's type by using some data from a Yu datum.  
Let $\Psi = \left( x, (G^{i})_{i=0}^{d}, (\rbf_{i})_{i=0}^{d}, (\Phibf_{i})_{i=0}^{d}, \rho \right)$ be a Yu datum.  

First, Yu constructed subgroups in $G$, which some representations are defined over.  
\begin{defn}
\label{defofKi}
For $i=0, \ldots, d$, let $\mathbf{s}_{i} = \rbf_{i}/2$. Put $\mathbf{s}_{-1}=0$.
\begin{enumerate}
\item 
$\begin{array}{rcl}
	K_{+}^{i} & = & G^{0}(F)_{x, 0+} G^{1}(F)_{x,\mathbf{s}_{0}+} \cdots G^{i}(F)_{x, \mathbf{s}_{i-1}+} \\
	& = & (G^{0}, \ldots, G^{i})(F)_{x,(0+, \mathbf{s}_{0}+, \ldots, \mathbf{s}_{i-1}+)}.  
\end{array}$
\item 
$\begin{array}{rcl}
	{}^{\circ}K^{i} & = & G^{0}(F)_{x, 0} G^{1}(F)_{x, \mathbf{s}_0} \cdots G^{i}(F)_{x, \mathbf{s}_{i-1}} \\
	& = & G^{0}(F)_{x, 0} (G^{1}, \ldots, G^{i})(F)_{x, (\mathbf{s}_{0}, \ldots, \mathbf{s}_{i-1})}.  
\end{array}$
\item $K^{i}=G^{0}(F)_{[x]}G^{1}(F)_{x, \mathbf{s}_0} \cdots G^{i}(F)_{x, \mathbf{s}_{i-1}}=G^{0}(F)_{[x]} {}^{\circ}K^{i}$.  Recall that $G^{0}(F)_{[x]}$ denotes the stabiliser of $[x]$ in $G^0(F)$. 
\end{enumerate}
\end{defn}

\begin{prop}
For any $i=0, \ldots, d$, the groups $K_{+}^{i}$ and ${}^{\circ} K^{i}$ are compact, and $K^{i}$ is compact modulo center.  
\end{prop}

Yu also defined subgroups in $G(F)$, which ``fill the gap" between subgroups defined as above.  

\begin{defn}
For $i=1, \ldots, d$, 
\begin{enumerate}
\item $J^{i}=(G^{i-1}, G^{i})(F)_{x, (\rbf_{i-1}, \mathbf{s}_{i-1})}$, 
\item $J_{+}^{i}=(G^{i-1}, G^{i})(F)_{x, (\rbf_{i-1}, \mathbf{s}_{i-1}+)}$.  
\end{enumerate}
\end{defn}
Note that in general $J^i$ is different from $G^i(F)_{x,\mathbf{s}_{i-1}}$.
Then, we have $K^{i}J^{i+1}=K^{i+1}$ and $K_{+}^{i}J_{+}^{i+1}=K_{+}^{i+1}$ for $i=0, \ldots, d-1$.  

Next, Yu defined characters $\hat{\Phibf}_{i}$ of $K_{+}^{d}$.  
The Lie algebra $\gfra(F)$ of $G(F)$ is equipped with a canonical $G(F)$-action.  
In particular, $Z(G^{i})^{\circ}(F)$ acts on $\gfra(F)$ by restricting the $G(F)$-action.  
Then $Z(G^{i})^{\circ}(F)$-fixed part of $\gfra(F)$ is equal to the Lie algebra $\gfra^{i}(F)$ of $G^{i}(F)$.  
Moreover, we have a decomposition $\gfra(F) = \gfra^{i}(F) \oplus \mathfrak{n}^{i}(F)$ as a $Z(G^{i})^{\circ}(F)$-representation.  
This decomposition is well-behaved on the Moy--Prasad filtration:  we have $\gfra(F)_{x,s} = \gfra^{i}(F)_{x,s} \oplus \mathfrak{n}^{i}(F)_{x,s}$ for any $s \in \tilde{\R}$, where $\mathfrak{n}^{i}(F)_{x,s} \subset \mathfrak{n}^{i}(F)$.  
Let $\pi_{i}:\gfra(F) = \gfra^{i}(F) \oplus \mathfrak{n}^{i}(F) \to \gfra^{i}(F)$ be the projection.  
Then $\pi_{i}$ induces $\gfra(F)_{x,\mathbf{s}_{i}+:\rbf_{i}+} \to \gfra^{i}(F)_{x,\mathbf{s}_{i}+:\rbf_{i}+}$, and we obtain a group homomorphism 
\[
	\xymatrix{
		\tilde{\pi}_{i}:G(F)_{x,\mathbf{s}_{i}+} \ar[r] & G(F)_{x,\mathbf{s}_{i}+:\rbf_{i}+} \ar[r]^-{\pi_{i}} & G^{i}(F)_{x,\mathbf{s}_{i}+:\rbf_{i}+}.  
	}
\]
Here, Yu defined a character $\hat{\Phibf}_{i}$ of $K_{+}^{d}$ as
\begin{eqnarray*}
	\hat{\Phibf}_{i} |_{K_{+}^{d} \cap G^{i}(F)} & = & \Phibf_{i}, \\
	\hat{\Phibf}_{i} |_{K_{+}^{d} \cap G(F)_{x, \mathbf{s}_{i}+}} & = & \Phibf_{i} \circ \tilde{\pi}_{i}, 
\end{eqnarray*}
where $K^{d}_{+}$ is generated by $K_{+}^{d} \cap G^{i}(F)$ and $K_{+}^{d} \cap G(F)_{x, \mathbf{s}_{i}+}$ as we have
\begin{eqnarray*}
	K_{+}^{d} \cap G^{i}(F) & = & G^{0}(F)_{x, 0+} G^{1}(F)_{x,\mathbf{s}_{0}+} \cdots G^{i}(F)_{x, \mathbf{s}_{i-1}+} = K_{+}^{i}, \\
	K_{+}^{d} \cap G(F)_{x, \mathbf{s}_{i}+} & = & G^{i+1}(F)_{x, \mathbf{s}_{i}+} \cdots G^{d}(F)_{x, \mathbf{s}_{d-1}+}.  
\end{eqnarray*}

Using $\hat{\Phibf}_{i}$, Yu constructed a representation $\rho_{j}$ of $K_{j}$ for $j=0, \ldots, d$.  

\begin{lem}[{{\cite[\S 4]{Yu}}}]
\label{Phitilde}
Let $0 \leq i \leq d-1$.  
There is an irreducible representation $\tilde{\Phibf}_{i}$ of $K^{i} \ltimes J^{i+1}$ such that 
\begin{enumerate}
\item $\tilde{\Phibf}_{i} | _{1 \ltimes J_{+}^{i+1}}$ is $\hat{\Phibf}_{i} | _{J_{+}^{i+1}}$-isotypic, and
\item $\tilde{\Phibf}_{i} | _{K_{+}^{i} \ltimes 1}$ is $\mathbf{1}$-isotypic.  
\end{enumerate}
\end{lem}

\begin{lem}[{{\cite[\S 4]{Yu}}}]
Let $0 \leq i \leq d-1$.  
Let $\inf(\Phibf_{i})$ be the inflation of $\Phibf_{i} |_{K^{i}}$ to $K^{i} \ltimes J^{i+1}$, and let $\tilde{\Phibf}_{i}$ be as in Lemma \ref{Phitilde}.  
Then the $K^{i} \ltimes J^{i+1}$-representation $\inf(\Phibf_{i}) \otimes \tilde{\Phibf}_{i}$ factors through $K^{i} \ltimes J^{i+1} \to K^{i}J^{i+1}=K^{i+1}$.  
\end{lem}

\begin{defn}
We denote by $\Phibf'_{i}$ the $K^{i+1}$-representation $\inf(\Phibf_{i}) \otimes \tilde{\Phibf}_{i}$.  
\end{defn}

To obtain $\rho_j$ constructed by Yu, we use a little different way from Yu, by Hakim--Murnaghan.  

\begin{lem}[{{\cite[3.4]{HM}}}]
\begin{enumerate}
\item For $i=1, \ldots, d-1$, we have $K^{i} \cap J^{i+1} = G^{i}(F)_{x, \rbf_{i}} \subset J^{i}$.  
\item For $i=0, \ldots, d-1$, let $\mu$ be a $K^{i}$-representation which is trivial on $K^{i} \cap J^{i+1}$.  
Then we can obtain the inflation $\inf_{K^{i}}^{K^{i+1}} \mu$ of $\mu$ to $K^{i+1}$ via $K^{i+1}/J^{i+1} \cong K^{i}/(K^{i} \cap J^{i+1})$.  
The representation $\inf_{K^{i}}^{K^{i+1}} \mu$ is trivial on $J^{i+1}$, and also trivial on $K^{i+1} \cap J^{i+2}$ if $i<d-1$.  
\item If $i, \mu$ is as in (ii) and $i \leq j \leq d$, then we can also obtain the inflation $\inf_{K^{i}}^{K^{j}} \mu$ of $\mu$ to $K^{j}$ as $\inf_{K^{i}}^{K^{j}} \mu = \inf_{K^{j-1}}^{K^{j}} \circ \cdots \circ \inf_{K^{i}}^{K^{i+1}} \mu$.  
\end{enumerate}
\end{lem}

\begin{defn}[{{\cite[3.4]{HM}}}] \label{infinf}
For $0 \leq i < j \leq d$, we put $\kappa_{i}^{j} = \inf_{K^{i+1}}^{K^{j}} \Phibf'_{i}$.  
For $0 \leq j \leq d$, we put $\kappa_{-1}^{j} = \inf_{K^{0}}^{K^{j}} \rho$ and $\kappa_{j}^{j} = \Phibf_{j}|_{K^{j}}$.  
And also, for $-1 \leq i \leq d$ we put $\kappa_{i}=\kappa_{i}^{d}$.  
\end{defn}

\begin{prop}
Let $0 \leq j \leq d$.  
The representation $\rho_{j}$ constructed by Yu is isomorphic to
\[
	\kappa_{-1}^{j} \otimes \kappa_{0}^{j} \otimes \cdots \otimes \kappa_{j}^{j}.  
\]
In particular, 
\[
	\rho_{d} \cong \kappa_{-1} \otimes \kappa_{0} \otimes \cdots \otimes \kappa_{d}.  
\]
\end{prop}

\begin{proof} \begin{sloppypar}
The representation $\rho _j $ is constructed in \cite{Yu} on page 592. Yu inductively constructs two representations: $\rho _j$ and ${\rho _j} '$.

 Let us show by induction on $j$ that ${{\rho _j }'= \kappa ^j _{-1} \otimes \kappa ^j _0 \otimes   \cdots \otimes \kappa ^j _{j-1}}$ and
 ${\rho _j = \kappa ^j _{-1} \otimes \kappa ^j _0  \otimes \cdots \otimes \kappa ^j _{j}}$  
 If $j=0$, then by definition the representation $\rho _0 '$ constructed by Yu is $\rho$ and $\rho _0$ is $\rho _0 ' \otimes (\boldsymbol{\Phi}_0 | _{K^0})$. We have $\kappa _{-1}^0 = \rho $ and $\kappa _0 ^0 = {\boldsymbol{\Phi}_0}| _{K^0} $. So the case $j=0$ is complete.
 Assume that ${\rho _{j-1}' = \kappa ^{j-1} _{-1} \otimes \kappa ^{j-1} _0  \cdots \otimes \kappa ^{j-1} _{j-2}}$ and ${\rho _{j-1} = \kappa ^{j-1} _{-1} \otimes \kappa ^{j-1} _0 \otimes \cdots \otimes \kappa ^{j-1} _{j-1}}$. 
 Then by definition $\rho _j '$ is equal to ${\inf _{K^{j-1}}^{K^j} (\rho _{j-1} ' ) \otimes \boldsymbol{\Phi} _{j-1} '} $. By definition $\boldsymbol{\Phi} _{j-1} '$ is equal to $\kappa ^j _{j-1}$. 
 Moreover \begin{align*} {\inf}  _{K^{j-1}}^{K^j} (\rho _{j-1} ' )& ={\inf}  _{K^{j-1}}^{K^j} (\kappa ^{j-1} _{-1} \otimes \kappa ^{j-1} _0 \otimes \cdots \otimes \kappa ^{j-1} _{j-2})\\&= { \inf}_{K^{j-1}}^{K^j}( \kappa ^{j-1} _{-1} )\otimes {\inf }_{K^{j-1}}^{K^j} (\kappa ^{j-1} _0) \otimes \cdots \otimes {\inf} _{K^{j-1}}^{K^j} (\kappa ^{j-1} _{j-2})\\
 &=\kappa ^j _{-1} \otimes \kappa ^j _0 \otimes  \cdots \otimes \kappa ^j _{j-2}
 \end{align*}

 Consequently ${{\rho _j }'= \kappa ^j _{-1} \otimes \kappa ^j _0 \otimes \cdots \otimes \kappa _i ^j \otimes \cdots \otimes \kappa ^j _{j-1}}$. Finally, by Yu's definition, $\rho _j $ is equal to $\rho _j ' \otimes \boldsymbol{\Phi} _j| _{K^j} $, and thus $\rho _j =\kappa ^j _{-1} \otimes \kappa ^j _0\otimes \cdots \otimes \kappa ^j _{j}$, as required. \end{sloppypar}
\end{proof}

Therefore, we obtain $\rho_{j}$ constructed by Yu. Spice recently found a problem in the proof of the following statement. This was recently discussed and corrected by Fintzen \cite{Fi2}, so that we are allowed to do not worry about this problem.

\begin{thm}[{{\cite[15.1]{Yu}}}]
The compactly induced representation $\cInd_{K^{j}}^{G^{j}(F)} \rho_{j}$ of $G^{j}(F)$ is irreducible and supercuspidal.  
\end{thm}

For later use, we recall the following proposition on the dimension of representation space of $\kappa_{i}$.  

\begin{prop} \label{dimkap}
Let $0 \leq i < j \leq d$.  
Then the dimension of $\kappa_{i}^{j}$ is equal to the dimension of $\Phibf'_{i}$, which is also equal to $(J^{i+1}:J_{+}^{i+1})^{1/2}$.  
\end{prop}

\begin{proof}
 By definition $\kappa _i ^j$ is an inflation of $\boldsymbol{\Phi}_i$, consequently theses representations have equal dimensions. The representation $\boldsymbol{\Phi}_i '$ is the unique representation of $K^{i+1}$ whose inflation to $K^i \ltimes J^{i+1}$ is  $\tilde{\boldsymbol{\Phi}}_i$. Thus, the dimension of $\boldsymbol{\Phi}_i '$ is equal to $\tilde{\boldsymbol{\Phi}}_i$. The representation $\tilde{\boldsymbol{\Phi}}_i$ is constructed in \cite[11.5]{Yu} and is the pull back of the Weil representation of ${Sp(J^{i+1}/J^{i+1}_+) \ltimes (J^{i+1}/N_i)}$ where $N_i= \ker (\hat{\boldsymbol{\Phi}}_i)$. Thus, the dimension of $\tilde{\boldsymbol{\Phi}}_i$ is $[J^{i+1} : J^{i+1}_+]^{\frac{1}{2}}$.
\end{proof}

\section{Tame simple strata}
\label{TameSimple}
In this section, we consider the class of simple strata corresponding to some Yu datum.  We fix a uniformiser $\varpi _F$ of $F$.

\begin{defn}
\begin{enumerate}
\item A pure stratum $[\Afra, n, r, \beta]$ is called tame if $E=F[\beta]$ is a tamely ramified extension of $F$.  
\item A simple type $(J, \lambda)$ associated with a simple stratum $[\Afra, n, 0, \beta]$ is called tame if $[\Afra, n, 0, \beta]$ is tame.  
\end{enumerate}
\end{defn}

\begin{rem}
\label{esstame}
\begin{enumerate}
\item By \cite[(2.6.2)(4)(b), 2.7 Proposition]{BH}, the above definition is independent of the choice of simple strata.  
\item Essentially tame supercuspidal representations, defined in \cite[2.8]{BH}, are $G$-representations containing some tame simple types.  
\end{enumerate}
\end{rem}

As explained in \S \ref{Secherre}, any simple strata has a defining sequence.  
Actually, if a simple stratum $[\Afra, n, 0, \beta]$ is tame, then we can show the existence of a ``nice" defining sequence of $[\Afra, n, 0, \beta]$.  
To discuss such a defining sequence, we state several related propositions.

\begin{lem} \label{valval} Let $\Afra$ be an hereditary $\ofra _F$-order in $A\cong \M_{N}(F)$, and let $E$ be a field in $A$ such that $E^{\times } \subset \mathfrak{K} (\Afra )$. Let $\beta$ be an element in $E$, then 

\begin{equation}
v _{\Afra} (\beta) e ( E | F ) = e (\Afra | \ofra _F ) v _{E} (\beta) . \label{eqval}
\end{equation}
\end{lem}
\begin{prf}Let $\varpi_E$ denote a uniformiser  element in $E$. Since $E^{\times} \subset \mathfrak{K}(\Afra)$, the elements $\varpi_E , \varpi _F $ and $\beta$ are in $\mathfrak{K}(\Afra)$. Thus, Equality \cite[1.1.3]{BK1} is valid for these elements and we  use it in the following equalities.
 On the one hand \begin{equation}\beta ^{e(E| F)}\Afra = \varpi _E ^{v _ E (\beta ) e(E|F)} \Afra= \varpi _F ^{v _E (\beta)} \Afra . \label{val1}\end{equation}
On the other hand \begin{equation}\beta ^{e(E| F)}\Afra = \Pfra ^{v _{\Afra} (\beta) e (E | F ) } .\label{val2}\end{equation}
Moreover by definition of $e(\Afra | \mathfrak{o}_F)$ (see \cite[1.1.2]{BK1}), we have \begin{equation}\varpi _F ^{v _E (\beta)} \Afra = \Pfra ^{e (\Afra  | \ofra _F )v _E (\beta )}. \label{val3}\end{equation}
Equalities \ref{val1} , \ref{val2} and \ref{val3} show that \begin{equation}\Pfra ^{v _{\Afra} (\beta) e (E | F )} =\Pfra ^{e (\Afra  | \ofra _F )v _E (\beta )}.\end{equation}
Consequently, $v _{\Afra} (\beta) e (E | F )= e (\Afra  | \ofra _F )v _E (\beta ) $ and equality \ref{eqval} holds as required.

\end{prf}

 The following is analogous to \cite[2.2.3]{BK1}, the main differences are that the tameness condition is assumed and a maximality condition is removed.

\begin{prop}
\label{Mforcons}
Assume $A \cong \M_{N}(F)$ for some $N$.  
Let $[\Afra, n, r, \beta]$ be a tame simple stratum with $r>0$.  
Let $[\Bfra_{\beta}, r, r-1, b]$ be a simple stratum, where $\Bfra_{\beta}=\Afra \cap \Cent_{A}(\beta)$.    
Then we have $F[\beta + b]=F[\beta, b]$ and $[\Afra, n, r, \beta +b]$ is a pure stratum with  $k_0 (\beta +b , \Afra  ) = \left\{ \begin{array}{ll}
        -r=k_0(b,\mathfrak{B}_E)$ if $b \not \in E \\
        k_0(\beta , \mathfrak{A})$ if $b\in E.
    \end{array}  
\right.$ 
\end{prop}

Before proving the proposition let us set $E=F[\beta]$. Also set $B_E = \End_E(V)=\Cent_A (\beta)$. Let $\Pfra$ be the Jacobson radical of $\Afra$. Set $\Bfra _E = \Bfra _{\beta} = \Afra \cap B_E $ and $\Qfra _E = \Pfra \cap \Bfra _E $. Thus, $\Bfra _E$ is an $\mathfrak{o}_E$-hereditary order in $B_E$ and $\Qfra _E$ is the Jacobson radical of $\Bfra _E$. We now prove the proposition.

\begin{prf} Let $s:A \to B_E$ be the tame corestriction that is the identity when restricted to $B_E$, we recall that such maps exist by \cite[Remark (1.3.8) (ii)]{BK1}. 

 Let $\mathcal{L}=\{L_i\} _{i\in \Z} $ be an $\mathfrak{o} _F$-lattice chain such that
  \begin{center}${\Afra = \{x \in A | x(L_i) \subset L_i , i \in \Z \}.}$ \end{center}
 By definition \cite[2.2.1]{BK1}, \begin{center}${\mathfrak{K} (\Afra) = \{ x \in G | x(L_i) \in  \mathcal{L}, i \in \Z \}} $\end{center} and \begin{center} ${ \mathfrak{K} (\Bfra _E)=\{x \in G_E \mid x(L_i) \in \mathcal{L}, i \in \Z \}.} $ \end{center} Thus, \begin{equation}\label{kk}\mathfrak{K}(\Bfra _E) \subset \mathfrak{K} (\Afra).\end{equation}The
 stratum $[\Bfra _{E} , r , r-1 , b ]$ is simple, thus the definition of a simple stratum shows that  \begin{equation} \label{ek}E[b] ^{\times } \subset \mathfrak{K}(\mathfrak{B}_E).\end{equation} Put $E_1= E[b]= F[\beta , b]$.
Equations \ref{kk} and \ref{ek} imply that $E_1 ^{\times} \subset \mathfrak{K} (\Afra)$. This allows us to use the machinery of \cite[1.2]{BK1} for $\Afra$ and $E_1$.

  Set $B_{E_1}= \End_{E_1} (V)$ and $\Bfra _{E_1} = \Afra \cap \End _{E_1} (V)$. Proposition \cite[1.2.4]{BK1} implies that $\mathfrak{B}_{E_1}$ is an $\mathfrak{o}_{E_1}$-hereditary order in $B_{E_1}$. Let $A(E_1)$ be the algebra $\End_F(E_1)$ and let $\Afra (E_1)$ be the $\mathfrak{o} _F$-hereditary order in $ A(E_1)$ defined by $\mathfrak{A}(E_1) = \{x \in \End _F (E_1) \mid x (\mathfrak{p} _{E_1} ^i) \subset \mathfrak{p} _{E_1} ^i , i \in \Z \}. $ Let $W$ be the $F$-span of an $\mathfrak{o}_{E_1}$-splitting basis of the $\mathfrak{o} _{E_1}$-lattice chain $\mathcal{L}$. Proposition \cite[1.2.8]{BK1} shows that the $(W,E_1)$-decomposition of $A$ restricts to an isomorphism $\Afra \simeq \Afra (E_1) \otimes _{\mathfrak{o} _{E_1} } \mathfrak{B} $ of $(\Afra (E_1) , \Bfra _{E_1})$-bimodules. Similarly, we have  a decomposition ${\Bfra _E \simeq \Bfra _E (E_1) \otimes _{\mathfrak{o} _{E_1}} \Bfra _{_{E_1}}}$.
    Set $B_E (E_1)= \End_E (E_1)$ and $\Bfra _E (E_1) = B_E (E_1) \cap \Afra (E_1)$. Also set $n(E_1)= \dfrac{n}{e(\Bfra _{E_1} | \ofra _{E_1})}$ and $r(E_1)= \dfrac{r}{e(\Bfra _{E_1} | \ofra _{E_1})}$. Let us prove the following equalities:

   \begin{equation} \label{111} v _{\Afra (E_1)} (\beta) =-n (E_1)     , \end{equation}

   \begin{equation}
    v _{\Bfra _E (E_1)}(b) =-r(E_1).  \label{eqv2}
    \end{equation}
    Let us prove that the equation \ref{111} holds. By definition of $E_1$,  the element $\beta$ is inside $E_1$ and thus $v _{\Afra (E_1)} (\beta) = v _{E_1} (\beta). $ Thus, Lemma \ref{valval} shows that \begin{equation} \label{v3}v _{\Afra} (\beta) e(E_1 | F) = e( \Afra | \ofra _F) v _{\Afra (E_1)} (\beta). \end{equation} Proposition \cite[1.2.4]{BK1} gives us the equality  \begin{equation}\label{44} e(\Bfra _{E_1} | \ofra _{E_1} ) =\dfrac{e(\Afra | \ofra _F)}{e(E_1 | F) }.\end{equation}
Because $[\Afra , n , r , \beta ]$ is a simple stratum, $n$ is equal to $- v _{\Afra} (\beta)$, consequently using equations \ref{v3} and \ref{44}, the following sequence of equality holds.

\begin{align*}
v _{\Afra (E_1)} (\beta) = \frac{ v _{\Afra} (\beta) e ( E_1 | F )}{e(\Afra | \mathfrak{o} _F )}= \frac{v _{\Afra} (\beta)}{e(\Bfra _{E_1} | \mathfrak{o}_{E_1})}=  \frac{-n}{e(\Bfra _{E_1} | \mathfrak{o}_{E_1})}=-n(E_1)
\end{align*}
This concludes the proof of the equality \ref{111} and the equality \ref{eqv2} is easily proven in the same way.
Proposition \cite[1.4.13]{BK1}  gives

\begin{center}
$\left\{ \begin{array}{ll}
        k_0(\beta,\Afra (E_1))=\dfrac{k_0 ( \beta , \Afra )}{e(\Bfra _{E_1} | \ofra _{E_1})} \\
        k_0(b , \Bfra _E (E_1) )= \dfrac{k_0 (b , \Bfra _E) }{e (\Bfra _{E_1} | \ofra _{E_1})}
    \end{array}  
\right.$
.
\end{center}
Consequently $[\Afra (E_1) , n (E_1) , r(E_1) , \beta ]$ and $[\Bfra _E (E_1) , r(E_1) , r(E_1) -1 , b ]$ are simple strata and satisfy the hypothesis of the proposition \cite[2.2.3]{BK1}. Consequently $[\Afra (E_1) , n , r-1, \beta + b ]$ is simple and the field $F[\beta + b]$ is equal to the field $F[\beta , b ]$.  Moreover, \cite[2.2.3]{BK1} implies that \begin{center}

$k_0 (\beta +b , \Afra (E_1) ) = \left\{ \begin{array}{ll}
        -r(E_1)=k_0(b,\mathfrak{B}_E(E_1))$ if $b \not \in E \\
        k_0(\beta , \mathfrak{A}(E_1))$ if $b\in E.
    \end{array}  
\right.$

\end{center}
The valuation $v_{\Afra (E_1)} (\beta +b)$ is equal to $-n(E_1)$ and the same argument as before shows that $v _{\Afra} (\beta +b ) =-n $. The proposition \cite[1.4.13]{BK1} shows that $k_0 (\beta +b , \Afra) = k_0 (\beta +b , \Afra (E_1) ) e(\Bfra _{E_1} | \ofra _{E_1} ) $.
Thus

\begin{center}
$k_0 (\beta +b , \Afra  ) = \left\{ \begin{array}{ll}
        -r=k_0(b,\mathfrak{B}_E)$ if $b \not \in E \\
        k_0(\beta , \mathfrak{A})$ if $b\in E.
    \end{array}  
\right.$
\end{center}
This completes the proof.

\end{prf}

Compare Proposition \ref{Mforcons} with \cite[Proposition 4.2]{BK94}.

\begin{prop}[{{\cite[3.1 Corollary]{BH}}}]
\label{BHforapp}
Let $E$ be a finite, tamely ramified extension of $F$ and let $\beta \in E$ such that $E=F[\beta]$.  
Let $[\Afra(E), n, r, \beta]$ be a pure stratum in $A(E)$ with $r=-k_{F}(\beta) < n$.  
Then there exists $\gamma \in E$ such that $[\Afra(E), n, r, \gamma]$ is simple and equivalent to $[\Afra(E), n, r, \beta]$.  
Moreover, if $\iota : E \hookrightarrow A$ is an $F$-algebra inclusion and $[\Afra, n', r', \iota(\beta)]$ is a pure stratum of $A$ with $r'=-k_{0}(\iota(\beta), \Afra)$, then $[\Afra, n', r', \iota(\gamma)]$ is simple and equivalent to $[\Afra, n', r', \beta]$.  
\end{prop}

\begin{prop}
\label{Mforapp}
Assume $A \cong \M_{N}(F)$ for some $N$.  
Let $[\Afra, n, r, \beta]$ be a tame, pure stratum of $A$ with $r=-k_{0} (\beta, \Afra)$.  
Let $\gamma \in E=F[\beta]$ such that $[\Afra, n, r, \gamma]$ is simple and equivalent to $[\Afra, n, r, \beta]$.  
Then $[\Bfra_{\gamma}, r, r-1, \beta-\gamma]$ is simple.  
\end{prop}

\begin{prf}
Using an argument similar to the one in the proof of Proposition \ref{Mforcons}, it
 is enough to prove the proposition in the case where $F[\beta]$ is a maximal subfield of the algebra $A=\End_F(V)$. 
So let $[\Afra , n , r , \beta ]$ be a tame pure stratum such that $F[\beta]$ is a maximal subfield of $A$ and $k_0 ( \beta , \Afra) = -r$. Let $\gamma$ be in $F[\beta]$ such that $[\Afra , n , r , \gamma ] $ is simple. The stratum $[\Bfra _{\gamma} , r , r-1 , \beta - \gamma ] $ is pure in the algebra $\End_{F[\gamma]} (V)$  because it is equivalent to a simple one by \cite[2.4.1]{BK1}. Moreover ${[\Bfra _{\gamma} , r , r-1 , \beta - \gamma ] }$
 is tame pure, so Proposition \ref{BHforapp} shows that there exists a simple stratum $[\Bfra _{\gamma} , r , r-1 , \alpha  ] $
  equivalent to $[\Bfra _{\gamma} , r , r-1 , \beta - \gamma ] $, such that $F[\gamma][\alpha] \subset F[\gamma][\beta - \gamma ].$ By Proposition \ref{Mforcons},
  $[\Afra , n , r-1 , \gamma + \alpha ]$ is simple and $F[\gamma + \alpha ]$ is equal to the field $F[\gamma , \alpha ]$. Set $\mathfrak{Q} _{\gamma} = \mathrm{rad} ( \Bfra _{\gamma} ) = \Bfra _{\gamma} \cap \mathfrak{P}$. The equivalence $[\Bfra _{\gamma} , r ,r-1 , \alpha ] \sim [\Bfra _{\gamma} , r , r-1 , \beta - \gamma ]$ shows that $ \alpha \equiv \beta - \gamma$ $(\mod~ \mathfrak{Q} _{\gamma} ^{-(r-1)})$. This implies $\gamma + \alpha \equiv \beta$ $ (\mod~ \Pfra ^{-(r-1)}). $ We deduce that $[\Afra , n , r -1 , \gamma + \alpha ] $ and $[\Afra , n , r-1 , \beta ]$ are two simple strata equivalent. Indeed, the first is simple by construction and the second is simple by hypothesis because $k_0 (\beta , \Afra )=-r$. The definitions shows that $F[\gamma + \alpha ] \subset F[\beta]$, and \cite[2.4.1]{BK1} shows that $[F[\gamma + \alpha ]:F] =[F[\beta]:F]$. Thus, $F[\gamma + \alpha ] = F[\beta]$. The trivial inclusions $F[\gamma + \alpha] \subset F[\gamma , \alpha] \subset F[\beta]$ then shows that $ F[\gamma + \alpha] = F [\gamma , \alpha ] = F [\beta ]$.
  We have thus obtained that the following three assertions hold:

 - The stratum $[\Bfra _{\gamma} , r , r-1 , \alpha ]$ is a simple stratum in $\End_{F[\gamma]}(V)$.

 - The field $F[\gamma][\alpha]$ is a maximal subfield of the $F[\gamma]$-algebra $\End_{F[\gamma]}(V)$.

 -$[\Bfra _{\gamma} , r , r-1 , \alpha ] \sim[\Bfra _{\gamma} , r ,r-1, \beta - \gamma ]$\\ 
Consequently, by \cite[2.2.2]{BK1}, $[\Bfra _{\gamma} , r ,r-1 , \beta - \gamma ]$ is simple,   as required.
\end{prf}

By these propositions, we obtain the following proposition needed in our case.  

\begin{prop}
\label{appfortame}
Let $[\Afra, n, r, \beta]$ be a tame pure stratum of $A$ with $r=-k_{0}(\beta, \Afra)$.  
Then there exists an element $\gamma$ in $F[\beta]$ satisfying the following conditions:  
\begin{enumerate}
\item The stratum $[\Afra, n, r, \gamma]$ is simple and equivalent to $[\Afra, n, r, \beta]$.   
\item $\beta-\gamma$ is minimal over $F[\gamma]$.  
\item The equality $v_{\Afra}(\beta-\gamma)=k_{0}(\beta, \Afra)$ holds.  
\end{enumerate}
\end{prop}

\begin{prf}
By Proposition \ref{BHforapp}, there exists $\gamma$ satisfying (1).  
We show that $\gamma$ also satisfy (2) and (3).  
We apply Proposition \ref{Mforapp} to the case $A=A(E)$.  
Then the stratum $[\Bfra', -k_{0} \left( \beta, \Afra(E) \right), -k_{0} \left( \beta, \Afra(E) \right) -1, \beta-\gamma]$ is simple, where $\Bfra'=\Cent_{A(E)}(\gamma) \cap \Afra(E)$.  
Since this stratum is simple, $\beta-\gamma$ is minimal over $F[\gamma]$ and (2) is satisfied.  
To obtain (3), we calculate $v_{\Afra}(\beta-\gamma)$ and $k_{0}(\beta, \Afra)$.   
First, we have
\[
	v_{E}(\beta-\gamma) = v_{\ofra_{E}}(\beta-\gamma) = v_{\Bfra'}(\beta-\gamma) = -\left( -k_{0} \left( \beta, \Afra(E) \right) \right) = k_{0} \left( \beta, \Afra(E) \right) .  
\]
Then we obtain 
\[
	v_{\Afra}(\beta-\gamma) = \frac{e(\Afra|\ofra_{F})}{e(E/F)}v_{E}(\beta-\gamma) = \frac{e(\Afra|\ofra_{F})}{e(E/F)} k_{0} \left( \beta, \Afra(E) \right) = k_{0}(\beta, \Afra)
\]
and (3) is also satisfied.  
\end{prf}

\section{Tame twisted Levi subgroups of $G$}
\label{TTLS}

First, we show some subgroups in $G$ are tame twisted Levi subgroups.  

Let $E/F$ be a field extension, and let $W$ be a right $D \otimes_{F} E$-module such that $\dim_{E}(W) < \infty$.  
Then we can define an $E$-scheme $\underline{\Aut}_{D \otimes_{F} E}(W)$ as
\[
	\underline{\Aut}_{D \otimes_{F} E}(W)(C)= \Aut_{D \otimes_{F} C}(W \otimes_{E} C)
\]
for an $E$-algebra $C$.  

Here, let $V$ be a right $D$-module and let $E/F$ be a field extension in $\End_{D}(V)$.  
Then $V$ can be equipped with the canonical right $D \otimes _{F} E$-module structure.  

Let $E'/E/F$ be a field extension in $\End_{D}(V)$ such that $E'$ is a tamely ramified extension of $F$.  
We put $G=\underline{\Aut}_{D}(V)$, $H=\Res_{E/F} \underline{\Aut}_{D \otimes _{F} E}(V)$ and $H'= \Res_{E'/F} \underline{\Aut}_{D \otimes_{F} E'}(V)$, where the functor $\Res$ is the Weil restriction.  
Then $H'$ is a closed subscheme in $H$ and $H$ is a closed subscheme in $G$.  

To show that $(H', H, G)$ is a tame twisted Levi sequence, we fix a maximal torus in $G$.  
We take a maximal subfield $L$ in $\End_{D}(V)$ such that $L$ is a tamely ramified extension of $E'$.  
We put $T=\Res_{L/F} \underline{\Aut}_{D \otimes _{F} L}(V)$.  

For a finite field extension $E_{0}/F$, we put $X_{E_{0}}=\Hom_{F}(E_{0}, \bar{F})$.  
Let $\tilde{L}$ be the Galois closure of $L/F$ in $\bar{F}$ and let $C$ be an $\tilde{L}$-algebra.  
Then we have an $E$-algebra isomorphism
\begin{eqnarray*}
	E \otimes_{F} C & \cong & \prod_{\sigma \in X_{E}} C_{\sigma} \\
	l \otimes a & \mapsto & \left( \sigma(l)a \right)_{\sigma}, 
\end{eqnarray*}
where $C_{\sigma}=C$ is equipped with an $E$-algebra structure by the inclusion $l \mapsto \sigma(l)$ from $E$ to $C$.  
The canonical inclusion $E \otimes_{F} C \hookrightarrow E' \otimes_{F} C$ induced from $E \subset E'$ is the map $(l_{\sigma})_{\sigma \in X_{E}} \mapsto (l_{\sigma'|_{E}})_{\sigma' \in X_{E'}}$.  
We also have $V \otimes_{F} C \cong V \otimes_{E} E \otimes_{F} C \cong V \otimes_{E} (\prod_{\sigma \in X_{E}}C_{\sigma}) \cong \bigoplus_{\sigma \in X_{E}} V \otimes_{E} C_{\sigma}$.  
We put $V_{\sigma}:= V \otimes_{E} C_{\sigma}$ for any $\sigma \in X_{E}$.  

We need the following lemma

\begin{lem}
\label{lem_for_prop_for_Levidecom}
Let $t_{0}$ be a positive integer, and let $R_{t}$ be a (non-commutative) ring for $t=1, \ldots, t_{0}$.  
We put $R:=\prod_{t=1}^{t_{0}} R_{t}$, and we regard $R_{t}$ as an $R$-submodule in $R \cong \bigoplus_{t=1}^{t_{0}} R_{t}$.  
Let $V$ be a right $R$-module.  
We put $V_{t}:=V \cdot 1_{t}$, where $1_{t} \in R_{t} \subset R$ is the identity element in $R_{t}$.  
Then we have $V = \bigoplus_{t=1}^{t_{0}} V_{t}$ and an isomorphism 
\begin{eqnarray*}
	\Aut_{R}(V) & \cong & \prod_{t=1}^{t_{0}} \Aut_{R_{t}}(V_{t}) \\
	f & \mapsto & (f|_{V_{t}})_{t}.  
\end{eqnarray*}
The inverse of this isomorphism is the map $(f_{t})_{t} \mapsto \bigoplus_{t=1}^{t_{0}} f_{t}$.  
\end{lem}

\begin{prop}
\label{prop_for_Levidecom}
Let $V$, $L/E'/E/F$, $H$ and $H'$ be as above.  
Moreover, let $C$ be an extension field of $\tilde{L}$.  
\begin{enumerate}
\item For $\sigma \in X_{E}$, we have $V_{\sigma} = \bigoplus_{\sigma ' \in X_{E'}, \sigma'|_{E}=\sigma} V_{\sigma'}$.  
\item We have a $C$-group scheme isomorphism $H \times_{F} C \cong \prod _{\sigma \in X_{E}} \underline{\Aut}_{D \otimes_{F} C_{\sigma} } ( V_{\sigma} )$.  
\item We have a commutative diagram of $C$-group schemes:  
\[
	\xymatrix{
		H' \times_{F} C \ar[r]^-{\cong} \ar[d] & \prod_{\sigma' \in X_{E'}} \underline{\Aut}_{D \otimes_{F} C_{\sigma'} } \left( V_{\sigma'}  \right) \ar[d] \\
		H \times_{F} C \ar[r]^-{\cong} & \prod_{\sigma \in X_{E}} \underline{\Aut}_{D \otimes_{F} C_{\sigma} } \left( V_{\sigma}  \right),  \\
	}
\]
\end{enumerate}
\end{prop}

\begin{prf}
Since we can regard $V$ as a right $D \otimes_{F} E$-module, the group $V \otimes_{F} C$ is equipped with a canonical right $D \otimes_{F} E \otimes _{F} C$-module structure.  
Moreover, by the canonical isomorphism $V \otimes_{F} C \cong V \otimes_{E} E \otimes_{F} C$ we can also equip $V \otimes_{E} (E \otimes_{F} C)$ with a right $D \otimes_{F} (E \otimes_{F} C)$-module structure.  
This action on $V \otimes_{E} E \otimes_{F} C$ is as follows:  
For $v \in V$, $z \in D$ and $b_{1}, b_{2} \in E \otimes_{F} C$, we have $(v \otimes_{E} b_{1}) \cdot (z \otimes_{F} b_{2}) = (vz) \otimes_{E} b_{1} b_{2}$.  
Let $\sigma \in X_{E}$.  
Let $1_{\sigma} \in C_{\sigma} \subset \bigoplus_{\tau \in X_{E}} C_{\tau}$ be the identity element in $C_{\sigma}$, and we regard  $1_{\sigma}$ as an element in $E \otimes_{F} C$ by the $E \otimes_{F} C$-module isomorphism $E \otimes_{F} C \cong \prod_{\tau \in X_{E}} C_{\tau} \cong \bigoplus_{\tau \in X_{E}} C_{\tau}$.  
Similarly, for $\sigma' \in X_{E'}$ we define $1_{\sigma'} \in C_{\sigma'} \subset E' \otimes_{F} C$.  

We have $(V \otimes_{F} C) \cdot (1 \otimes_{F} 1_{\sigma}) = (V \otimes_{E} (E \otimes_{F} C)) \cdot (1 \otimes_{F} 1_{\sigma}) = (V \cdot 1) \otimes_{E} ((\prod_{\tau \in X_{E}} C_{\tau}) \cdot 1_{\sigma}) = V \otimes_{E} C_{\sigma} = V_{\sigma}$.  
By the same argument, we also obtain $V_{\sigma'} = (V \otimes_{F} C) \cdot (1 \otimes_{F} 1_{\sigma'})$ for $\sigma' \in X_{E'}$.  
Since $1_{\sigma} = \sum_{\sigma' \in X_{E}, \sigma'|_{E} = \sigma} 1_{\sigma'}$, we obtain
\begin{eqnarray*}
	V_{\sigma} & = & (V \otimes_{F} C) \cdot (1 \otimes_{F} 1_{\sigma}) = \left(V \otimes_{E'} ( E' \otimes_{F} C) \right) \cdot \left(1 \otimes_{F}\left( \sum_{\sigma' \in X_{E}, \sigma'|_{E} = \sigma} 1_{\sigma'} \right) \right) \\
	& = & V \otimes_{E'} \left( \bigoplus_{\sigma' \in X_{E}, \sigma'|_{E} = \sigma} C_{\sigma'} \right)  = \bigoplus_{\sigma' \in X_{E}, \sigma'|_{E} = \sigma} (V \otimes_{E'} C_{\sigma'}) = \bigoplus_{\sigma' \in X_{E}, \sigma'|_{E} = \sigma} V_{\sigma'}, 
\end{eqnarray*}
which is the equality in (1).  

To show (2), let $R$ be a $C$-algebra.  
Then we have 
\begin{eqnarray*}
	H \times_{F} C (R) & = & \Res_{E/F} \underline{\Aut}_{D \otimes_{F} E} (V) (R)  = \underline{\Aut}_{D \otimes_{F} E} (V) (E \otimes_{F} R) \\
	& = & \Aut_{D \otimes_{F} E \otimes_{E} E \otimes_{F} R}(V \otimes_{E} E \otimes_{F} R) \\
	& \cong & \Aut_{D \otimes_{F} (E \otimes_{F} C) \otimes_{C} R} (V \otimes_{E} (E \otimes_{F} C ) \otimes_{C} R).  
\end{eqnarray*}

Here, we have a ring isomorphism $D \otimes_{F} (E \otimes_{F} C) \otimes_{C} R \cong \prod_{\sigma \in X_{E}} D \otimes_{F} C_{\sigma} \otimes_{C} R$, and the identity element in $D \otimes_{F} C_{\sigma} \otimes_{C} R$ is $1 \otimes_{F} 1_{\sigma} \otimes_{C} 1$.  
Moreover, $(V \otimes_{E} (E \otimes_{F} C) \otimes_{C} R) \cdot (1 \otimes_{F} 1_{\sigma} \otimes_{C} 1) = V \otimes_{E} C_{\sigma} \otimes_{C} R = V_{\sigma} \otimes R$.  
Then, by Lemma \ref{lem_for_prop_for_Levidecom}, we obtain $\Aut_{D \otimes_{F} (E \otimes_{F} C) \otimes_{C} R} (V \otimes_{E} (E \otimes_{F} C ) \otimes_{C} R) \cong \prod_{\sigma \in X_{E}} \Aut_{D \otimes_{F} C_{\sigma} \otimes_{C} R} (V_{\sigma} \otimes_{C} R) =\left( \prod_{\sigma \in X_{E}} \underline{\Aut}_{D \otimes_{F} C_{\sigma}}(V_{\sigma}) \right)(R)$, which completes the proof of (2).  

(3) is the result from (2) and the fact that there exists a canonical inclusion $H' \subset H$.  
\end{prf}

\begin{rem}
\label{rem_for_Levidecom}
We describe the right vertical morphism in (3).  
First, the isomorphism
\begin{eqnarray*}
H \times_{F} C(R) & = & \Aut_{D \otimes _{F} (\prod_{\sigma \in X_{E}} C_{\sigma}) \otimes_{C} R} \left( \bigoplus_{\sigma \in X_{E}} (V_{\sigma} \otimes_{C} R) \right) \\
 & \cong & \prod_{\sigma \in X_{E}} \Aut_{D \otimes_{F} C_{\sigma} \otimes_{C} R} (V_{\sigma} \otimes_{C} R) = \left( \prod_{\sigma \in X_{E}} \underline{\Aut}_{D \otimes_{F} C_{\sigma}}(V_{\sigma}) \right) (R)
\end{eqnarray*}
is given as $f \mapsto (f|_{V_{\sigma} \otimes_{C} R})_{\sigma}$ by Lemma \ref{lem_for_prop_for_Levidecom}.  
Moreover, the monomorphism $H' \times_{F} C \to H \times_{F} C$ induced by $H \subset H'$ is as follows:  
For a $C$-algebra $R$, $H' \times_{F} C (R) = H'(R) = \Aut_{D \otimes_{F} E' \otimes_{F} R}(V \otimes_{F} R) \subset \Aut_{D \otimes_{F} E \otimes_{F} R}(V \otimes_{F} R) = H(R) = H' \times_{F} C(R)$.  
Therefore, by Proposition \ref{prop_for_Levidecom} (1), the monomorphism 
\[
	\prod_{\sigma' \in X_{E'}} \underline{\Aut}_{D \otimes_{F} C_{\sigma'} } \left( V_{\sigma'}  \right) \hookrightarrow \prod_{\sigma \in X_{E}} \underline{\Aut}_{D \otimes_{F} C_{\sigma} } \left( V_{\sigma} \right)
\]
is given as $(f_{\sigma'})_{\sigma' \in X_{E'}} \mapsto (\prod_{\sigma' \in X_{E}, \sigma'|_{E}=\sigma} f_{\sigma'})_{\sigma \in X_{E}}$.  
\end{rem}

We put $I_{1} = \{ 1, \ldots, [E:F] \}$, $I_{2} = \{ 1, \ldots, [E':E] \}$ and $I_{3} = \{ 1, \ldots, [L:E'] \}$.  
Let $\sigma_1, \sigma_2, \ldots, \sigma_{[E:F]}$ be distinct elements in $X_{E}=\Hom_{F}(E, \bar{F})$.  
For $i \in I_{1}$, let $\sigma_{i,1}, \sigma_{i,2}, \ldots, \sigma_{i,[E':E]}$ be distinct elements in $\Hom_{F}(E', \bar{F})$ whose restrictions to $E$ are equal to $\sigma_{i}$.  
For $(i,j) \in I_{1} \times I_{2}$, let $\sigma_{i,j,1}, \sigma_{i,j,2}, \ldots, \sigma_{i,j,[L:E']}$ be distinct elements in $\Hom_{F}(L, \bar{F})$ whose restrictions to $E'$ are equal to $\sigma_{i,j}$.  
Then we have 
\begin{eqnarray*}
X_{E'} = \Hom_{F}(E', \bar{F}) &=& \left\{ \sigma_{i,j} | (i,j) \in I_{1} \times I_{2} \right\} \\
X_{L} = \Hom_{F}(L, \bar{F}) &=& \left\{ \sigma_{i,j,k} | (i,j,k) \in I_{1} \times I_{2} \times I_{3} \right\}
\end{eqnarray*}
as $L/F$ is separable.  
For $(i,j,k) \in I_{1} \times I_{2} \times I_{3}$ and an $\tilde{L}$-algebra $C$, we put $C_{i,j,k}:=C_{\sigma_{i,j,k}}$ and $V_{i,j,k}:=C_{\sigma_{i,j,k}}$.  
We can similarly define an $E$-algebra $C_{i}$, an $E'$-algebra $C_{i,j}$, a $D \otimes_{F} C_{i}$-module $V_{i}$ and a $D \otimes_{F} C_{i,j}$-module $V_{i,j}$.  

\begin{prop}
\label{Levidecom}
\begin{enumerate}
\item Let $C$ be an extension field of $\tilde{L}$.  
Then we have a commutative diagram of $C$-schemes:  

\[
	\xymatrix{
		T \times_{F} C \ar[r]^-{\cong} \ar[d] & \prod_{i,j,k} \underline{\Aut}_{D \otimes_{F} C_{i,j,k} } (V_{i,j,k}) \ar[d] \\
		H' \times_{F} C \ar[r]^-{\cong} \ar[d] & \prod_{i,j} \underline{\Aut}_{D \otimes_{F} C_{i,j} } \left( \bigoplus_{k} V_{i,j,k}  \right) \ar[d] \\
		H \times_{F} C \ar[r]^-{\cong} \ar[d] & \prod_{i} \underline{\Aut}_{D \otimes_{F} C_{i} } \left( \bigoplus_{j,k} V_{i,j,k}  \right) \ar[d] \\
		G \times_{F} C \ar[r]^-{\cong} &  \underline{\Aut}_{D \otimes_{F} C } \left( \bigoplus_{i,j,k} V_{i,j,k}  \right).   \\
	}
\]

\item We have a commutative diagram of $C$-vector spaces:
\[
	\xymatrix{
		\Lie(T \times_{F} C) \ar[r]^-{\cong} \ar[d] & \prod_{i,j,k} \End_{D \otimes_{F} C_{i,j,k} } (V_{i,j,k}) \ar[d] \\
		\Lie(H' \times_{F} C) \ar[r]^-{\cong} \ar[d] & \prod_{i,j} \End_{D \otimes_{F} C_{i,j} } \left( \bigoplus_{k} V_{i,j,k}  \right) \ar[d] \\
		\Lie(H \times_{F} C) \ar[r]^-{\cong} \ar[d] & \prod_{i} \End_{D \otimes_{F} C_{i} } \left( \bigoplus_{j,k} V_{i,j,k}  \right) \ar[d] \\
		\Lie(G \times_{F} C) \ar[r]^-{\cong} & \End_{D \otimes_{F} C } \left( \bigoplus_{i,j,k} V_{i,j,k}  \right),  \\
	}
\]
where the vertical morphisms are all monomorphisms.  

\item Let $c \in L$, and let $m_{c} \in \Lie(T) = \End_{D \otimes L}(V)$ be the map $v \mapsto cv$ for $v \in V$.  
We put $m_{c, C}=m_{c} \otimes_{F} 1 \in \Lie(T) \otimes _{F} C = \Lie(T \times_{F} C)$.  
When we regard $m_{c,C}$ as an element in $ \End_{D \otimes_{F} C } \left( \bigoplus_{i,j,k} V_{i,j,k}  \right)$ via the morphisms in (2), for $v_{i,j,k} \in V_{i,j,k}$ we have $m_{c,C}(v_{i,j,k})=v_{i,j,k} \cdot \sigma_{i,j,k}(c)$.  
\end{enumerate}
\end{prop}

\begin{prf}
(1) is the result from Proposition \ref{prop_for_Levidecom}.  
By taking the Lie algebra of (1), we obtain (2).  

To show (3), let $v \in V$, $l \in L$ and $b \in C$.  
Then $m_{c,C}(v \otimes_{L} l \otimes_{F} b) = m_{c,C}(lv \otimes_{F} b) = clv \otimes_{F} b = v \otimes_{L} (c \otimes_{F} 1 \cdot l \otimes_{F} b) \in V \otimes_{L} (L \otimes_{F} C)$.  
Here, let $v \in V$ and $a \in C_{i,j,k}$. 
Then we have $v \otimes_{L} a \in V \otimes_{L} C_{i,j,k} = V_{i,j,k} \subset V \otimes_{L} (L \otimes_{F} C)$ and $m_{c,C}(v \otimes_{L} a) = v \otimes_{L}\left( (c \otimes_{F} 1) \cdot a \right)$. 
Since $(c \otimes_{F} 1) \cdot a = \sigma_{i,j,k}(c)a$ by the $L$-algebra structure in $C_{i,j,k}$, we obtain $m_{c,C}(v \otimes_{L} a) = v \otimes_{L} \sigma_{i,j,k}(c)a = (v \otimes_{L} a) \cdot \sigma_{i,j,k}(c)$.  
Since $V_{i,j,k}$ is generated by elements of the form $v \otimes_{L} a$ for some $v \in V$ and $a \in C_{i,j,k}$, we obtain $m_{c,C}(v_{i,j,k})=v_{i,j,k} \cdot \sigma_{i,j,k}(c)$ for $v_{i,j,k} \in V_{i,j,k}$.  
\end{prf}

\begin{cor}
\label{findTTLS}
The sequence $(H', H, G)$ is a tame twisted Levi sequence.  
Moreover, $Z(H')/Z(G)$ is anisotropic.  
\end{cor}

\begin{prf}
We put $C=\tilde{L}$, which is a finite, tamely ramified Galois extension of $F$.  
Since $L$ is a maximal $F$-subfield in $A$, the right $D \otimes _{F} L$-module $V$ is simple.  
Then for any $(i,j,k) \in I_{1} \times I_{2} \times I_{3}$ and $C$-algebra $\tilde{C}$, we have 
\begin{eqnarray*}
	\underline{\End}_{D \otimes _{F} C_{i,j,k}}(V_{i,j,k})(\tilde{C}) & = & \End_{D \otimes _{F} C_{i,j,k} \otimes_{C} \tilde{C}}(V \otimes_{L} C_{i,j,k} \otimes_{C} \tilde{C}) \\
	& \cong & \End_{D \otimes_{F} L \otimes_{L} \tilde{C}}(V \otimes_{L} \tilde{C}) \\
	& \cong & \End_{D \otimes_{F} L} (V) \otimes_{L} \tilde{C} \\
	& \cong & L \otimes_{L} \tilde{C} \cong \tilde{C} = \underline{\End}_{C}(C) (\tilde{C}).  
\end{eqnarray*}
Therefore we have $\underline{\End}_{D \otimes _{F} C_{i,j,k}}(V_{i,j,k}) \cong \underline{\End}_{C}(C)$ as $C$-schemes.  
We also have 
\begin{eqnarray*}
	\prod_{i,j,k} \underline{\End}_{D \otimes_{F} C_{i,j,k}}(V_{i,j,k}) & \cong & \prod_{i,j,k} \underline{\End}_{C}(C), \\
	\prod_{i,j} \underline{\End}_{D \otimes_{F} C_{i,j}} \left( \bigoplus_{k} V_{i,j,k} \right) & \cong & \prod_{i,j} \underline{\End}_{C}\left( C^{\oplus |I_{3}|} \right), \\
	\prod_{i} \underline{\End}_{D \otimes_{F} C_{i}} \left( \bigoplus_{j,k} V_{i,j,k} \right) & \cong & \prod_{i} \underline{\End}_{C} \left( C^{ \oplus (|I_{2}| \times |I_{3}|) } \right), \\
	\underline{\End}_{D \otimes_{F} C} \left( \bigoplus_{i,j,k} V_{i,j,k} \right) & \cong & \underline{\End}_{C} \left( C^{\oplus (|I_{1}| \times |I_{2}| \times |I_{3}|}) \right).  
\end{eqnarray*}
By taking the multiplicative group, we obtain
\[
\begin{array}{rcccl}
	T \times_{F} C & \cong & \prod_{i,j,k} \underline{\Aut}_{D \otimes_{F} C_{i,j,k}}(V_{i,j,k}) & \cong & \mathbb{G}_{m}{}^{\times (|I_{1}| \times |I_{2}| \times |I_{3}|)}, \\
	H' \times_{F} C & \cong & \prod_{i,j} \underline{\Aut}_{D \otimes_{F} C_{i,j}} \left( \bigoplus_{k} V_{i,j,k} \right) & \cong &  \GL_{|I_{3}|}{}^{\times (|I_{1}| \times |I_{2}|)}, \\
	H \times_{F} C & \cong & \prod_{i} \underline{\Aut}_{D \otimes_{F} C_{i}} \left( \bigoplus_{j,k} V_{i,j,k} \right) & \cong &  \GL_{|I_{2}| \times |I_{3}|}{}^{\times |I_{1}| }, \\
	G \times_{F} C & \cong & \underline{\Aut}_{D \otimes{F} C} \left( \bigoplus_{i,j,k} V_{i,j,k} \right) & \cong & \GL_{|I_{1}| \times |I_{2}| \times |I_{3}|}.  
\end{array}
\]
Therefore $H' \times_{F} C$ and $H \times_{F} C$ are Levi subgroups in $G \times _{F} C$ with a split maximal torus $T \times_{F} C$.  
Since $C$ is a finite, tamely ramified Galois extension of $F$, the sequence $(H', H, G)$ is a tame twisted Levi sequence.  

Moreover, we have $\left( Z(H')/Z(G) \right) (F)=E'^{\times}/F^{\times}$, which is compact.  
Then $Z(H')/Z(G)$ is anisotropic.  
\end{prf}

Let $C=\bar{F}$.  
For distinct elements $(i',j',k'), (i'', j'', k'') \in I_{1} \times I_{2} \times I_{3}$, we define the root $\alpha_{(i',j',k'),(i'',j'',k'')} \in \Phi(G, T; \bar{F})$ as
\[
	\alpha_{(i',j',k'),(i'',j'',k'')} : \prod_{i,j,k} \Aut_{D \otimes_{F} \bar{F}_{i,j,k} } (V_{i,j,k}) \to \bar{F}^{\times} ; (t_{i,j,k})_{i,j,k} \mapsto t_{i',j',k'} t_{i'',j'',k''}^{-1}.  
\]

Therefore we have
\begin{eqnarray*}
	\Phi(H, T; \bar{F}) & = & \left\{ \alpha_{(i',j',k'),(i'',j'',k'')} \in \Phi(G, T; \bar{F}) | i'=i'' \right\} \\
	\Phi(H', T; \bar{F}) & = & \left\{ \alpha_{(i',j',k'),(i'',j'',k'')} \in \Phi(G, T; \bar{F}) | i'=i'', j'=j'' \right\}, 
\end{eqnarray*}
and we obtain
\[
	\Phi(H, T; \bar{F}) \setminus \Phi(H', T; \bar{F}) = \left\{ \alpha_{(i',j',k'),(i'',j'',k'')} \in \Phi(G, T; \bar{F}) | i'=i'', j' \neq j'' \right\}.  
\]

Moreover, the coroot $\check{\alpha}_{(i',j',k'),(i'',j'',k'')}$ with respect to $\alpha_{(i',j',k'),(i'',j'',k'')}$ is as follows:  
\[
	\check{\alpha}_{(i',j',k'),(i'',j'',k'')} : \bar{F}^{\times} \to \prod_{i,j,k} \Aut_{D \otimes_{F} \bar{F}_{i,j,k} } (V_{i,j,k}) \cong \prod_{i,j,k} \bar{F}^{\times} ; t \mapsto (t_{i,j,k})_{i,j,k}, 
\]
where 
\[
	t_{i,j,k} = \begin{cases}
		t & \left( (i,j,k) = (i',j',k') \right), \\
		t^{-1} & \left( (i,j,k) = (i'',j'',k'') \right), \\
		1 & otherwise.  
		\end{cases}
\]

Then we have $\dr \check{\alpha}_{(i',j',k'),(i'',j'',k'')}(u)=(u_{i,j,k})_{i,j,k}$ where
\[
	u_{i,j,k} = \begin{cases}
		u & \left( (i,j,k) = (i',j',k') \right), \\
		-u & \left( (i,j,k) = (i'',j'',k'') \right), \\
		0 & otherwise.  
		\end{cases}
\]

Conversely, we determine the set of tame twisted Levi subgroup $G'$ in $G$ with $Z(G')/Z(G)$ anisotropic.  

\begin{lem}
\label{detofTLG}
Let $G'$ be a tame twisted Levi subgroup of $G=\underline{\Aut}_{D}(V)$.  
Suppose $Z(G')/Z(G)$ is anisotropic.  
Then there exists a finite, tamely ramified extension $E$ of $F$ such that $G' \cong \Res_{E/F} \underline{\Aut}_{D \otimes _{F} E} (V)$.  
\end{lem}

\begin{prf}
Let $F^{\mathrm{tr}}$ be the maximal tamely ramified extension of $F$.  
Since $G'$ is a tame twisted Levi subgroup in $G$, $G'_{F^{\mathrm{tr}}}$ is a Levi subgroup in $G_{F^{\mathrm{tr}}} \cong \underline{\Aut}_{D \otimes F^{\mathrm{tr}}}(V \otimes F^{\mathrm{tr}})$.  
There exists a one-to-one relationship between Levi subgroups in $G_{F^{\mathrm{tr}}}$ and direct decompositions of $V \otimes F^{\mathrm{tr}}$ as a right $D \otimes F^{\mathrm{tr}}$-module.  
Let $V \otimes F^{\mathrm{tr}} = \bigoplus_{k=1}^{j} V_{k}$ be the corresponding decomposition with $G'_{F^{\mathrm{tr}}}$.  
Then we have $G'_{F^{\mathrm{tr}}} = \prod_{k=1}^{j} \underline{\Aut}_{D \otimes F^{\mathrm{tr}}}(V_{k})$.  
We remark that the right-hand-side group is the multiplicative group of $\underline{\End}_{D \otimes F^{\mathrm{tr}}}(V_{k})$ with a $Gal(F^{\mathrm{tr}}/F)$-action defined by its $F$-structure.  
Let $Z_{k}$ be the center of $\underline{\End}_{D \otimes F^{\mathrm{tr}}}(V_{k})$, which is $F^{\mathrm{tr}}$-isomorphic to $\underline{\End}_{F^{\mathrm{tr}}}(F^{\mathrm{tr}})$.  
Then $Z(G')_{F^{\mathrm{tr}}}$ is the multiplicative group of $Z=\prod_{k=1}^{j} Z_{k}$, equipped with the same $Gal(F^{\mathrm{tr}}/F)$-action.  
Therefore, we consider the structure of $Z_{k}$.  
Let $\mathbf{1}_{k}$ be (the $F^{\mathrm{tr}}$-rational point corresponding to) the identity element in $Z_{k}$.  
Since the $\Gal(F^{\mathrm{tr}}/F)$-action to $Z$ preserves the $F$-algebra structure, the set $\{ \mathbf{1}_{k} \mid k=1, \ldots, j \}$ is $\Gal(F^{\mathrm{tr}}/F)$-invariant.  
Then by changing the indices if necessary, we may assume there exist integers $0 = n_{0} < n_{1} < \cdots < n_{l} = j$ such that $\Gal(F^{\mathrm{tr}}/F)$ acts the set $\{ \mathbf{1}_{n_{i-1}+1}, \ldots, \mathbf{1}_{n_{i}} \}$ transitively for $l=1, \ldots, i$.  
We put $Y_{i}=\prod_{k=n_{i-1}+1}^{n_{i}} Z_{k}$.  
Since $a \in F^{\mathrm{tr}}, b \in Z$ and $\gamma \in \Gal(F^{\mathrm{tr}}/F)$ we have $\gamma (ab) = \gamma(a) \gamma(b)$ and $\{ \mathbf{1}_{n_{i-1}+1}, \ldots, \mathbf{1}_{n_{i}} \}$ is $\Gal(F^{\mathrm{tr}}/F)$-invariant, $Y_{i}$  is also $\Gal(F^{\mathrm{tr}}/F)$-invariant.  
Then $Y_{i}$ is defined over $F$.  
Let $X_{i}$ be the Galois descent of $Y_{i}$ to $F$.  
Let $\Gal(F^{\mathrm{tr}}/F_{i})$ be the stabilizer of $\mathbf{1}_{n_{i}}$. 
The fields $F_{i}$ is tamely ramified, and finite-dimensional over $F$ since $\Gal(F^{\mathrm{tr}}/F)/\Gal(F^{\mathrm{tr}}/F_{i})$ is $\Gal(F^{\mathrm{tr}}/F)$-isomorphic to the finite set $\{ \mathbf{1}_{n_{i-1}+1}, \ldots, \mathbf{1}_{n_{i}} \}$.  

We show $X_{i}$ is isomorphic to $\Res_{F_{i}/F} \underline{\End}_{F_{i}}(F_{i})$.  
If this follows, then we can show the multiplicative group of $X_{i}$ is $\Res_{F_{i}/F} \mathbb{G}_{m}$ and $Z(G') = \prod_{i=1}^{l} \Res_{F_{i}/F} \mathbb{G}_{m}$.  

Any $F^{\mathrm{tr}}$-rational point of $Y_{i}$ is uniquely represented as the form $\sum_{k'=n_{i-1}+1}^{n_{i}} a_{k'}\mathbf{1}_{k'}$, where $a_{k'} \in F^{\mathrm{tr}}$.  
Suppose $z=\sum_{k'=n_{i-1}+1}^{n_{i}} a_{k'}\mathbf{1}_{k'}$ is stabilized by $\Gal(F^{\mathrm{tr}}/F)$.  
For any $\gamma \in \Gal(F^{\mathrm{tr}}/F_{i})$, we have $z=\gamma(z)=\sum_{k'=n_{i-1}+1}^{n_{i}-1} \gamma(a_{k'})\gamma(\mathbf{1}_{k'}) + \gamma(a_{n_{i}}) \mathbf{1}_{n_{i}}$.  
Then we have $\gamma(a_{n_{i}}) = a_{n_{i}}$, that is, $a_{n_{i}} \in F_{i}$.  
For $n_{i-1}<k'<n_{i}$, we pick $\gamma_{k'} \in \Gal(F^{\mathrm{tr}}/F)$ such that $\gamma_{k'}(\mathbf{1}_{n_{i}})=\mathbf{1}_{k'}$.  
Then we have
\[
z = \gamma_{k'}(z) = \sum_{k''=n_{i-1}+1}^{n_{i}-1} \gamma_{k''}(a_{k''}) \gamma_{k'}(\mathbf{1}_{k'}) + \gamma_{k'}(a_{n_i}) \mathbf{1}_{k'}, 
\]
whence $a_{k'} = \gamma_{k'}(a_{n_i})$.  
Therefore any $F$-rational point of $X_{i}$ is the form
\[
\sum_{k'=n_{i-1}+1}^{n_{i}-1} \gamma_{k'}(a_{n_{i}}) \mathbf{1}_{k'} + a_{n_{i}} \mathbf{1}_{n_{i}}, 
\]
where $a_{n_{i}} \in F_{i}$, and the ring structure of $X_{i}(F)$ is isomorphic to $F_{i}$.  
Since the ring structure of $X_{i}(C)$ is isomorphic to $X_{i}(F) \otimes C$ for any $F$-algebra $C$, we obtain $X_{i} \cong \Res_{F_{i}/F} \underline{\End}_{F_{i}}(F_{i})$.  

We have shown $Z(G') = \prod_{i=1}^{l} \Res_{F_{i}/F} \mathbb{G}_{m}$.  
Since $Z(G')/Z(G)$ is anisotropic and $Z(G)=\mathbb{G}_{m}$, we have $l=1$ and $Z(G') = \Res_{E/F} \mathbb{G}_{m}$, where we put $E=F_{1}$.  

The field $E$ can be regarded as a $F$-subfield in $A$ via $X \subset \underline{\End}_{D}(V)$.  
We put $H=\underline{\Aut}_{D \otimes E}(V)$.  
Then $H$ is a tame twisted Levi subgroup in $G$ and we have $Z(H)=Z(G')$.  
Since there exists a one-to-one relationship between subtori in $G$ defined over $F$ and Levi subgroups in $G$ defined over $F$, we obtain $G'=H$.  
\end{prf}

\section{Embeddings of buildings for Levi sequences of $G$}
\label{BroussousLemaire}

\subsection{Lattice functions in $V$}

First, we recall the lattice functions in $V$ and their properties from \cite{BL}.  

\begin{defn}
The map $\Lcal$ from $\R$ to the set of $\ofra_{D}$-lattices in $V$ is a lattice function in $V$ if
\begin{enumerate}
\item we have $\Lcal(r)\varpi_{D}=\Lcal \left( r+(1/d) \right)$ for some uniformizer $\varpi_{D}$ of $D$ and $r \in \R$, 
\item $\Lcal$ is decreasing, that is, $\Lcal(r) \supset \Lcal(r')$ if $r \leq r'$, and
\item $\Lcal$ is left-continuous, where the set of lattices in $V$ is equipped with the discrete topology.  
\end{enumerate}
\end{defn}

The set of lattice functions in $V$ is denoted by $\Latt^{1}(V)$.  
The groups $G$ and $\R$ act on $\Latt^{1}(V)$ by $(g \cdot \Lcal)(r) = g \cdot \left( \Lcal(r) \right)$ and $(r' \cdot \Lcal) (r)=\Lcal(r+r')$ for $g \in G$, $r, r' \in \R$ and $\Lcal \in \Latt^{1}(V)$.  
These actions are compatible, and then $\Latt(V) := \Latt^{1}(V)/\R$ is equipped with the canonical $G$-action.  
The $G$-sets $\Latt^{1}(V)$ and $\Latt(V)$ are also equipped with an affine structure.  
Then there exists a canonical $G$-equivariant, affine isomorphism $\Bscr^{E}(G,F) \to \Latt^{1}(V)$.  
This isomorphism induces a $G$-equivariant, affine isomorphism $\Bscr^{R}(G,F) \to \Latt(V)$.  

We construct lattice functions from $\ofra_{D}$-sequences.  
Let $c \in \R$ and let $(\Lcal_{i})_{i \in \Z}$ be an $\ofra_{D}$-sequence with period $e$.  
Then 
\[
	\Lcal(r) = \Lcal_{\lceil de(r - c) \rceil}, \, r \in \R
\]
is a lattice function in $V$.  

\begin{prop}
\label{lfconstbylc}
Let $\Lcal$ be a lattice function in $V$.  
The following assertions are equivalent:  
\begin{enumerate}
\item $\Lcal$ is constructed from an $\ofra_{D}$-chain.  
\item There exists $c \in \R$ and $e \in \Z_{>0}$ such that the set of discontinuous points of $\Lcal$ is equal to $c+(de)^{-1}\Z$.  
\end{enumerate}
Moreover, if (1) (and (2)) holds, $e$ is equal to the period of some $\ofra_{D}$-chain which $\Lcal$ is constructed from.  
\end{prop}

\begin{prf}
First, suppose $\Lcal$ is constructed from an $\ofra_{D}$-chain.  
Then there exists $c \in \R$ and an $\ofra_{D}$-chain $(\Lcal_{i})_{i \in \Z}$ with period $e$ such that $\Lcal(r) = \Lcal_{\lceil de(r-c) \rceil}$ for $r \in \R$.  
Since $(\Lcal_{i})$ is an $\ofra_{D}$-chain, the set of discontinuous points of $\Lcal$ is equal to $c+(de)^{-1}\Z$, whence (2) holds.  

Conversely, suppose (2) holds.  
For $i \in \Z$, we put $\Lcal_{i}=\Lcal(c+(de)^{-1}i)$.  
Since $\Lcal$ is not right-continuous at $r=c+(de)^{-1}i$, we have 
\[
	\Lcal_{i} = \Lcal(c+(de)^{-1}i) \supsetneq \Lcal(c+(de)^{-1}(i+1)) = \Lcal_{i+1}.  
\]
Moreover, we also have
\[
	\Lcal_{i+e} = \Lcal(c+(de)^{-1}(i+e)) = \Lcal(c+(de)^{-1}i+d^{-1}) = \Lcal(c+(de)^{-1}i)\varpi_{D} = \Lcal_{i}\varpi_{D}.  
\]
Then $(\Lcal_{i})_{i \in \Z}$ is an $\ofra_{D}$-chain with period $e$.  

Let $\Lcal'$ be the lattice function constructed from $c \in \R$ and the $\ofra_{D}$-chain $(\Lcal_{i})$.  
We show $\Lcal = \Lcal'$.  
For $i \in \Z$, we have $\Lcal'(c+(de)^{-1}i) = \Lcal_{i} = \Lcal(c+(de)^{-1}i)$ and $\Lcal = \Lcal'$ on $c+(de)^{-1} \Z$.  
For $r \in \R$, there exists $i \in \Z$ such that $r \in ( c+(de)^{-1}(i-1), c+(de)^{-1}i ]$.  
Since the set of discontinuous points of $\Lcal$ is $c+(de)^{-1}\Z$, then $\Lcal|_{( c+(de)^{-1}(i-1), c+(de)^{-1}i ]}$ is continuous and 
\[
	\Lcal(r) = \Lcal(c+(de)^{-1}i) = \Lcal_{i} = \Lcal_{\lceil de \left( r-c \right) \rceil} = \Lcal'(r).  
\]
Therefore $\Lcal=\Lcal'$ is the lattice function constructed from the $\ofra_{D}$-chain $(\Lcal_{i})$ of period $e$.  
The last assertion follows from the above argument.  
\end{prf}

Conversely, for any lattice function $\Lcal$ there exists an $\ofra_{D}$-chain $(\Lcal_{i})_{i \in \Z}$ such that $\{ \Lcal(r) \mid r \in \R \} = \{ \Lcal_{i} \mid i \in \Z \}$, unique up to translation.  
Since $\Lcal(r+(1/d)) = \Lcal(r) \varpi_{D}$ for $r \in \R$, the period of $(\Lcal_{i})$ is equal to the number of discontinuous points of $\Lcal$ in $[0, 1/d)$.  

\subsection{Comparison of filtrations:  hereditary orders and Moy--Prasad filtration}

Let $x$ be an element in $\Bscr^{E}(G, F)$, corresponding to a lattice function $\Lcal$ via $\Bscr^{E}(G, F) \cong \Latt^{1}(V)$.  
We can define a filtration $(\mathfrak{a}_{x,r})_{r \in \R}$ in $A$ associated with $x$ as 
\[
	\mathfrak{a}_{x,r} = \mathfrak{a}_{\Lcal,r} = \{ a \in A \mid a\Lcal(r') \subset \Lcal(r+r'), \, r' \in \R \}
\]
for $r \in \R$.  
We also put $\mathfrak{a}_{x,r+} = \bigcup _{r<r'} \mathfrak{a}_{x,r'}$.  
Then we can define a hereditary $\ofra_{F}$-order $\Afra = \mathfrak{a}_{x,0}$ associated with $x$.  
The radical of $\Afra$ is equal to $\Pfra = \mathfrak{a}_{x,0+}$.  
We also put $\U_{0}(x) = \Afra^{\times}$, and  $\U_{r}(x) = 1 + \mathfrak{a}_{x,r}$ for $r \in \R_{> 0}$.  

\begin{prop}[{{\cite[Appendix A]{BL}}}]
\label{compoffiltA}
Let $x \in \Bscr^{E}(G,F)$.  
\begin{enumerate}
\item When we identify $A$ with the Lie algebra $\gfra(F)$ of $G$, we have $\mathfrak{a}_{x,r} = \gfra(F)_{x,r}$ for $r \in \R$.  
\item For $r \geq 0$, we have $\U_{r}(x) = G(F)_{x,r}$.  
\end{enumerate}
\end{prop}

Suppose $\Lcal$ is constructed from an $\ofra_{D}$-chain.  
Then there exist $c \in \R$ and an $\ofra_{D}$-chain $(\Lcal_{i})_{i \in \Z}$ with period $e$ such that $\Lcal(r) = \Lcal_{ \lceil de(r-c) \rceil}$.  
Since $\Lcal_{i+e} = \Lcal_{i}\varpi_{D}$ for $i \in \Z$, we have $\Lcal_{i+de} = \Lcal_{i}\varpi_{F}$, and then $e(\Afra|\ofra_{F}) = de$.  

\begin{prop}
\label{compoffiltG}
Let $x, \Lcal$ be as above, and let $r \in \R$.  
\begin{enumerate}
\item We have $\Pfra^{\lceil r \rceil} = \gfra(F)_{x, r/e(\Afra|\ofra_{F})}$.  
\item Suppose $r \geq 0$.  Then $\U^{\lceil r \rceil}(\Afra) = G(F)_{x,r/e(\Afra|\ofra_{F})}$.  
\item We have $\Kfra(\Afra) = G(F)_{[x]}$.  
\end{enumerate}
\end{prop}

\begin{prf}
We show (1).  
By Proposition \ref{compoffiltA} (1), it suffices to show $\Pfra^{\lceil r \rceil} = \mathfrak{a}_{\Lcal, r/e(\Afra|\ofra_{F})}$.  
We put $n= \lceil r \rceil$.  
Suppose $a \in \mathfrak{a}_{x, r/e(\Afra|\ofra_{F})}$.  
For $n' \in \Z$, we put $r_{n'} = c + e(\Afra|\ofra_{F})^{-1}n'$.  
Then we have $\Lcal(r_{n'}) = \Lcal_{\lceil de(r_{n'} -c) \rceil} = \Lcal_{n'}$, and $\Lcal\left( e(\Afra|\ofra_{F})^{-1}r+r_{n'} \right) = \Lcal_{\lceil de\left( e(\Afra|\ofra_{F})^{-1}r+r_{n'} -c \right) \rceil} = \Lcal_{n'+\lceil n \rceil}$.  
Since $a \in \mathfrak{a}_{\Lcal, r/e(\Afra|\ofra_{F})}$, in particular
\[
	a\Lcal_{n'} = a\Lcal(r_{n'}) \subset \Lcal(e(\Afra)|\ofra_{F})^{-1}r+r_{n'}) = \Lcal_{n+n'}
\]
for $n' \in \Z$.  
Since $\{ a \in A \mid a\Lcal_{n'} \subset \Lcal_{n+n'}, \, n' \in \Z \} = \Pfra^{n}$, we have $a \in \Pfra^{n}$.  

Conversely, suppose $a \in \Pfra^{n}$.  
For $r' \in \R$, we have $\Lcal(r') = \Lcal_{\lceil de(r'-c) \rceil}$ and $\Lcal(e(\Afra|\ofra_{F})^{-1}r+r') = \Lcal _{\lceil r+de(r'-c) \rceil}$.  
Since $\lceil r+de(r'-c) \rceil < r+de(r'-c)+1$ and $\lceil de(r'-c) \rceil \geq de(r'-c)$, we have 
\[
	\lceil r+de(r'-c) \rceil - \lceil de(r'-c) \rceil <r+de(r'-c)+1 - de(r'-c) = r+1.  
\]
Since $\lceil r+de(r'-c) \rceil - \lceil de(r'-c) \rceil \in \Z$, we also have $\lceil r+de(r'-c) \rceil - \lceil de(r'-c) \rceil \leq \lceil r \rceil$.  
When we put $n' = \lceil de(r'-c) \rceil$, we have $n + n' \geq \lceil r+de(r'-c) \rceil$.  
Therefore, 
\[
	a \Lcal(r') = a \Lcal_{ \lceil de(r'-c) \rceil} = a \Lcal_{n'} \subset \Lcal_{n+n'} \subset \Lcal_{\lceil r+de(r'-c) \rceil} = \Lcal(e(\Afra|\ofra_{F})^{-1}r+r') 
\]
for $r' \in \R$, which implies $a \in \mathfrak{a}_{\Lcal, r/e(\Afra|\ofra_{F})}$.  
Thus (1) holds.  

To show (2), it is enough to show $\U^{\lceil r \rceil}(\Afra) = \U_{r/e(\Afra|\ofra_{F})}(x)$ by Proposition \ref{compoffiltA} (2).  
Therefore (2) follows from (1).  

(3) is a corollary of \cite[I Lemma 7.3]{BL}, as $\Lcal$ is constructed from an $\ofra_{D}$-chain.  
\end{prf}

\begin{prop}
\label{compoffiltS}
Let $x \in \Bscr^{E}(G, F)$ correspond with a lattice function constructed from an $\ofra_{D}$-chain, and let $n \in \Z$.  
\begin{enumerate}
\item \begin{enumerate}
	\item $\Pfra^{n} = \gfra(F)_{x, n/e(\Afra|\ofra_{F})}$, 
	\item $\Pfra^{n+1} = \gfra(F)_{x, n/e(\Afra|\ofra_{F})+}$, 
	\item $\Pfra^{\lfloor (n+1)/2 \rfloor} = \gfra(F)_{x, n/2e(\Afra|\ofra_{F})}$, 
	\item $\Pfra^{\lfloor n/2 \rfloor +1} = \gfra(F)_{x, n/2e(\Afra|\ofra_{F})+}$.  
	\end{enumerate}
\item Suppose $n \geq 0$.  
	Then we have 
	\begin{enumerate}
	\item $\U^{n}(\Afra) = G(F)_{x, n/e(\Afra|\ofra_{F})}$, 
	\item $\U^{n+1}(\Afra) = G(F)_{x, n/e(\Afra|\ofra_{F})+}$, 
	\item $\U^{\lfloor (n+1)/2 \rfloor}(\Afra) = G(F)_{x, n/2e(\Afra|\ofra_{F})}$, 
	\item $\U^{\lfloor n/2 \rfloor +1}(\Afra) = G(F)_{x, n/2e(\Afra|\ofra_{F})+}$.  
	\end{enumerate}
\end{enumerate}
\end{prop}

\begin{prf}
We show (1), and (2) can be shown in the same way as (1).  

(a) follows from Proposition \ref{compoffiltG} (1).  
(c) also follows from Proposition \ref{compoffiltG} (1) and the fact $\lfloor (n+1)/2 \rfloor = \lceil n/2 \rceil$ for $n \in \Z$.  

We show (b).  
For $r \in (n, n+1]$, we have $\lceil r \rceil = n+1$.  
Then we have 
\begin{eqnarray*}
	\gfra(F)_{x,n/e(\Afra|\ofra_{F})+} & = & \bigcup_{n/e(\Afra|\ofra_{F})<r'} \gfra(F)_{x,r'} \\
	& = & \bigcup_{n/e(\Afra|\ofra_{F})<r' \leq (n+1)/e(\Afra|\ofra_{F})} \Pfra^{\lceil r'e(\Afra|\ofra_{F}) \rceil} \\
	& = & \Pfra^{n+1}.  
\end{eqnarray*}

To show (d), we consider two cases.  
First, suppose $n \in 2\Z$.  
Then we have $\Pfra^{\lfloor n/2 \rfloor + 1} = \Pfra^{(n/2)+1} = \gfra(F)_{x, ((n/2)+1)/e(\Afra|\ofra_{F})}$ by (a).  
Since $n/2 \in \Z$, for any $r \in ( n/2, (n/2)+1 ]$ we have $\lceil r \rceil = (n/2)+1$ and $\gfra(F)_{x,r/e(\Afra|\ofra_{F})} = \Pfra^{\lceil r \rceil} = \Pfra^{(n/2)+1}$.  
Therefore 
\begin{eqnarray*}
	\gfra(F)_{x,n/2e(\Afra|\ofra_{F})+} & = & \bigcup_{n/2e(\Afra|\ofra_{F}) < r'}\gfra(F)_{x, r'} \\
	& = & \bigcup_{n/2e(\Afra|\ofra_{F}) < r' \leq ((n/2)+1)/e(\Afra|\ofra_{F})}\Pfra^{\lceil r'e(\Afra|\ofra_{F}) \rceil} \\
	& = & \Pfra^{(n/2)+1} = \Pfra^{\lfloor n/2 \rfloor + 1}.  
\end{eqnarray*}

Next, suppose $n \in \Z \setminus 2\Z$.  
Then we have $\lfloor n/2 \rfloor + 1 = (n+1)/2 = \lceil (n+1)/2 \rceil$ and $\Pfra^{\lfloor n/2 \rfloor + 1} = \gfra(F)_{x,n/2e(\Afra|\ofra_{F})}$ by (b).  
Since $\lceil r \rceil = (n+1)/2$ for $r \in (n/2, (n+1)/2]$, we obtain
\begin{eqnarray*}
	\gfra(F)_{x, n/2e(\Afra|\ofra_{F})+} & = & \bigcup_{n/2e(\Afra|\ofra_{F}) < r'}\gfra(F)_{x, r'} \\
	& = & \bigcup_{n/2e(\Afra|\ofra_{F}) < r' \leq (n+1)/2e(\Afra|\ofra_{F})}\Pfra^{\lceil r'e(\Afra|\ofra_{F}) \rceil} \\
	& = & \Pfra^{(n+1)/2} = \Pfra^{\lceil n/2 \rceil +1}.  
\end{eqnarray*}
\end{prf}

Let $(H', H, G)$ be a tame twisted Levi sequence.  
Then there exists a tower $E'/E/F$ of tamely ramified extensions in $A$ such that $H' = \Res_{E'/F}\underline{\Aut}_{D \otimes_{F} E'}(V)$ and $H = \Res_{E/F}\underline{\Aut}_{D \otimes_{F} E}(V)$.  
We put $B=\Cent_{A}(E)$ and $B'=\Cent_{A}(E')$.  
There exist a division $E$-algebra $D_{E}$ and a right $D_{E}$-module $W$ such that $B \cong \End_{D_{E}}(W)$.  
Similarly, there exist a division $E'$-algebra $D_{E'}$ and a right $D_{E'}$-module $W'$ such that $B' \cong \End_{D_{E'}}(W')$.  
Since $E'/E/F$ is a tower of tamely ramified extensions, we have canonical identifications 
\begin{eqnarray*}
	\Bscr^{E}(H, F) \cong & \Bscr^{E}( \underline{\Aut}_{D \otimes E}(V), E) & \cong \Bscr^{E}(\underline{\Aut}_{D_{E}}(W), E), \\
	\Bscr^{E}(H', F) \cong & \Bscr^{E}( \underline{\Aut}_{D \otimes E'}(V), E') & \cong \Bscr^{E}(\underline{\Aut}_{D_{E'}}(W'), E').  
\end{eqnarray*}

Let $x \in \Bscr^{E}(H', F) \cong \Bscr^{E}(\underline{\Aut}_{D_{E'}}(W'), E')$, and let $\Lcal$ be the corresponding lattice function in $W'$ with $x$.  

\begin{prop}
\label{beingvertex}
The following assertions are equivalent.  
\begin{enumerate}
\item $[x]$ is a vertex in $\Bscr^{R}(H', F)$.  
\item The hereditary $\ofra_{E'}$-order $\Bfra'$ associated with $x$ is maximal.  
\item $\Lcal$ is constructed from an $\ofra_{D_{E'}}$-chain of period 1.  
\end{enumerate}
\end{prop}

\begin{prf}
The element $[x]$ is a vertex if and only if the stabilizer $\Stab_{H'(F)}(x)$ of $x$ in $H'(F)$ is a maximal compact subgroup in $H'(F)$.  
Since $\Lcal$ is identified with $x$ via the $H'(F)$-isomorphism $\Latt^{1}(W') \cong \Bscr^{E}(H', F)$, we have $\Stab_{H'(F)}(x) = \Stab_{\Aut_{D_{E'}}(W')}(\Lcal) = \U(\Bfra')$.  
The group $\U(\Bfra')$ is a maximal compact subgroup in $H'(F)$ if and only if $\Bfra'$ is maximal, which implies the equivalence of (1) and (2).  

To show the equivalence of (2) and (3), let $(\Lcal_{i})$ be an $\ofra_{D_{E'}}$-chain in $W'$ such that $\{ \Lcal(r) \mid r \in \R \} = \{ \Lcal_{i} \mid i \in \Z \}$.  
Since $\Bfra' = \{ b' \in B' \mid b' \Lcal(r) \subset \Lcal(r), \, r \in \R \} = \{ b' \in B' \mid b' \Lcal_{i} \subset \Lcal_{i}, \, i \in \Z \}$, the hereditary $\ofra_{E'}$-order $\Bfra'$ is maximal if and only if the period of $(\Lcal_{i})$ is equal to 1.  
Since the period of $(\Lcal_{i})$ is also equal to the number of discontinuous points of $\Lcal$ in $[0, 1/d_{E'})$, where $d_{E'} = (\dim_{E'}D_{E'})^{1/2}$, (2) holds if and only if there exists a unique discontinuous point $c$ in $[0,1/d_{E'})$.  
Here, since $\Lcal(r+(1/d_{E'}))=\Lcal(r)\varpi_{D_{E'}}$, $\Lcal$ is discontinuous at $c \in \R$ if and only if $\Lcal$ is discontinuous at the unique element $c'$ in $(c+d_{E'}^{-1} \Z) \cap [0, 1/d_{E'})$.  
Therefore (2) holds if and only if the discontinuous points of $\Lcal$ is equal to $c+d_{E'}^{-1} \Z$ for some $c \in \R$, which is also equivalent to (3) by Proposition \ref{lfconstbylc}.  
\end{prf}

We fix an $H'(F)$-equivalent, affine embedding $\iota_{H/H'}:\Bscr^{E}(H', F) \hookrightarrow \Bscr^{E}(H, F)$ and an $H(F)$-equivalent, affine embedding $\iota_{G/H}:\Bscr^{E}(H, F) \hookrightarrow \Bscr^{E}(G, F)$.  
We also put $\iota_{G/H'} = \iota_{G/H} \circ \iota_{H/H'}$.  

\begin{prop}
\label{resultofBL}
Let $x \in \Bscr^{E}(H, F)$.  
\begin{enumerate}
\item The canonical identification $\Bscr^{E}(H, F) \cong \Bscr^{E}(\underline{\Aut}_{D_{E}}(W), E)$ and $\iota_{G/H}$ induce 
\[
	j:\Bscr^{R}(\underline{\Aut}_{D_{E}}(W), E) \hookrightarrow \Bscr^{R}(G, F), 
\]
which is equal to $j_{E}^{-1}$ in \cite[II-Theorem 1.1]{BL}.  
\item Let $(\mathfrak{a}_{\iota_{G/H}(x),r})_{r \in \R}$ be the filtration in $A$ associated with $\iota_{G/H}(x)$, and let $(\mathfrak{b}_{x,r})_{r \in \R}$ be the filtration in $B$ associated with $x$.  
Then 
\[
	\mathfrak{b}_{x, r} = B \cap \mathfrak{a}_{\iota_{G/H}(x), r/e(E/F)}.  
\]
\item The hereditary $\ofra_{F}$-order $\mathfrak{a}_{\iota_{G/H}(x), 0}$ is $E$-pure.  
\end{enumerate}
\end{prop}

\begin{prf}
Since $\Bscr^{E}(H, F) \cong \Bscr^{E}(\underline{\Aut}_{D_{E}}(W), E)$ and $\iota_{G/H}$ are $H(F)$-equivalent and affine, they induce the $H(F)$-equivalent, affine embedding
\[
	j:\Bscr^{R}(\underline{\Aut}_{D_{E}}(W), E) \cong \Bscr^{R}(H', F) \hookrightarrow \Bscr^{R}(G, F).  
\]
However, $H(F)$-equivalent, affine embedding $\Bscr^{R}(\underline{\Aut}_{D_{E}}(W), E) \hookrightarrow \Bscr^{R}(G, F)$ is unique.  
Since $j$ and $j_{E}^{-1}$ are $H(F)$-equivalent and affine, we obtain $j = j_{E}^{-1}$.  
The remainder assertions are results from \cite[II-Theorem 1.1]{BL}.  
\end{prf}

\begin{prop}
\label{beingprincipal}
Let $x \in \Bscr^{E}(H', F)$ such that $[x]$ is a vertex.  
\begin{enumerate}
\item The corresponding lattice function $\Lcal$ in $W$ with $\iota_{H/H'}(x)$ is constructed from a uniform $\ofra_{D_{E}}$-chain.  
In particular, the hereditary $\ofra_{E}$-order $\Bfra$ in $B$ associated with $\Lcal$ is principal.  
\item Let $\Bfra'$ be the hereditary $\ofra_{E'}$-order in $B'$ associated with $x$.  
Then $\Bfra$ is the unique $E'$-pure hereditary $\ofra_{E}$-order in $B$ such that $\Bfra' = B' \cap \Bfra$.  
\end{enumerate}
\end{prop}

\begin{prf}
By Proposition \ref{beingvertex}, the corresponding lattice function in $W'$ with $x$ is constructed from an $\ofra_{D_{E'}}$-chain with period 1.  
Since an $\ofra_{D_{E'}}$-chain with period 1 is uniform, (1) follows from Proposition \ref{resultofBL} and \cite[II-Proposition 5.4]{BL}.  
The claim (2) follows from Proposition \ref{resultofBL} and \cite[Lemme 1.6]{S2}.  
\end{prf}

We regard $\Bscr^{E}(H', F)$ as a subset in $\Bscr^{E}(H, F)$ via $\iota_{H/H'}$, and $\Bscr^{E}(H, F)$ as a subset in $\Bscr^{E}(G, F)$ via $\iota_{G/H}$.  

\begin{prop}
\label{compoffiltC}
Let $x \in \Bscr^{E}(H', F)$ such that $[x]$ is a vertex.  
Let $\Afra$ be the hereditary $\ofra_{F}$-order in $A$ associated with $x \in \Bscr^{E}(G, F)$, and let $\Pfra$ be the radical of $\Afra$.  
We put $\mathfrak{h}(F) = \Lie(H) = B$.  
\begin{enumerate}
\item Let $n \in \Z$.  
	\begin{enumerate}
	\item $B \cap \Pfra^{n} = \mathfrak{h}(F)_{x, n/e(\Afra|\ofra_{F})}$, 
	\item $B \cap \Pfra^{n+1} = \mathfrak{h}(F)_{x, n/e(\Afra|\ofra_{F})+}$, 
	\item $B \cap \Pfra^{ \lfloor (n+1)/2 \rfloor } = \mathfrak{h}(F)_{x, n/2e(\Afra|\ofra_{F})}$, 
	\item $B \cap \Pfra^{ \lfloor n/2 \rfloor +1 } = \mathfrak{h}(F)_{x, n/2e(\Afra|\ofra_{F})+}$.  
	\end{enumerate}
\item Let $n \in \Z_{\geq 0}$.  
	\begin{enumerate}
	\item $B^{\times} \cap \U^{n}(\Afra) = H(F)_{x, n/e(\Afra|\ofra_{F})}$, 
	\item $B^{\times} \cap \U^{n+1}(\Afra) = H(F)_{x, n/e(\Afra|\ofra_{F})+}$, 
	\item $B^{\times} \cap \U^{ \lfloor (n+1)/2 \rfloor }(\Afra) = H(F)_{x, n/2e(\Afra|\ofra_{F})}$, 
	\item $B^{\times} \cap \U^{ \lfloor n/2 \rfloor +1}(\Afra) = H(F)_{x, n/2e(\Afra|\ofra_{F})+}$.  
	\end{enumerate}
\end{enumerate}
\end{prop}

\begin{prf}
By \cite[Proposition 1.9.1]{Ad}, we have $B \cap \gfra_{x, r} = \mathfrak{h}(F) \cap \gfra_{x, r} = \mathfrak{h}_{x, r}$ for $r \in \tilde{\R}$ and $B^{\times} \cap G(F)_{x, r} = H(F) \cap G(F)_{x, r} = H(F)_{x, r}$ for $r \in \tilde{\R}_{\geq 0}$.  
On the other hand, $x \in \Bscr^{E}(G, F)$ is constructed from an $\ofra_{D}$-chain by Proposition \ref{beingprincipal}.  
Then we can apply Proposition \ref{compoffiltS} and assertions follow.  
\end{prf}

\begin{lem}
\label{lemforfindc}
Let $x \in \Bscr^{E}(H', F)$ such that $[x]$ is a vertex.  
Let $\Bfra$ be the hereditary $\ofra_{E}$-order in $B$ with $x \in \Bscr^{E}(H, F)$, and let $\Qfra$ be the radical of $\Bfra$.  
\begin{enumerate}
\item For $r \in \R$, we have $\Qfra^{\lceil r \rceil} =\mathfrak{h}(F)_{x, r/e(\Bfra|\ofra_{E})e(E/F)}= \mathfrak{b}_{x, r/e(\Bfra|\ofra_{E})}$.  
\item For $r \in \R_{\geq 0}$, we have $\U^{\lceil r \rceil}(\Bfra) = H(F)_{x, r/e(\Bfra|\ofra_{E})e(E/F)}$.  
\item Let $r \in \R_{\geq 0}$.  
If $H(F)_{x, r} \neq H(F)_{x, r+}$, then $n=re(\Bfra|\ofra_{E})e(E/F)$ is an integer, and we have 
	\begin{enumerate}
	\item $H(F)_{x, r} = \U^{n}(\Bfra)$, 
	\item $H(F)_{x, r+} = \U^{n+1}(\Bfra)$, and
	\item $H(F)_{x, r/2+} = \U^{\lfloor n/2 \rfloor +1}(\Bfra)$.  
	\end{enumerate}
\end{enumerate}
\end{lem}

\begin{prf}
We show (1), and (2) follows from (1).  
Let $(\mathfrak{a}_{x,r})$ be the filtration in $A$ with $x$, and let $(\mathfrak{b}_{x,r})$ be the filtration in $B$ with $x$.  
Since $[x] \in \Bscr^{R}(H', F)$ is a vertex, by Proposition \ref{beingprincipal} (1) $x \in \Bscr^{E}(H, F)$ is constructed from an $\ofra_{D_{E}}$-chain.  
Then by Proposition \ref{compoffiltA} and Proposition \ref{compoffiltG} we have $\Qfra^{\lceil r \rceil} = \mathfrak{b}_{x, r/e(\Bfra|\ofra_{E})}$.  
On the other hand, by Proposition \ref{resultofBL} (2), we also have $\mathfrak{b}_{x, r/e(\Bfra|\ofra_{E})}=B \cap \mathfrak{a}_{x, r/e(\Bfra|\ofra_{E})e(E/F)}$.  
Since $\mathfrak{a}_{x,r/e(\Bfra|\ofra_{E})e(E/F)} = \gfra(F)_{x, r/e(\Bfra|\ofra_{E})e(E/F)}$ by Proposition \ref{compoffiltA}, we obtain $\Qfra^{\lceil r \rceil} = \mathfrak{h}(F) \cap \gfra(F)_{x, r/e(\Bfra|\ofra_{E})e(E/F)} = \mathfrak{h}(F)_{x, r/e(\Bfra|\ofra_{E})e(E/F)}$, where the last equality follows from \cite[Proposition 1.9.1]{Ad}.  

To show (3), let $r \in \R_{\geq 0}$ and suppose $H(F)_{x, r} \neq H(F)_{x, r+}$.  
If $re(\Bfra|\ofra_{E})e(E/F) \notin \Z$, then $( re(\Bfra|\ofra_{E})e(E/F) , \lceil re(\Bfra|\ofra_{E})e(E/F) \rceil ]$ is nonempty and
\begin{eqnarray*}
	H(F)_{x, r+} & = & \bigcup_{r < r'} H(F)_{x, r'} \\
	& = & \bigcup_{re(\Bfra|\ofra_{E})e(E/F) < r' \leq \lceil re(\Bfra|\ofra_{E})e(E/F) \rceil} H(F)_{x, r'/e(\Bfra|\ofra_{E})e(E/F)} \\
	& = &  \U^{\lceil re(\Bfra|\ofra_{E})e(E/F) \rceil}(\Bfra) \\
	& = & H(F)_{x,r}, 
\end{eqnarray*}
which is a contradiction.  Therefore we have $n=re(\Bfra|\ofra_{E})e(E/F) \in \Z$.  
We put $G' = \underline{\Aut}_{D \otimes_{F} E}(V)$.  
Then we can regard $x$ as an element in $\Bscr^{E}(G', E)$, and for any $r' \in \R_{\geq 0}$ we have $\U^{\lceil r' \rceil}(\Bfra) = G'(E)_{x, r'/e(\Bfra|\ofra_{E})}$ by Proposition \ref{compoffiltG} (2).   
Therefore we obtain $H(F)_{x,r'} = \U^{\lceil r'e(\Bfra|\ofra_{E})e(E/F) \rceil}(\Bfra) = G'(E)_{x, r'e(E/F)}$ and $H(F)_{x,r'+} = G'(E)_{x, r'e(E/F)+}$ for $r' \in \R$.  
Then by Proposition \ref{compoffiltS} (2)-(a), (b) and (d), 
\begin{eqnarray*}
	\U^{n}(\Bfra) = & G'(E)_{x,re(E/F)} & = H(F)_{x,r}, \\
	\U^{n+1}(\Bfra) = & G'(E)_{x, re(E/F)+} & = H(F)_{x,r+} \\
	\U^{\lfloor n/2 \rfloor +1}(\Bfra) = & G'(E)_{x, re(E/F)/2+} & = H(F)_{x, r/2+}, 
\end{eqnarray*}
which completes the proof of (3).  
\end{prf}

\section{Generic elements and generic characters of $G$}
\label{Generic}

In this section, we discuss generic elements and generic characters, using descriptions of tame twisted Levi subgroups in $G$, given in \S \ref{TTLS}. Moreover, we relate minimal elements to generic characters using standard representatives, a notion used by Howe \cite{Ho} in the $\GL_N$ tame case. 

\subsection{Standard representatives}
 In this section we introduce and discuss standard representatives. Certain results appear  in \cite{Ho}. Since we can not find detailed proofs, we give a complete exposition.
We fix a uniformizer $\varpi_{F}$ of $F$.  
Let $E$ be a finite, tamely ramified extension of $F$.  
Then we can consider the subgroup $C_{E}$ of ``standard representatives" in $E^{\times}$.  
We recall the construction of $C_{E}$.  

\begin{lem} \label{corost} There exists a uniformiser $\varpi _E$ of $E$ and a root of unity $z \in E$, of order prime to $p$, such that $\varpi _ E ^e z = \varpi _F$.

\end{lem}
 \begin{prf} Let $\varpi$ be a uniformiser of $E$. Let $q=p^f$ be the number of elements in the residue field of $E$. Let  
  $\mu _{q-1}$  denote the group of $(q-1)$-th roots of unity in $E$. Then \cite[Chapter 2 Proposition 5.7]{Neuk} shows that there exists an isomorphism $f:E^{\times} \simeq \varpi ^{\Z} \times \mu _{q-1} \times G' $ where $G' = 1+ \mathfrak{p} _E$ is a multiplicatively denoted group. Each element of $G'$ has an $e$-th root. Indeed, \cite[Chapter 2 Proposition 5.7]{Neuk} shows that $1+ \mathfrak{p}  _E$ is isomorphic to an additive
  group $\Z / p ^a \Z \times \Z _p ^d $ or to an additive group $\Z _p ^{\mathbf{N}}$. The
   image of $\varpi _F$ by $f$ is $(e,z,g)$, where ${(e,z,g)\in \varpi _E ^{\Z} \times \mu _{q-1} \times G' }$; that is, $\varpi_F = \varpi  ^e z g $. Let $r$ be in $G'$ such that $r^e =g$. Then $r \varpi $ is a uniformiser of $E$ and  $\varpi _F =(r \varpi )^e z$. So $\varpi _E = r \varpi $ has the required property.

 \end{prf} 

\begin{defn}
Let $E/F$ be a finite, tamely ramified extension, and let $\varpi_{E} \in E$ be as above.  
We denote by $C_{E}$ the subgroup in $E^{\times}$ which is generated by $\varpi_{E}$ and roots of unity in $E$ with order prime to $p$.  
\end{defn}

\begin{prop} \label{ceindep}The group $C_E$ is well-defined, i.e. it is independent of the choice of 
$\varpi _E$ used to define it.
\end{prop}

\begin{prf} Let $\varpi _1$ and $\varpi _2$ be two uniformisers of $E$ and $z_1$, $z_2$ be two roots of unity of order prime to $p$ such
 that ${\varpi ^e z_1 = \varpi _F} $ and $\varpi _2 ^e z_2 = \varpi _F $. Let $C^1$ be the 
group generated by $\varpi _1$ and the root of unity of order prime to $p$. Let $C^2$ be the group generated by $\varpi _2$ and the root of unity
 of order prime to $p$. By symmetry, it is enough show that $C^1 \subset C^2$. It is also enough to show that $\varpi _1 \in C_2$. The equation $\varpi _1 ^e z_1 = \varpi _F $ implies
 that $\varpi _1 ^e \in C_2$, thus there exists a root of unity $z$ of order prime to $p$ such that $\varpi _1 ^e = \varpi _2 ^e z$. We have $(\varpi _1 \varpi _2 ^{-1} )^e = z $. Let $o_z$ be the order of $z$, which is an integer prime
  to $p$. We have  $(\varpi _1 \varpi _2 ^{-1} )^{eo_z} = 1 $.  The integer $e o_z$ is prime to $p$, indeed $e=e(E | F)$ is prime to $p$ because $E/F$ is a tamely ramified extension and $o_z$ is prime to $p$. Consequently $\varpi _1 \varpi _2 ^{-1}$ is a root of unity of order prime to $p$. This implies that $\varpi _1 \in C_2$, as required.

\end{prf}
Then $C_{E}$ depends only the choice of $\varpi_{F}$, which we already fixed.  
We recall properties of $C_{E}$.  

\begin{prop}
\label{propforsr}
Let $E/F$ be a finite, tamely ramified extension.  
\begin{enumerate}
\item Let $c \in E^{\times}$.  
Then there exists a unique $\sr(c) \in C_{E}$, called the standard representative of $c$, such that $\sr(c) \in c(1+\pfra_{E})$.  
\item For any $c \in E^{\times}$, the standard representative $\sr(c)$ is the unique element in $C_{E}$ such that $\ord \left( \sr(c)-c \right) >\ord(c)$.  
\item Let $E'/E$ be also a finite, tamely ramified extension.  
Then we have an inclusion $C_{E} \subset C_{E'}$ as groups.  
\item Let $s \in C_{E}$.  
Let $\sigma_{1}, \sigma_{2} \in \Hom_{F}(E, \bar{F})$ such that $\sigma_{1}(s) \neq \sigma_{2}(s)$.  
Then we have $\ord \left( \sigma_{1}(s) - \sigma_{2}(s) \right) = \ord (s)$.  
\end{enumerate}
\end{prop}

\begin{prf}
 We have $E^{\times} \simeq C_E \times ( 1+ \mathfrak{p} _E )$ and (1) is a consequence.
The element $sr(c)$ is the unique element in $C_E$ such that $c=sr(c) \times (1+ y)$, with $y \in \mathfrak{p}_E$. Thus, $sr(c)$ is the unique element in $C_E$ such that $c- sr(c) \in sr(c) \mathfrak{p} _E$. Thus, (2) holds, remarking that $sr(c)$ and $c$ have the same valuation.
Recall that the group $C_E$ and $C_{E'}$ are independent of the choices of uniformisers used to define them. Let $\varpi _E$ be a uniformiser of $E$ and $z$ a root of unity of order prime to $p$ in $E$ such that $\varpi _E ^{e(E|F)} z = \varpi _F$. Because $E'/E$ is tamely ramified, there exists a uniformiser $\varpi _{E'} \in E'$ and a root of unity $w$ of order prime to $p$ in $E'$ such that $\varpi _{E'} ^{e(E' | E )} w = \varpi _E$. Elevating to the power $e(E| F)$, we have $\varpi _{E'} ^{e(E' | E ) e(E | F)} w ^{e(E| F)}= \varpi _E ^{e(E| F)}$. We thus get $\varpi _{E'} ^{e(E| F)} w ^{e(E| F)} z  = \varpi _F$. The element $w ^{e(E|F)} z $ is a root of unity of order prime to $p$.
 Consequently, $C_{E'}$ is the group generated by $\varpi _{E'}$ and the roots of unity of order prime to $p$ in $E'$. The equation $\varpi _{E'} ^{e(E' | E )} w = \varpi _E$ shows that $\varpi  _E$ is inside $C_{E'}$. Trivially, the roots of unity of order prime to $p$ in $E$ are inside the roots of unity of order prime to $p$ in $E'$. Consequently $C_E$ is inside $C_{E'}$, as required for (3).

We now prove the claim (4) when the extension $E/F$ is Galois.  Let $\sigma \in Gal(E/F)$, and let $\varpi _E$ be an element such that $\varpi _E ^e z = \varpi _F$ for $z $ a root of unity in $E$ of order prime to $p$. Let $o_z$ be the order of $z$. It is enough to show that $z$ and $\varpi _E$ are mapped in $C_E$ by $\sigma$. The equality $(\sigma (z))^{o_z}=1 $ shows that $\sigma (z) $ is a root of unity of order prime to $p$ and thus inside $C_E$. The equality $\sigma (\varpi _E )^e \sigma (z) = \varpi _F$ together with Proposition \ref{ceindep} show that we can use $\sigma (\varpi _E )$ to define $C_E$, and thus $\sigma  (\varpi _E )$ is inside $C_E$.
      The element $\sigma _1(s)$ is in $C_E$, so $sr(\sigma _1 (s))=\sigma _1 (s)$.
      Consequently 
      $v_E (\sigma _ 1(s) - \sigma _2(s))= v _E (\sigma_1(s))$,
       indeed assume ${v _E (\sigma _1 (s) -\sigma _2 (s) )\not = v_E (\sigma _1 (s) )}$, then $v_E (\sigma _1 (s) -\sigma _2 (s) )> \sigma _1 (s) $, and so $\sigma _2 (s) = sr(\sigma _1 (s)) = \sigma _1 (s)$ by the previous criterion (2), this is a contradiction. This prove (4) for Galois extensions.
For general $E$, since $E/F$ is tamely ramified, then $E/F$ is separable and we can take the Galois closure $\tilde{E}$ of $E$ in $\bar{F}$.  
Then $\tilde{E}/F$ is a finite Galois extension.  
Let $\tilde{\sigma}_{1}, \tilde{\sigma}_{2} \in \Hom_{F}(\tilde{E}, \bar{F})$ be a extension of $\sigma_{1}, \sigma_{2}$, respectively.  
By (3), $s \in C_{E} \subset C_{\tilde{E}}$.  
We also have $\tilde{\sigma}_{1}(s) = \sigma_{1}(s) \neq \sigma_{2}(s) = \tilde{\sigma}_{2}(s)$.  
Therefore by applying (4) for $\tilde{E}/F$ we have
\[
	\ord (s) = \ord \left( \tilde{\sigma}_{1}(s) -\tilde{\sigma}_{2}(s) \right) = \ord \left( \sigma_{1}(s) - \sigma_{2}(s) \right), 
\]
which is what we wanted.  
\end{prf}

By using standard representatives, we can judge whether some element in $E$ is minimal or not.  

\begin{prop}
\label{srformin}
Let $E/F$ be a finite, tamely ramified extension, and let $c \in E^\times$ such that $E=F[c]$.  
Then the following assertions are equivalent.  
\begin{enumerate}
\item $c$ is minimal over $F$.  
\item $E=F[\sr(c)]$.  
\item Put $\ord(c)=-r$. For all morphisms of $F$-algebras $\sigma  \neq \sigma ' $ from $E$ to $\overline{F}$, we have \[\ord (\sigma ( c ) - \sigma ' (c)  ) =-r .\]
\end{enumerate}
\end{prop}

\begin{prf} We need a Lemma.

 \begin{lem} \label{reduuni} Let $E/F$ be a finite unramified extension. Let $z \in E$ be a root of unity of order prime to $p$. Then, $z$ generates $E/F$ if and only if $z+ \mathfrak{p} _E$ generates the residual field extension $k_E /k_F$.
 \end{lem} 

 \begin{prf} If $z$ generates $E$ over $F$, then $z$ generates $\mathfrak{o} _E $ over $\mathfrak{o}_F$ by \cite[7.12]{Neuk}. Thus, $z$ generates the residual field extension $k_E /k_F $.  Let us check the reverse implication.  Assume that $z+ \mathfrak{p}_E$ generates $k_E /k_F $. The field extension $E/F$ is unramified, so  $[k_E : k_F] = [E:F]$. Let $P_z \in F[X]$ be the minimal polynomial of $z$ and $d$ its degree, clearly $P_z $ is in $\mathfrak{o}_F[X]$. It is enough to show that $d=[E:F]$.  We have ${d \leq [E:F]}$. The reduction $\mod \mathfrak{p}_E$ of $P_z$ is of degree $d$ and annihilates $z+ \mathfrak{p} _E$, a generator of $k_E / k_F$, and thus $[k_E : k_F] \leq d $. So ${[k_E : k_F] \leq d \leq [E:F]}$. So $d =[E:F]$, and this concludes the proof.\end{prf}
 
 We now prove Proposition \ref{srformin}. 
Let us prove that (1) implies (2). Assume that $c $ is minimal over $F$. Let us remark that the definition of $sr(c)$ implies trivially that $F[sr(c)] \subset E$. Let $E^{\nr}$ denote the maximal unramified extension contained in $E$. To prove the opposite inclusion $ E \subset F[sr(c)]$, it is enough to show that  ${E^{\nr}\subset F[sr(c)]}$ and ${E\subset E^{\nr}[sr(c)]}$. Put $v=v_E(c)$, $ e =e(E |F)$. The valuation of  $\varpi_F^{-v}c^e$ is equal to $0$, consequently by Proposition \ref{propforsr} (2) we have  $v_E(sr(\varpi_F^{-v}c^e)-\varpi_F^{-v}c ^e)>0$, and so ${sr(\varpi_F^{-v}c^e)+\mathfrak{p}_E=\varpi_F^{-v}c ^e+\mathfrak{p}_E.}$ We have $sr(\varpi_F^{-v}c ^e)=\varpi_F^{-v}sr(c)^e$, and this is a root of unity of order prime to $p$. The definition of being minimal implies that $\varpi_F^{-v}sr(c)^e +\mathfrak{p}_E$ generates $k_E / k_F$. So $\varpi_F^{-v}sr(c)^e$ generates $E^{\nr}$ by Lemma \ref{reduuni}. So $E^{\nr}\subset F[sr(c)]$.
We have $v_E(c)=v_E(sr(c)),$ so $\mathrm{gcd}(v_E(sr(c)),e)=1.$ Let $a $ and $b$  be integers such that $av_E(sr(c))+be=1$. Thus, $v_E(sr(c)^a \varpi_F ^b)=1$ and so $E^{\nr}[sr(c)^a \varpi_F ^b]=E$ because a finite totally ramified extension is generated by an arbitrary uniformiser. So $E^{\nr}[sr(c)]=E$ and ($i$) hold. We have thus shown that ${E^{\nr}\subset F[sr(c)]}$ and ${E\subset E^{\nr}[sr(c)]}$ and so  ($i$) implies ($ii$).

Let us prove that (2) implies (1). Assume that $F[sr(c)]=E$. We start by showing that $e$ is prime to $v$. The field $E ^{\nr}$ is generated over $F$ by the roots of unity of order
 prime to $p$ contained in $E$. Let $d=\gcd (v , e)$ and $b=\frac{e}{d}$. Let $\varpi_E$ be a uniformiser in $E$ such that $\varpi _E ^{e} z = \varpi _F$ with $z$ a root of unity of order prime to $p$. The element $sr(c ) $ is in $C_E$ and so $sr(c) = \varpi _E ^{v} w$ with $w$ a root of unity of order prime to $p$ in $E$. Equalities 
${sr(c)^b = (\varpi _E ^e ) ^{\frac{v}{gcd(v , e)} }w^b= (\varpi _F z ^{-1})^{\frac{v}{gcd(v , e)} }w^b}$ show that $sr(c) ^b$ is contained in $E^{\nr}$. By hypothesis, the element $sr(c)$ generates
 $E$ over $F$ and so generates $E$ over $E^{\nr}$. Consequently, the field $E$ is generated
 by an element whose $b$-th power is in $E^{\nr}$. Therefore, the inequality  $[E:E^{\nr}]\leq b$ holds. The extension $E^{\nr}$ is the maximal unramified extension contained in $E$, so $[E:E^{\nr}]=e$. Thus, the inequality  $e \leq b \leq \frac{e}{d}$ holds. This implies $d=1$ and so $v$ is prime to $e$.
 Let us prove that $\varpi _F ^{- v } c ^e + \mathfrak{p} _E $ generates the residue field extension $k _E $ over $k _F$. Because $\varpi _F ^{- v} c ^e + \mathfrak{p} _E = \varpi _F ^{- v} sr(c)^e + \mathfrak{p} _E$, it is equivalent to show that $x+\mathfrak{p} _E$ generates $k_E$ over $k _F$ , where $x=\varpi _F ^{-v} sr (c) ^e. $ The element $sr(c)$ generates $E$ over $F$ by hypothesis; that is, $E=F[sr(c)]$.  So the inequality $[E : F[x]]\leq e$ holds, indeed $E $ is generated over $F[x]$ by the element $sr(c)$ whose $e$-th power is in $F[x]$. Because  $x$ is a root of unity of order prime to $p$, the field $F[x]$ is a subset of $E^{\nr}$, so $[E : E^{\nr}] \leq [E : F[x]] $. Consequently, the identity $e=[E : E^{\nr}] \leq [E : F[x]] \leq e $ holds. Because $F[x] \subset E^{\nr}$, the previous identity implies that $F[x]=E^{\nr}$. Thus by \ref{reduuni} the element $x +\mathfrak{p} _E$ generates $k_E$ over $k_F$. So $c$ is minimal over $F$.

 Let us prove that (2) is equivalent to (3). Let $\sigma \neq \sigma '$ be two morphisms of $F$-algebras from $E$ to $\overline{F}$. Put \begin{align*}
 &A= \sigma (c) - \sigma ' (c) \\
 & B = \sigma (sr(c)) - \sigma ' ( sr(c)).
 \end{align*}
 We have $\ord (A) \geq -r $ and $\ord (B) \geq -r$. We have \begin{align*}
 \ord (A-B)& = \ord \big( \sigma (c) - \sigma ' (c) - (\sigma (sr(c))- \sigma '(sr(c))  )    \big) \\
 &= \ord  \big(  \sigma (c) - \sigma (sr(c)) - ( \sigma ' (c) -\sigma ' (sr(c))) \big)\\
  &= \ord  \big( \sigma (c) -  sr(\sigma(c)) - ( \sigma ' (c) - sr(\sigma '(c))) \big)\\
 &>-r ~~~~
 \end{align*}
 because $ \ord \big(\sigma (c) - sr(\sigma(c))\big) >-r $ and $ \ord \big(\sigma '(c) - sr(\sigma '(c))\big) >-r$,  by definition of standard representatives.
 Let us prove (2) $\Rightarrow$ (3). If (2) holds, then using Proposition \ref{propforsr} we have $\sigma ( sr(c) )\neq \sigma ' ( sr(c))$ and $\ord ( B ) = \ord ( \sigma (sr (c )) - \sigma ' ( sr (c)) ) =-r $, because  $\sigma (sr (c ))$ and $\sigma ' ( sr (c))$ are both in $C_{\overline{E}}$. So $\ord(A) \geq -r , \ord (B) =-r , \ord ( A-B) >-r$, this implies $\ord (A)=-r$. Let us prove (3)$ \Rightarrow$ (2). We assume that (3) holds and we want to prove that $sr(c)$ generates $E/F$. It is enough to show that $\sigma (sr(c)) \neq \sigma ' (sr(c))$ for any $\sigma , \sigma '$ as before. 
 By hypothesis $\ord (A) =-r$; because of $\ord (B) \geq -r$ and $\ord (A-B) >-r$, we deduce $\ord (B) =-r$, in particular $\sigma (sr(c)) \neq \sigma ' (sr(c))$ as required.
 This finishes the proof.
\end{prf}

\begin{lem}
\label{preservemin}
Let $E/F$ be a finite, tamely ramified extension, and let $c, c' \in E^\times$ such that $c^{-1}c' \in 1+\pfra_{E}$.  
Then $c$ is minimal relative to $E/F$ if and only if $c'$ is minimal relative to $E/F$.  
\end{lem}

\begin{prf}
It suffices to show that if $c$ is minimal relative to $E/F$, then $c'$ is also minimal relative to $E/F$.  
Suppose $c$ is minimal relative to $E/F$.  
In particular, $E$ is generated by $c$ over $F$.  
Then we have $E=F[\sr(c)]$ by Proposition \ref{srformin}.  
Since $\sr(c) \in c(1+\pfra_{E}) = c'(1+\pfra_{E})$, we have $\sr(c') = \sr(c)$ by Proposition \label{propforsr} (1).  
If $E$ is also generated by $c'$ over $F$, then we can apply Proposition \ref{srformin} and $c'$ is minimal relative to $E/F$.  
Thus it is enough to show $E=F[c']$.  

We put $\Hom_{F}(E, \bar{F}) = \{ \tau_{1}, \ldots, \tau_{[E:F]} \}$.  
We have $\tau_{i} \neq \tau_{j}$ for distinct $i,j \in \{ 1, \ldots, [E:F] \}$ as $E/F$ is separable.  
Since $E=F[\sr(c)]$, if $i \neq j$ we have $\tau_{i}(\sr(c)) \neq \tau_{j}(\sr(c))$ and $\ord \left( \tau_{i}(\sr(c'))-\tau_{j}(\sr(c')) \right) = \ord(c')$ by Proposition \ref{propforsr} (4).  
On the other hand, since $\ord( \sr(c') -c' ) > \ord(c')$ by Proposition \ref{propforsr} we have
\[
	\ord \left( \tau_{i}(\sr(c')-c') \right) = \ord \left( \sr(c')-c' \right) > \ord(c').  
\]
For $i \neq j$, we obtain
\begin{eqnarray*}
	& & \ord ( \tau_{i}(c') - \tau_{j}(c') ) \\
	& = & \ord \Bigl( \bigl( \tau_{i}(\sr(c')) - \tau_{j}(\sr(c')) \bigr) - \bigl( \tau_{i}(\sr(c')-c') \bigr) + \bigl( \tau_{j}(\sr(c')-c') \bigr) \Bigr),  
\end{eqnarray*}
and then 
\[
	\ord ( \tau_{i}(c') - \tau_{j}(c') ) = \ord \left( \tau_{i}(\sr(c'))-\tau_{j}(\sr(c')) \right) = \ord(c') \in \R.  
\]
In particular, we have $\tau_{i}(c') \neq \tau_{j}(c')$.  
Since $\Hom_{F}(E, \bar{F}) = \{ \tau_{1}, \ldots, \tau_{[E:F]} \}$, the element $c'$ generates $E$ over $F$.  
\end{prf}

\subsection{Concrete description of \textbf{GE1} for $G$}
\label{defofXc}

Let $E'/E/F$ be a tamely ramified field extension in $A$.  
We put 
\[
	H=\Res_{E/F} \underline{\Aut}_{D \otimes _{F} E}(V), \, H'=\Res_{E'/F} \underline{\Aut}_{D \otimes _{F} E'}(V).  
\]
Then $(H', H, G)$ is a tame twisted Levi sequence by Corollary \ref{findTTLS}.  
And also, we have a natural isomorphism $\Lie(H') \cong \End_{D \otimes _{F} E'}(V)$.  
For $c \in \End_{D \otimes _{F} E'}(V)$, we can define $X_{c}^{*} \in \Lie^{*}(H')$ as
\[
X_{c}^{*}(z)= \Tr_{E'/F} \circ \Trd_{\End_{D \otimes_{F} E'}(V)/E'}(cz), 
\]
for $z \in \Lie(H') \cong \End_{D \otimes_{F} E'} (V)$, where $cz$ is a product of $c$ and $z$ as elements in $\End_{D \otimes _{F} E'}(V)$.  
Since $E'/F$ is separable, $\Tr_{E'/F}$ is surjective and there exists $e' \in E'$ such that $\Tr_{E'/F}(e') \neq 0$.  
Here, suppose $c \neq 0$.  
Since the map $(c, z) \mapsto \Trd_{\End_{D \otimes_{F} E'}(V)/E'}(cz)$ is a non-degenerate bilinear form on $\End_{D \otimes _{F} E'}(V)$, there exists $z \in \End_{D \otimes _{F} E'}(V)$ such that $\Trd_{\End_{D \otimes_{F} E'}(V)/E'}(cz) = e'$.  
In this case, we have $X_{c}^{*}(z) \neq 0$.  
Then, the map $c \mapsto X_{c}^{*}$ gives an isomorphism
\[
\Lie(\Res_{E'/F} \underline{\Aut}_{D \otimes _{F} E'}(V)) \cong \Lie^{*}(\Res_{E'/F} \underline{\Aut}_{D \otimes_{F} E'}(V)).  
\]
Since $\Trd_{A/F} |_{\End_{D \otimes_{F} E'}(V)} = \Tr_{E'/F} \circ \Trd_{\End_{D \otimes _{F} E'}(V)/E'}$, we also have
\[
X_{c}^{*}(z) = \Trd_{A/F}(cz).  
\]
For any $h \in H'(F)$ and $z \in \Lie(H')$, we have
\[
X_{c}^{*}(hzh^{-1})=\Trd_{A/F}(chzh^{-1})=\Trd_{A/F}(h^{-1}chz)=X_{h^{-1}ch}^{*}(z).  
\]
Then the linear form $X_{c}^{*}$ is invariant under $H'(F)$-conjugation if and only if $c = h^{-1}ch$ for any $h \in H'(F)=\Aut_{D \otimes_{F}E'}(V)$, that is, $c \in \Cent \left( \End_{D \otimes_{F} E'}(V) \right) = E'$.  Therefore, $X_{c}^{*}$ belongs to $(\underline{\Lie^{*}} (H'))^{H'} ( {F} )$.  

Let $c \in E'^{\times}$.  
We denote by $X_{c, \bar{F}}^{*}$ the image of $X_{c}^{*}$ in $(\underline{\Lie^{*}} (H')) ( \bar{F} )$.

To describe $X_{c,\bar{F}}^{*}(H_{\alpha})$ concretely, we use the notations in \S \ref{TTLS}.  

\begin{prop}
\label{concreteGE1}
Let $c \in E'^{\times}$ and $\alpha=\alpha_{(i',j',k'),(i'',j'',k'')} \in \Phi(G, T; \bar{F})$.  
Then we have $X_{c,\bar{F}}^{*}(H_{\alpha})=\sigma_{i',j'}(c)-\sigma_{i'',j''}(c)$.  
\end{prop}

\begin{prf}
Let $z=\sum_{i}z_{i} \otimes_{F} a_{i} \in \Lie^{*} (G) \otimes_{F} \bar{F} \cong \Lie \left( G \times_{F} \bar{F} \right)$.  
Then we have 
\begin{eqnarray*}
	X_{c,\bar{F}}^{*}(z) & = & \sum_{i} \Trd_{A/F}(cz_{i}) \otimes_{F} a_{i} = \sum_{i} \Tr_{A \otimes_{F} \bar{F} / \bar{F}} (cz_{i} \otimes_{F} a_{i}) \\
	& = & \Tr_{A \otimes_{F} \bar{F}/\bar{F}} \left( (c \otimes_{F} 1) \sum_{i}z_{i} \otimes_{F} a_{i} \right) \\
	& = & \Tr_{. \left( \End_{D \otimes \bar{F}}(\bigoplus_{i,j,k}V_{i,j,k}) \right) /\bar{F}} (m_{c,\bar{F}} z),  
\end{eqnarray*}
where $\End_{D \otimes \bar{F}}(\bigoplus_{i,j,k}V_{i,j,k}) \cong \M_{|I_{1}| \times |I_{2}| \times |I_{3}|}\left( \End_{D \otimes \bar{F}} (V \otimes_{L} \bar{F}) \right) \cong \M_{[L:F]}(\bar{F})$.  
Then, to calculate $\Tr_{A \otimes \bar{F}/\bar{F}}(m_{c,\bar{F}}H_{\alpha})$ we consider the value $m_{c,\bar{F}} \circ H_{\alpha} (v_{i,j,k})$ for some $v_{i,j,k} \in V_{i,j,k} \setminus \{ 0 \}$.  
By construction of $H_{\alpha}$ and Proposition \ref{Levidecom} (iii), we obtain
\[
	m_{c,\bar{F}} \circ H_{\alpha} (v_{i,j,k}) = \begin{cases}
		v_{i',j',k'} \cdot \sigma_{i',j',k'}(c) & \left( (i,j,k) = (i',j',k') \right), \\
		v_{i'',j'',k''} \cdot (-\sigma_{i'',j'',k''}(c)) & \left( (i,j,k) = (i'',j'',k'') \right), \\
		0 & otherwise.  
		\end{cases}
\]
Then we have $X_{c}^{*}(H_{\alpha})=\Tr_{A \otimes \bar{F}/\bar{F}}(m_{c,\bar{F}}H_{\alpha})=\sigma_{i',j',k'}(c)-\sigma_{i'',j'',k''}(c)$.  
Since $c \in E'$, we have $\sigma_{i',j',k'}(c)=\sigma_{i',j'}(c)$ and $\sigma_{i'',j'',k''}(c) = \sigma_{i'',j''}(c)$, which complete the proof.  
\end{prf}

\subsection{General elements of $G$}

\begin{lem}
\label{lem_for_genelem}
Suppose $\Afra$ is a principal hereditary $\ofra_{F}$-order in $A$ with its radical $\Pfra$ and $e=e(\Afra | \ofra_{F})$.  
Therefore, we have $\Trd_{A/F}(\Pfra^{n})=\pfra_{F}^{\lceil n/e \rceil}$ for any $n \in \Z$.  
\end{lem}

\begin{prf}
Since $\Trd_{A/F}$ is invariant by $A^{\times}$-conjugation, we may assume 
\[
	\Afra = \left(
	\begin{array}{ccc}
		\M_{md/e}(\ofra_{D}) & \cdots & \M_{md/e}(\ofra_{D}) \\
		\vdots & \ddots & \vdots \\
		\M_{md/e}(\pfra_{D}) & \cdots & \M_{md/e}(\ofra_{D})
	\end{array}
	\right).  	
\]

First we show the lemma when $A$ is split.  
In this case, we have 
\[
	\Pfra^{n} = \left(
	\begin{array}{ccc}
		\M_{N/e} \left( \pfra_{F}^{\lceil n/e \rceil} \right) & & * \\
		 & \ddots & \\
		* & & \M_{N/e} \left( \pfra_{F}^{\lceil n/e \rceil} \right)
	\end{array}
	\right), 
\]
where each block in * is contained in $\M_{N/e} (F)$.  
Then $\Tr_{A/F} \left( \Pfra^{n} \right) \subset \pfra_{F}^{\lceil n/e \rceil}$.  
To obtain the converse inclusion, let $b \in \pfra_{F}^{\lceil n/e \rceil}$.  
Let $a$ be an element in $A$ with the $(1,1)$-entry $b$, and other entries 0.  
Then $a \in \Pfra^{n}$ and $\Tr_{A/F} (a) = b$.  

In general case, we take a maximal unramified extension $E/F$ in $D$.  
Then $A \otimes_{F} E$ is split, and the subring $\Afra_{E}:=\Afra \otimes_{\ofra_{F}} \ofra_{E}$ in $A \otimes_{F} E$ is a hereditary $\ofra_{E}$-order with $e(\Afra_{E}|\ofra_{E}) = e(\Afra|\ofra_{F})=e$.  
Let $\Pfra_{E}$ be the radical of $\Afra_{E}$.  
Then ${\Pfra_{E}}^{n} = \Pfra^{n} \otimes_{\ofra_{F}} \ofra_{E}$ and $\Tr_{A \otimes_{F}E/E}({\Pfra_{E}}^{n})=\pfra_{E}^{\lceil n/e \rceil}$ by the split case.  
Since $\Trd_{A/F}(A) = \Tr_{A \otimes_{F} E/E} \left(A \otimes_{F} 1 \right) = F$, we have 
\[
	\Trd_{A/F}(\Pfra^{n}) \subset \pfra_{E}^{\lceil n/e \rceil} \cap F = \pfra_{F}^{\lceil n/e \rceil}, 
\]
where the last equality follows from the assumption $E/F$ is unramified.  

To obtain the converse inclusion, let $b \in \pfra_{F}^{\lceil n/e \rceil}$.  
Since $E/F$ is unramified, we have $\Tr_{E/F} \left(\pfra_{E}^{\lceil n/e \rceil} \right) = \pfra_{F}^{\lceil n/e \rceil}$, and there exists $b' \in \pfra_{E}^{\lceil n/e \rceil}$ such that $\Tr_{E/F}(b') = b$.  
Let $a$ be an element in $A \cong \M_{m}(D)$ with the $(1,1)$-entry $b'$, and other entries 0.  
Then $a \in \Cent_{A}(E) \cong \M_{m}(E)$, and
\[
	\Trd_{A/F}(a) = \Tr_{E/F} \circ \Tr_{\Cent_{A}(E)/E}(a) = \Tr_{E/F}(b') = b.  
\]
Therefore it suffice to check $a \in \Pfra^{n}$.  
We have 
\[
	\Pfra^{n} = \left(
	\begin{array}{ccc}
		\M_{md/e} \left( \pfra_{D}^{\lceil nd/e \rceil} \right) & & * \\
		 & \ddots & \\
		** & & \M_{N/e} \left( \pfra_{D}^{\lceil nd/e \rceil} \right)
	\end{array}
	\right), 
\]
Then $a \in \Pfra^{n}$ if and only if $b' \in \pfra_{D}^{\lceil nd/e \rceil}$.  
However, $b' \in \pfra_{E}^{\lceil n/e \rceil} \subset \pfra_{D}^{\lceil n/e \rceil d} \subset \pfra_{D}^{\lceil nd/e \rceil}$, where $\lceil n/e \rceil d \geq \lceil nd/e \rceil$ since $nd/e \leq \lceil n/e \rceil d \in \Z$.  
\end{prf}

\begin{prop}
\label{genelem}
Let $c \in E'^{\times}$.  
We put $r = -\ord(c)=\ord(c^{-1})$.  
\begin{enumerate}
\item $X_{c}^{*} \in \Lie^{*}(H')_{x,-r}$ for any $x \in \Bscr^{E}(H', F) $.  
\item $X_{c}^{*}$ is $H$-generic of depth $r$ if and only if $c$ is minimal relative to $E'/E$.  
\end{enumerate}
\end{prop}

\begin{prf}
We show (1). 
By \cite[Lemma 2.3]{Fide21}, it is enough to show $X_{c}^{*} \in \Lie^{*}(H')_{x,-r}$ for some $x \in \mathscr{B}^{E}(H',F)$.  
Then, we may assume $[x]$ is a vertex in $\mathscr{B}^{R}(H',F)$.  
By Proposition \ref{beingvertex}, the corresponding hereditary $\ofra_{E'}$-order $\Bfra'$ in $B' := \End_{D \otimes_{F} E'}(V) \cong \M_{n_{E'}}(D_{E'})$ is maximal, where $D_{E'}$ is a central division $E'$-algebra and $n_{E'} \in \N$.   
Therefore, we may also assume $\Bfra' = \M_{n_{E'}}(\ofra_{D_{E'}})$ by taking an isomorphism $B' \cong \M_{n_{E'}}(D_{E'})$ of $D_{E'}$-modules.  

We will describe $\mathfrak{h}'(F)_{x,r+}=\Lie^{*}(H')_{x,r+}$ concretely.  
Let $\Qfra'$ be the radical of $\Bfra'$.  
By Proposition \ref{lemforfindc} (1), we have $\Lie(H')_{x,s}=\mathfrak{h}'(F)_{x,s}=\Qfra'^{\lceil se(\Bfra' | \ofra_{E'}) e(E'/F) \rceil}$ for $s \in \R$.  
Here, we have $re(\Bfra')e(E'/F)=\ord(c^{-1})e(\Bfra' | \ofra_{E'})e(E'/F)=v_{E'}(c^{-1})e(\Bfra' | \ofra_{E'})=v_{\Bfra'}(c^{-1}) \in \Z$.  
Then, for any sufficiently small $\varepsilon>0$, we have $\Qfra'^{\lceil (r+\varepsilon)e(\Bfra' | \ofra_{E'}) e(E'/F) \rceil}=(\Qfra')^{v_{\Bfra'}(c^{-1}) +\lceil \varepsilon e(\Bfra' | \ofra_{E'}) e(E'/F) \rceil}=(\Qfra')^{v_{\Bfra'}(c^{-1})+1}=c^{-1}\Qfra'$.  
Therefore, we obtain $\mathfrak{h}'(F)_{x,r+}=c^{-1}\Qfra'$.  

By the definition of $\Lie^{*}(H')_{x,-r}$, to show (1) it is enough to show that $X_{c}^{*}(\mathfrak{h}'(F)_{x,r+}) \subset \pfra_{F}$.  
Here, we have $X_{c}^{*}(\mathfrak{h}'(F)_{x,r+})=\Tr_{E'/F} \circ \Trd_{B'/E'}(c \cdot c^{-1}\Qfra')=\Tr_{E'/F} \circ \Trd_{B'/E'}(\Qfra')$.  
Since $[x]$ is a vertex, the hereditary order $\Bfra'$ is principal and we have $\Trd_{B'/E'}(\Qfra')=\pfra_{E'}$ by Lemma \ref{lem_for_genelem}.  
Moreover, $\Tr_{E'/F}(\pfra_{E'})=\pfra_{F}$ since $E'/F$ is tame.  
Therefore, we obtain $X_{c}^{*}(\mathfrak{h}'(F)_{x,-r+})=\Tr_{E'/F} \circ \Trd_{B'/E'}(\Qfra')=\Tr_{E'/F}(\pfra_{E'})=\pfra_{F}$ and complete the proof of (1).  

To show (2), first suppose $X_{c}^{*}$ is $H$-generic of depth $r$.  

We will show $E'=E[c]$.  
We fix an embedding $\sigma_{i}:E \to \bar{F}$.  
Then we have $\Hom_{E}(E', \bar{F}) = \{ \sigma_{i,j} \mid j \in I_{2} \}$.  
Since $E/F$ is separable, to show $E'=E[c]$ it suffices to show $\sigma_{i,j}(c) \neq \sigma_{i,j'}(c)$ for any distinct $j,j' \in I_{2}$.  
We fix $k \in I_{3}$ and we put $\alpha = \alpha_{(i,j,k),(i,j',k)} \in \Phi(G, T; \bar{F})$.  
Then $\alpha \in \Phi(H, T; \bar{F}) \setminus \Phi(H', T; \bar{F})$.  
Since $X_{c}^{*}$ is $H$-generic of depth $r$, we have $-r=\ord \left( X_{c,\bar{F}}^{*}(H_{\alpha}) \right)=\ord \left( \sigma_{i,j}(c)-\sigma_{i,j'}(c) \right)$, where the last equality follows from Proposition \ref{concreteGE1}.  
In particular, we have $\ord \left( \sigma_{i,j}(c)-\sigma_{i,j'}(c) \right) \in \R$.  
Then $\sigma_{i,j}(c)-\sigma_{i,j'}(c) \neq 0$, that is, $\sigma_{i,j}(c) \neq \sigma_{i,j'}(c)$.  
Therefore, we obtain $E'=E[c]$.  
Moreover, we already know $\ord \left( \sigma_{i,j}(c)-\sigma_{i,j'}(c) \right)=-r$.  
Then, by Proposition \ref{srformin}, the element $c$ in $E'$ is minimal.  

Conversely, suppose $c$ is minimal relative to $E'/E$.  
In particular, we have $E'=E[c]$.  
By Corollary \ref{GE1toGE2}, to show that $X_{c}^{*}$ is $H$-generic it suffices to check $X_{c}^{*}$ satisfies \textbf{GE1}.  
Let $\alpha = \alpha_{(i,j,k), (i',j',k')} \in \Phi(H, T; F) \setminus \Phi(H', T; F)$.  
Then we have $i=i'$ and $j \neq j'$.  
We equip $\bar{F}$ with $E$-structure via $\sigma_{i}$.  
Then we have $\Hom_{E}(E', \bar{F}) = \{ \sigma_{i,j} \mid j \in I_{2} \}$.  
Since $c$ is minimal over $E$ and $E'=E[c]$, we have $\ord \left( \sigma_{i,j}(c) - \sigma_{i,j'}(c) \right) = -r$ by Proposition \ref{srformin}.  
Therefore $X_{c}^{*}(H_{\alpha}) = \ord \left( \sigma_{i,j}(c) - \sigma_{i,j'}(c) \right) = -r$, which implies $X_{c}^{*}$ is $H$-generic of depth $r$.  
\end{prf}

\subsection{General characters of $G$}

In this subsection, we discuss smooth characters of $G$.  
Let $\chi$ be a smooth character of $G$.  
Let $\Afra$ be a principal hereditary $\ofra_F$-order. 
The goal of this subsection is to prove the following proposition.
\begin{prop}
\label{propforgench}
 
Suppose $\chi$ is trivial on $\U^{n+1}(\Afra)$, but not trivial on $\U^{n}(\Afra)$ for some $n \in \Z_{ \geq 0}$.  
Then there exists $c \in F$ such that $v_{F}(c)=-n/e(\Afra|\ofra_{F})$ and
\[
	\chi |_{\U^{\lfloor n/2 \rfloor +1}(\Afra)} (1+y) = \psi \circ \Trd_{A/F}(cy)
\]
for $y \in \Pfra^{\lfloor n/2 \rfloor +1}$.  
\end{prop}

To prove Proposition \ref{propforgench}, we need some preliminary.  
We put $e=e(\Afra|\ofra_{F})$.  
If Proposition \ref{propforgench} holds for some $\chi, n$ and $\Afra$, then it also holds for any $G$-conjugation of $\Afra$ and the same $\chi, n$ as above.  
Therefore we may assume
\[
	\Afra = \left(
	\begin{array}{ccc}
		\M_{md/e}(\ofra_{D}) & \cdots & \M_{md/e}(\ofra_{D}) \\
		\vdots & \ddots & \vdots \\
		\M_{md/e}(\pfra_{D}) & \cdots & \M_{md/e}(\ofra_{D})
	\end{array}
	\right).  	
\]

\begin{lem}
\label{exttrivpart}
Suppose $\chi$ is trivial on $\U^{e(n+1)}(\Afra)$.  
Then $\chi$ is also trivial on $\U^{en+1}(\Afra)$.  
\end{lem}

\begin{prf}
Since $\chi$ factors through $\Nrd_{A/F}$, it is enough to show that
\[
\Nrd_{A/F} \left( \U^{e(n+1)}(\Afra) \right) = \Nrd_{A/F} \left( \U^{en+1}(\Afra) \right).  
\]
We can deduce it from the following lemma.  
\end{prf}

\begin{lem}
We have $\Nrd_{A/F}(1+\Pfra^{n})=1+\pfra_{F}^{\lceil n/e \rceil}$.  
\end{lem}

\begin{prf}
First we show the lemma when $A$ is split.  
In this case, we have 
\[
	1+\Pfra^{n} = \left(
	\begin{array}{ccc}
		1+\M_{N/e} \left( \pfra_{F}^{\lceil n/e \rceil} \right) & & * \\
		 & \ddots & \\
		** & & 1+\M_{N/e} \left( \pfra_{F}^{\lceil n/e \rceil} \right)
	\end{array}
	\right), 
\]
where each block in $**$ is equal to $\M_{N/e} \left( \pfra_{F}^{\lceil n/e \rceil} \right)$ or $\M_{N/e} \left( \pfra_{F}^{\lceil n/e \rceil +1} \right)$.  
Then any element $a$ in $1+\Pfra^{n}$ are upper triangular modulo $\pfra_{F}^{\lceil n/e \rceil}$, and $\det_{A/F}(a)$ is 1 modulo $\pfra_{F}^{\lceil n/e \rceil}$, whence $\det_{A/F} \left( 1+\Pfra^{n} \right) \subset 1 + \pfra_{F}^{\lceil n/e \rceil}$.  
To obtain the converse inclusion, let $1+b \in 1+\pfra_{F}^{\lceil n/e \rceil}$.  
Let $a$ be an element in $A$ with the $(1,1)$-entry $b$, and other entries 0.  
Then $1+a \in 1+\Pfra^{n}$ and $\det_{A/F} (1+a) = 1+b$.  

In general case, we take a maximal unramified extension $E/F$ in $D$.  
Then $A \otimes_{F} E$ is split, and the subring $\Afra_{E}:=\Afra \otimes_{\ofra_{F}} \ofra_{E}$ in $A \otimes_{F} E$ is a hereditary $\ofra_{E}$-order with $e(\Afra_{E}|\ofra_{E}) = e(\Afra|\ofra_{F})=e$.  
Let $\Pfra_{E}$ be the radical of $\Afra_{E}$.  
Then $\Pfra_{E}^{n} = \Pfra^{n} \otimes_{\ofra_{F}} \ofra_{E}$ and $\det_{A \otimes_{F}E/E}(1+\Pfra_{E}^{n})=1+\pfra_{E}^{\lceil n/e \rceil}$ by the split case.  
Since $\Nrd_{A/F}(A^{\times}) = \det_{A \otimes_{F} E/E} \left((A \otimes_{F} 1)^{\times} \right) = F^{\times}$, we have 
\[
	\Nrd_{A/F}(1+\Pfra^{n}) \subset \left( 1+\pfra_{E}^{\lceil n/e \rceil} \right) \cap F^{\times} = 1+\pfra_{F}^{\lceil n/e \rceil}, 
\]
where the last equality follows from the assumption $E/F$ is unramified.  

To obtain the converse inclusion, let $1+b \in 1+\pfra_{F}^{\lceil n/e \rceil}$.  
Since $E/F$ is unramified, we have $\Nrm_{E/F} \left(1+\pfra_{E}^{\lceil n/e \rceil} \right) = 1+\pfra_{F}^{\lceil n/e \rceil}$, and there exists $b' \in \pfra_{E}^{\lceil n/e \rceil}$ such that $\Nrm_{E/F}(1+b') = 1+b$.  
Let $a$ be an element in $A \cong \M_{m}(D)$ with the $(1,1)$-entry $b'$, and other entries 0.  
Then $1+a \in \Cent_{A}(E) \cong \M_{m}(E)$, and
\[
	\Nrd_{A/F}(1+a) = \Nrm_{E/F} \circ \det{}_{\Cent_{A}(E)/E}(1+a) = \Nrm_{E/F}(1+b') = 1+b.  
\]
Therefore it suffice to check $a \in \Pfra^{n}$.  
We have 
\[
	\Pfra^{n} = \left(
	\begin{array}{ccc}
		\M_{md/e} \left( \pfra_{D}^{\lceil nd/e \rceil} \right) & & * \\
		 & \ddots & \\
		** & & \M_{N/e} \left( \pfra_{D}^{\lceil nd/e \rceil} \right)
	\end{array}
	\right), 
\]
Then $a \in \Pfra^{n}$ if and only if $b' \in \pfra_{D}^{\lceil nd/e \rceil}$.  
However, $b' \in \pfra_{E}^{\lceil n/e \rceil} \subset \pfra_{D}^{\lceil n/e \rceil d} \subset \pfra_{D}^{\lceil nd/e \rceil}$, where $\lceil n/e \rceil d \geq \lceil nd/e \rceil$ since $nd/e \leq \lceil n/e \rceil d \in \Z$.  
\end{prf}

\begin{prop}
\label{correctbyF}
Suppose $n>0$.  
Furthermore, assume $\chi$ is trivial on $\U^{en+1}(\Afra)$, but not on $\U^{en}(\Afra)$.  
Then there exists $c \in F$ with $v_{F}(c) = -n$ such that
\[
	\chi |_{\U^{en}(\Afra)}(1+y) = \psi \circ \Trd_{A/F}(cy)
\]
for $y \in \Pfra^{en}$.  
\end{prop}

\begin{prf}
We have $\U^{en}(\Afra)/\U^{en+1}(\Afra) \cong \Pfra^{en}/\Pfra^{en+1}$, we can regard any smooth character $\U^{en}(\Afra)/\U^{en+1}(\Afra)$ as a smooth character of $\Pfra^{en}/\Pfra^{en+1}$.  
For any smooth character $\phi$ of $\Pfra^{en}/\Pfra^{en+1}$, there exists $c_{0} \in \Pfra^{-en}$, unique up to modulo $\Pfra^{-en+1}$, such that $\phi(y) = \psi \circ \Trd_{A/F}(c_0 y)$ for any $y \in \Pfra^{en}$.  In particular, there exists $c_{0} \in \Pfra^{-en}$ such that $\chi(1+y) = \psi \circ \Trd_{A/F}(c_{0}y)$ for any $y \in \Pfra^{en}$.
Since $\chi$ is not trivial on $\U^{en}(\Afra)$, we have $c_{0} \notin \Pfra^{-en+1}$.  
By the uniqueness of $c_{0}$, it suffices to show that $c_{0} + \Pfra^{-en+1}$ contains some element $c$ in $F$ with $v_{F}(c)=-n$.  
  
Here, let $g \in \Kfra(\Afra)$ and $y \in \Pfra^{-en}$.  
Since $\chi$ is a character of $G$, we have $\chi(1+y) = \chi \left( g (1+y) g^{-1} \right)$.  
However, we have $g (1+y) g^{-1} = 1+gyg^{-1}$ and $gyg^{-1} \in \Pfra^{en}$ since $g \in \Kfra(\Afra)$.  
Then we obtain
\begin{eqnarray*}
	\chi \left( g (1+y) g^{-1} \right) & = & \chi \left( 1+gyg^{-1} \right) \\
	& = & \psi \circ \Trd_{A/F}(c_{0}gyg^{-1}) \\
	& = & \psi \circ \Trd_{A/F}(g^{-1}c_{0}gy).  
\end{eqnarray*}
Since $g^{-1}c_{0}g \in \Pfra^{-en}$, we have $c_{0} + \Pfra^{-en+1} = g^{-1}c_{0}g + \Pfra^{-en+1}$ by the uniqueness of $c_{0}$.  
We take $t \in F^{\times}$ such that $v_{F}(t) = n$.  
Then we have 
\[
	tc_{0} + \Pfra = t(c_{0} + \Pfra^{-en+1}) = tg^{-1}c_{0}g+t\Pfra^{-en+1} = g^{-1}(tc_{0})g + \Pfra
\]
for $g \in \Kfra(\Afra)$.  
If we put $c' = tc_{0}$, then $c', g^{-1}c'g \in \Afra$ and $c' + \Pfra = g^{-1}c'g + \Pfra$.  
Since $c_{0} \in \Pfra^{-en} \setminus \Pfra^{-en+1}$, we have $c' \in t \left( \Pfra^{en} \setminus \Pfra^{-en+1} \right) = \Afra \setminus \Pfra$.  
Therefore we obtain $\overline{c'} = \overline{g^{-1}c'g}$ for $g \in \Kfra(\Afra)$, where for $a \in \Afra$ we denote by $\overline{a}$ the image of $a$ in $\Afra/\Pfra$.  
By the form of $\Afra$, we have an isomorphism $\Afra/\Pfra \cong \M_{md/e}(k_{D})^{\frac{e}{d}}$ as
\begin{eqnarray*}
	\Afra/\Pfra \cong \left(
	\begin{array}{ccc}
		\M_{md/e}(k_{D}) & & \\
		 & \ddots & \\
		 & & \M_{md/e}(k_{D})
	\end{array}
	\right)
	& \ni & \left(
	\begin{array}{ccc}
		b_{1}  & \\
		 & \ddots & \\
		 & & b_{e/d}
	\end{array}
	\right) \\
	 & \mapsto & (b_{1}, \ldots, b_{e/d}) \in \prod_{i=1}^{e/d} \M_{md/e}(k_{D}).  
\end{eqnarray*}
Here, let $g \in \U(\Afra)$.  
Then $g \in \Afra$ and we have $\overline{c'} = \overline{g}^{-1} \cdot \overline{c'} \cdot \overline{g}$.  
Since $\U(\Afra) \to (\Afra/\Pfra)^{\times}$ is surjective, $\overline{c'} \in Z(\Afra/\Pfra) \cong Z \left( \prod_{i=1}^{e/d} \M_{md/e}(k_{D}) \right) = \prod_{i=1}^{e/d} k_{D}$.  
Let $(b_1, \ldots, b_{e/d})$ be the image of $\overline{c'}$ in $\prod_{i=1}^{e/d} k_{D}$.  

We take $g \in \Kfra(\Afra)$ with $v_{\Afra}(g) = -1$.  
Then $\overline{g^{-1}c'g} = (b_{2}, \ldots, b_{e/d}, \tau(b_{1}))$, where $\tau \in \Gal(k_{D}/k_{F})$ is a generator.  
Since $\overline{c'}=\overline{g^{-1}c'g}$, we have $b_{1} = b_{2} = \cdots = b_{e/d} = \tau(b_{1})$.  
Since $\tau$ is a generator of $\Gal(k_{D}/k_{F})$, the element $b_{1}$ is stabilized by $\Gal(k_{D}/k_{F})$, that is, $b_{1} \in k_{F}$.  
Therefore $\overline{c'} \in k_{F} \subset \prod_{i=1}^{e/d} k_{D}$.  
We take a lift $a$ of $b_{1}$ to $\ofra_{F}$.  
Since $\overline{c'} \neq 0$, we have $b_{1} \neq 0$ and then $a \in \ofra_{F}^{\times}$.  
Therefore $c=t^{-1}a$ satisfies the desired condition.  
\end{prf}

\begin{lem}
\label{existchar}
Let $c \in F^{\times}$ such that $v_{F}(c)=-n<0$.  
Then there exists a smooth character $\theta$ of $A^{\times}$ such that 
\[
	\theta | _{\U^{\lfloor en/2 \rfloor +1}(\Afra)}(1+y) = \psi \circ \Trd_{A/F}(cy) 
\]
for $y \in \Pfra^{en}$.  
\end{lem}

\begin{prf}
Since $v_{\Afra}(c)=-en$, the 4-tuple $[\Afra, en, 0, c]$ is a simple stratum.  
Then we can take an element $\theta$ in $\mathscr{C}(c, 0, \Afra)$, which is nonempty by Remark \ref{existsimpch}.  
Since $\theta$ is simple, $\theta |_{\Cent_{A}(F[c])^{\times} \cap H^{1}(c, \Afra)}$ can be extended to a character of $\Cent_{A}(F[c])^{\times}$.  
However, we have $F[c]=F$ and then $\Cent_{A}(F[c]) = A$.  
Therefore, $\theta$ can be extended to a character of $A^{\times}$.  
Since $\theta$ is simple and $c \in F$ is minimal over $F$, we have 
\[
	\theta |_{\U^{\lfloor en/2 \rfloor +1}(\Afra)}(1+y) = \psi_{c}(1+y) = \psi \circ \Trd_{A/F}(cy)
\]
for $y \in \Pfra^{\lfloor en/2 \rfloor}+1$.  
\end{prf}

Let us start the proof of Proposition \ref{propforgench}.  

\begin{prf}
First, if $n=0$, then $c=1$ satisfies the condition.  
Then we may assume $n>0$.  

If $n \notin e\Z$ and $\chi$ is trivial on $\U^{n+1}(\Afra)$, then $\chi$ is also trivial in $\U^{n}(\Afra)$ by Lemma \ref{exttrivpart}, which is a contradiction.  
Then $n \in e\Z$.  
Let $i_{0}$ be the smallest integer satisfying $\lfloor n/2 \rfloor +1 \leq ei_{0}$.  
Since $n \geq 1$, we have $i_{0} \geq 1$.  
For $i = i_{0}, \ldots, n/e$, we construct $c_{i} \in F$ and a character $\theta_{i}$ of $F^{\times}$ such that $\theta_{i} |_{\U^{\lfloor ei/2 \rfloor +1}(\Afra)} = \psi_{c_{i}}$ and $\chi \cdot \left( \prod_{j=i}^{n/e} \theta_{j} \right) ^{-1}$ is trivial on $\U^{ei}(\Afra)$, by downward induction.  

Let $i=n/e$.  
Since $\chi$ is not trivial on $\U^{n}(\Afra)$, then there exists $c_{n/e} \in F$ such that $v_{F}(c_{n/e}) = -n$ and $\chi$ is equal to $\psi_{c_{n/e}}$ by Proposition \ref{correctbyF}.  
Then we take a character $\theta_{n/e}$ of $F^{\times}$ as in Lemma \ref{existchar} for $c_{i}$, and $\chi \cdot \theta_{n/e}{}^{-1}$ is trivial on $\U^{n}(\Afra)=\U^{ei}(\Afra)$.  

Let $i_{0} \leq i<n/e$, and suppose we construct $c_{j}$ and $\theta_{i}$ for $i<j \leq n/e$.  
Since $\chi \cdot \left( \prod_{j=i+1}^{n/e} \theta_{j} \right) ^{-1}$ is trivial on $\U^{e(i+1)}(\Afra)$ by induction hypothesis, it is also trivial on $\U^{ei+1}(\Afra)$ by Lemma \ref{exttrivpart}.  
If $\chi \cdot \left( \prod_{j=i+1}^{n/e} \theta_{j} \right) ^{-1}$ is also trivial on $\U^{ei}(\Afra)$, then we put $c_{i}=0$ and $\theta_{i}=1$, whence $c_{i}$ and $\theta_{i}$ satisfy the condition.  
Otherwise, there exists $c_{i} \in F$ such that $v_{F}(c_{i})=-i$ and $\chi \cdot \left( \prod_{j=i+1}^{n/e} \theta_{j} \right) ^{-1} $ is equal to $\psi_{c_{i}}$ on $\U^{ei}(\Afra)$ by Proposition \ref{correctbyF}.  
Then we take a character $\theta_{i}$ of $F^{\times}$ as Lemma \ref{existchar} for $c_{i}$, and $\chi \cdot \left( \prod_{j=i}^{n/e} \theta_{j} \right) ^{-1}$ is trivial on $\U^{ei}(\Afra)$.  

Therefore $\chi \cdot \left( \prod_{i=i_{0}}^{n/e} \theta_{i} \right) ^{-1}$ is trivial on $\U^{ei_{0}}(\Afra)$.  
By Lemma \ref{exttrivpart}, it is also trivial on $\U^{e(i_{0}-1)+1}(\Afra)$.  
Since $i_{0}$ is the smallest integer satisfying $\lfloor n/2 \rfloor +1 \leq ei_{0}$, we have $e(i_{0}-1) < \lfloor n/2 \rfloor +1$, that is, $e(i_{0}-1)+1 \leq \lfloor n/2 \rfloor +1$.  
Then $\U^{\lfloor n/2 \rfloor +1}(\Afra) \subset \U^{e(i_{0}-1)+1}(\Afra)$, whence $\chi \cdot \left( \prod_{i=i_{0}}^{n/e} \theta_{i} \right) ^{-1}$ is trivial on $\U^{\lfloor n/2 \rfloor +1}(\Afra)$.  
This implies $\chi$ is equal to $\prod_{i=i_{0}}^{n/e} \theta_{i}$ on $\U^{ \lfloor n/2 \rfloor +1}(\Afra)$.  
For $i=i_{0}, \ldots, n/e$, we have $\lfloor ei/2 \rfloor +1 \leq \lfloor e(n/e)/2 \rfloor + 1 = \lfloor n/2 \rfloor +1$.  
By construction of $\theta_{i}$, the restriction of $\theta _i$ to $\U^{.\lfloor n/2 \rfloor +1 }(\Afra) \subset \U^{\lfloor ei/2 \rfloor +1}$ is equal to $\psi_{c_{i}}$.  
Then $\chi$ is equal to $\prod_{i=i_{0}}^{n/e} \psi_{c_{i}} = \psi_{\left( \sum_{i=i_{0}}^{n/e}c_{i} \right)}$ on $\U^{.\lfloor n/2 \rfloor +1 }(\Afra)$.  
We put $c=\sum_{i=i_{0}}^{n/e}c_{i}$.  
Since $v_{F}(c_{n/e})=-n$ and $v_{F}(c_{i}) \geq -i > -n$ for $i=i_{0}, \ldots, (n/e)-1$, we have $v_{F}(c)=-n$, which completes the proof.  
\end{prf}

\section{Some lemmas on maximal simple types}
\label{Depth0}

In this section, we show some lemmas which are used when we take the ``depth-zero" part of S\'echerre's datum or Yu's datum.  

\begin{lem}
\label{lem1forfindingdepth0}
Let $\Lambda, \Lambda'$ be extensions of a maximal simple type $(J, \lambda)$ to $\tilde{J}=\tilde{J}(\lambda)$.  
Then there exists a character $\chi$ of $\tilde{J}(\lambda)/J$ such that $\Lambda' \cong \chi \otimes \Lambda$.  
\end{lem}

\begin{prf}
Since $\Lambda|_{J} = \lambda = \Lambda'|_{J}$ is irreducible, we have $\Hom_{J}(\Lambda, \Lambda') \cong \C$.  
The group $\tilde{J}$ acts on $\Hom_{J}(\Lambda, \Lambda') \cong \C$ as the character $\chi$ of $\tilde{J}
$ by
\[
	g \cdot f := \Lambda'(g) \circ f \circ \Lambda(g^{-1}) = \chi(g)f
\]
for $g \in \tilde{J}$ and $f \in \Hom_{J}(\Lambda, \Lambda')$.  
Since $f$ is a $J$-homomorphism, $\chi$ is trivial on $J$.  
We take a nonzero element $f$ in $\Hom_{J}(\Lambda, \Lambda')$.  
Then for $g \in \tilde{J}$ we have 
\[
	\Lambda'(g) \circ f = f \circ \left( \chi(g)\Lambda(g) \right) = f \circ \left( \Lambda \otimes \chi (g) \right)
\]
and an $\tilde{J}$-isomorphism $\Lambda' \cong \Lambda \otimes \chi$.  
\end{prf}

If a maximal simple type $(J, \lambda)$ is associated with a simple stratum $[\Afra, n, 0, \beta]$, we put $\hat{J} = \hat{J}(\beta, \Afra)$ as in Definition \ref{defofJhat}.  

\begin{lem}
\label{lem2forfindingdepth0}
Let $(J=\U(\Afra), \lambda)$ be a simple type of depth zero, where $\Afra$ is a maximal hereditary $\ofra_{F}$-order in $A$, and let $(\tilde{J}, \Lambda)$ be a maximal extension of $(J, \lambda)$.  
We put $\rho=\Ind_{\tilde{J}}^{\hat{J}}\Lambda$.  
\begin{enumerate}
\item $\cInd_{\hat{J}}^{G} \rho$ is irreducible and supercuspidal.  
\item $\rho$ is irreducible.  
\item $\rho$ is trivial on $\U^{1}(\Afra)$.  
\end{enumerate}
\end{lem}

\begin{prf}
Since $(\tilde{J}, \Lambda)$ is a maximal extension of a simple type of depth zero, $\cInd_{\tilde{J}}^{G} \Lambda$ is irreducible and supercuspidal.  
However, by the transitivity of compact induction, we also have $\cInd_{\tilde{J}}^{G} \Lambda = \cInd_{\hat{J}}^{G} \cInd_{\tilde{J}}^{\hat{J}} \Lambda = \cInd_{\hat{J}}^{G} \rho$, which implies (1).  

Since $\cInd_{\hat{J}}^{G} \rho$ is irreducible, $\rho$ is also irreducible, that is, (2) holds.  

To show (3), we consider the Mackey decomposition of $\Res_{J}^{\hat{J}} \Ind_{\tilde{J}}^{\hat{J}} \Lambda$.  
We have 
\[
	\Res_{J}^{\hat{J}} \Ind_{\tilde{J}}^{\hat{J}} \Lambda = \bigoplus_{g \in J \backslash \hat{J} / \tilde{J}} \Ind_{J \cap {}^{g}\tilde{J} }^{J} \Res_{J \cap {}^{g}\tilde{J}}^{{}^{g} \tilde{J}} {}^{g} \Lambda = \bigoplus_{i=0}^{i-1} {}^{h^i} \lambda, 
\]
where $l = ( \hat{J} : \tilde{J} )$ and $h \in \hat{J}$ such that the image of $h$ in $\hat{J} / J \cong \Z$ is 1.  
Since $h\U^{1}(\Afra)h^{-1}=\U^{1}(\Afra)$, the representation ${}^{h^i} \lambda$ is trivial on $\U^{1}(\Afra)$ for $i=0, \ldots, l-1$.  
Therefore $\rho$ is also trivial on $\U^{1}(\Afra)$.  
\end{prf}

\begin{lem}
\label{inf_and_ind}
Let $[\Afra, n, 0, \beta]$ be a simple stratum with $\Bfra$ maximal.  
Let $\sigma^{0}$ be an irreducible cuspidal representation of $\U(\Bfra)/\U^{1}(\Bfra)$, and let $(\tilde{J}^{0}, \tilde{\sigma}^{0})$ be a maximal extension of $(\U(\Bfra), \sigma^{0})$ in $\Kfra(\Bfra)$.  
We put $\rho = \cInd_{\tilde{J}^{0}}^{\Kfra(\Bfra)} \tilde{\sigma}^{0}$.  
We denote by $\tilde{\sigma}$ the representation $\tilde{\sigma}^{0}$ as a representation of $\tilde{J}=\tilde{J}^{0}J^{1}(\beta, \Afra)$ via the isomorphism $\tilde{J}^{0}/\U^{1}(\Bfra) \cong \tilde{J}/J^{1}(\beta, \Afra)$.  
Then $\cInd_{\tilde{J}}^{\hat{J}(\beta, \Afra)} \tilde{\sigma}$ is the representation $\rho$ regarded as a representation of $\hat{J}=\hat{J}(\beta, \Afra)$ via $\Kfra(\Bfra)/\U^{1}(\Bfra) \cong \hat{J}(\beta, \Afra)/J^{1}(\beta, \Afra)$.  
\end{lem}

\begin{prf}
Since $(\U(\Bfra), \sigma^{0})$ is a simple type of $B^{\times}$ of depth zero, $\rho$ is trivial on $\U^{1}(\Bfra)$ by Lemma \ref{lem2forfindingdepth0} (3).  
Then we can regard $\rho$ as a $\hat{J}(\beta, \Afra)$-representation.  

Since $\rho=\cInd_{\tilde{J}^{0}}^{\Kfra(\Bfra)} \tilde{\sigma}^{0}$, the dimension of $\rho$ is equal to $(\Kfra(\Bfra) : \tilde{J}^{0}) \dim \tilde{\sigma}^{0}$.  
On the other hand, the dimension of $\cInd_{\tilde{J}}^{\hat{J}} \tilde{\sigma}$ is equal to $(\hat{J} : \tilde{J}) \dim \tilde{\sigma}$.  
Since $\tilde{\sigma}$ is an extension of $\tilde{\sigma}^{0}$, we have $\dim \tilde{\sigma}^{0} = \dim \tilde{\sigma}$.  
Moreover, we also have $\Kfra(\Bfra)/\tilde{J}^{0} \cong \hat{J}/\tilde{J}$ and $(\Kfra(\Bfra) : \tilde{J}^{0})=(\hat{J} : \tilde{J})$ as $\hat{J}=\Kfra(\Bfra)J(\beta, \Afra)=\Kfra(\Bfra)\tilde{J}$ and $\Kfra(\Bfra) \cap \tilde{J} = \tilde{J}^{0}$.  
Since $\rho$ is irreducible by \ref{lem2forfindingdepth0} (2), it is enough to show that there exists a nonzero $\hat{J}$-homomorphism $\rho \to \cInd_{\tilde{J}}^{\hat{J}} \tilde{\sigma}$.  

First, since $\hat{J}$ is compact modulo center in $G$ and $\tilde{J}^{0}$ contains the center of $G$, for any subgroups $J' \subset J''$ between $\hat{J}$ and $\tilde{J}^{0}$ we have $\Ind_{J'}^{J''} = \cInd_{J'}^{J''}$.   
By the Frobenius reciprocity, $\Hom_{\tilde{J}} \left( \Ind_{\tilde{J}}^{\hat{J}} \tilde{\sigma}, \tilde{\sigma} \right) \neq 0$.  
Restricting these representations to $\tilde{J}^{0}$, we have $\Hom_{\tilde{J}^{0}} \left( \Ind_{\tilde{J}}^{\hat{J}} \tilde{\sigma}, \tilde{\sigma}^{0} \right) \neq 0$.  
Using the Frobenius reciprocity, we have $\Hom_{\Kfra(\Bfra)} \left( \Ind_{\tilde{J}}^{\hat{J}} \tilde{\sigma}, \Ind_{\tilde{J}^{0}}^{\Kfra(\Bfra)} \tilde{\sigma}^{0} \right) \neq 0$.  
Since $\Kfra(\Bfra)$ is compact modulo center, every $\mathfrak{K}(\mathfrak{B})$-representation of finite length with a central character is semisimple and $\Hom_{\Kfra(\Bfra)} \left( \Ind_{\tilde{J}^{0}}^{\Kfra(\Bfra)} \tilde{\sigma}^{0}, \Ind_{\tilde{J}}^{\hat{J}} \tilde{\sigma} \right) \neq 0$.  

Here, since $J^{1}(\beta, \Afra)$ is normal in $\hat{J}$ and $\tilde{\sigma}$ is trivial on $J^{1}(\beta, \Afra)$, the restriction of $\Ind_{\tilde{J}}^{\hat{J}} \tilde{\sigma}$ to $J^{1}(\beta, \Afra)$ is also trivial.  
Then, if we extend $\Ind_{\tilde{J}^{0}}^{\Kfra(\Bfra)} \tilde{\sigma}^{0} = \rho$ to $\hat{J}=\Kfra(\Bfra)J^1(\beta, \Afra)$ as trivial on $J^{1}(\beta, \Afra)$, there exists a nonzero $\hat{J}$-homomorphism $\rho \to \Ind_{\tilde{J}}^{\hat{J}} \tilde{\sigma}$.  
\end{prf}

The following lemma guarantees the existence of extensions of $\beta$-extensions for simple characters.  

\begin{lem}
\label{exist_ext_of_b-ext}
Let $[\Afra, n, 0, \beta]$ be a simple stratum of $A$ with $\Bfra$ maximal.  
Let $\theta \in \mathscr{C}(\beta, 0, \Afra)$, and let $\kappa$ be a $\beta$-extension of the Heisenberg representation $\eta_{\theta}$ of $\theta$ to $J(\beta, \Afra)$.  
\begin{enumerate}
\item There exists an extension $\hat{\kappa}$ of $\kappa$ to $\hat{J}(\beta, \Afra)$.  
\item Let $\hat{\kappa}'$ be another extension of $\eta_{\theta}$ to $\hat{J}(\beta, \Afra)$.  
Then there exists a character $\chi$ of $\hat{J}(\beta, \Afra)/J^{1}(\beta, \Afra)$ such that $\hat{\kappa}' \cong \hat{\kappa} \otimes \chi$.  
\end{enumerate}
\end{lem}

\begin{prf}
We fix $g \in \Kfra(\Bfra)$ with $v_{\Bfra}(g) = 1$.  
Since $\Kfra(\Bfra) \subset B^{\times} \subset I_{G}(\kappa)$ and $\Kfra(\Bfra)$ normalizes $J(\beta, \Afra)$, we can take a $J(\beta, \Afra)$-isomorphism $f:{}^{g}\kappa \to \kappa$.  
The group $\hat{J}(\beta, \Afra)/J(\beta, \Afra)$ is a cyclic group generated by the image of $g$, and then we can define $\hat{\kappa}$ as
\[
	\hat{\kappa}(g^{l}u) = f^{l} \circ \kappa(u)
\]
for $l \in \Z$ and $u \in J(\beta, \Afra)$.  
It is enough to show $\hat{\kappa}$ is a group homomorphism.  
Let $g_{1}, g_{2} \in \hat{J}(\beta, \Afra)$.  
Then there exist $l_{1}, l_{2} \in \Z$ and $u_{1}, u_{2} \in J(\beta, \Afra)$ such that $g_{i} = g^{l_{i}}u_{i}$ for $i=1,2$.  
We have $g_{1}g_{2} = g^{l_{1}+l_{2}}(g^{-l_{2}}u_{1}g^{l_{2}})u_{2}$ with $g^{-l_{2}}u_{1}g^{l_{2}} \in J(\beta, \Afra)$.  
Therefore we obtain
\begin{eqnarray*}
	\hat{\kappa}(g_{1}g_{2}) & = & f^{l_{1}+l_{2}} \circ \kappa(g^{-l_{2}}u_{1}g^{l_{2}}) \circ \kappa(u_{2}) \\
	& = & f^{l_{1}} \circ \kappa(u_{1}) \circ f^{l_{2}} \circ \kappa(u_{2}) = \hat{\kappa}(g_{1}) \circ \hat{\kappa}(g_{2}), 
\end{eqnarray*}
whence (1) holds.  

Let $\hat{\kappa}'$ be another extension of $\eta_{\theta}$ to $\hat{J}(\beta, \Afra)$.  
Then we have $\Hom_{J^{1}(\beta, \Afra)} (\hat{\kappa}, \hat{\kappa}') = \Hom_{J^{1}(\beta, \Afra)} (\eta_{\theta}, \eta_{\theta}) \cong \C$.  
The group $\hat{J}(\beta, \Afra)$ acts on $\Hom_{J^{1}(\beta, \Afra)} (\hat{\kappa}, \hat{\kappa}') \cong \C$.  
Then as in the proof of Lemma \ref{lem1forfindingdepth0} we obtain $\chi$ and (2) also holds.  
\end{prf}

The following proposition is one of the key points to construct a Yu datum from a S\'echerre datum.  

\begin{prop}
\label{findingdepth0}
Let $(J, \lambda)$ be a maximal simple type associated to a simple stratum $[\Afra, n, 0, \beta]$.  
Let $\theta \in \mathscr{C}(\beta, 0, \Afra)$ be a subrepresentation in $\lambda$, and let $\eta_{\theta}$ be the Heisenberg representation of $\theta$.  
For any extension $\Lambda$ of $\lambda$ to $\tilde{J}$ and any extension $\hat{\kappa'}$ of $\eta_{\theta}$ to $\hat{J}$, there exists an irreducible $\Kfra(\Bfra)$-representation $\rho$ such that
\begin{enumerate}
\item $\rho |_{\U(\Bfra)}$ is trivial on $\U^1(\Bfra)$ and cuspidal as a representation of $\U(\Bfra)/\U^1(\Bfra)$, 
\item $\cInd _{\Kfra(\Bfra)}^{B^{\times}} \rho$ is irreducible and supercuspidal, and
\item regarding $\rho$ as a $\hat{J}$-representation via the isomorphism $\Kfra(\Bfra)/\U^{1}(\Bfra) \cong \hat{J}/J^1$, the representation $\hat{\kappa'} \otimes \rho$ is isomorphic to $\cInd_{\tilde{J}}^{\hat{J}} \Lambda$.  
\end{enumerate}
\end{prop}

\begin{prf}
Let $\lambda = \kappa \otimes \sigma$ be a decomposition as in Definition \ref{defofsimpletype}.  
We take an extension $\hat{\kappa}$ of $\kappa$ to $\hat{J}$, which exists by Lemma \ref{exist_ext_of_b-ext} (1).  
Then there exists a character $\chi_{1}$ of $\hat{J}/J^{1}(\beta, \Afra)$ such that $\hat{\kappa} \cong \hat{\kappa}' \otimes \chi_{1}$ by Lemma \ref{exist_ext_of_b-ext}.  
Let $\tilde{\sigma}$ be an extension of $\sigma$ to $\tilde{J}$.  
Then the $\tilde{J}$-representations $\Lambda$ and $\hat{\kappa}' \otimes \chi_{1} \otimes \tilde{\sigma}$ are extensions of $\lambda$.  
By Lemma \ref{lem1forfindingdepth0}, there exists a character $\chi_2$ of $\tilde{J}$ such that $\Lambda \cong \hat{\kappa}' \otimes \chi_1 \otimes \tilde{\sigma} \otimes \chi_2$.  
Since $\chi_2$ is trivial on $J$ and $\hat{J}/J \cong \Z$, we can extend $\chi_2$ to $\hat{J}$.  
Let $J'$ be a subgroup in $\Kfra(\Bfra)$ corresponding to $\tilde{J}$ via the isomorphism $\Kfra(\Bfra)/\U(\Bfra) \cong \hat{J}/J$.  
Then $(J', \tilde{\sigma} \otimes \chi_1 \chi_2)$ is a maximal extension of the depth zero simple type $(\U(\Bfra), \sigma)$.  
Therefore we obtain a $\Kfra(\Bfra)$-representation $\rho = \cInd_{J'}^{\Kfra(\Bfra)} (\tilde{\sigma} \otimes \chi_1 \chi_2)$.  
Regarding $\rho$ as a $\hat{J}$-representation, $\rho$ is equal to $\cInd_{\tilde{J}}^{\hat{J}} (\tilde{\sigma} \otimes \chi_1 \chi_2)$ by Lemma \ref{inf_and_ind}.  
Then we have 
\[
	\hat{\kappa'} \otimes \rho = \hat{\kappa'} \otimes \cInd_{\tilde{J}}^{\hat{J}} (\tilde{\sigma} \otimes \chi_1 \chi_2) \cong \cInd_{\tilde{J}}^{\hat{J}} (\hat{\kappa'} \otimes \tilde{\sigma} \otimes \chi_1 \chi_2) \cong \cInd_{\tilde{J}}^{\hat{J}} \Lambda.  
\]
Therefore $\rho$ satisfies the desired conditions by Lemma \ref{lem2forfindingdepth0}.  
\end{prf}

Conversely, the following proposition is used to construct S\'echerre data from Yu data.  

\begin{prop}
\label{depth0ofYu}
Let $(x, (G^{i}), (\rbf_{i}), (\Phibf_{i}), \rho)$ be a Yu datum of $G \cong \GL_{m}(D)$.  
\begin{enumerate}
\item Then $[x]$ is a vertex in $\Bscr^{R}(G^{0}, F)$.  
\item There exists a simple type $(G^{0}(F)_{x}, \sigma)$ of depth zero and a maximal extension $(\tilde{J}, \tilde{\sigma})$ of $(G^{0}(F)_{x}, \sigma)$ such that $\rho \cong \Ind_{\tilde{J}}^{G^{0}(F)_{[x]}} \tilde{\sigma}$.  
\end{enumerate}
\end{prop}

\begin{prf}
In the beginning, $G^{0}$ is a tame twisted Levi subgroup in $G$ with $Z(G^{0})/Z(G)$ anisotropic.  
Then there exists a tamely ramified field extension $E_{0}/F$ in $A \cong \M_{m}(D)$ such that $G^{0}(F)$ is the multiplicative group of $\Cent_{A}(E_{0})$.  
Since $\Cent_{A}(E_{0})$ is a central simple $E_{0}$-algebra, there exists $m_{E_{0}} \in \Z_{>0}$ and a division $E_{0}$-algebra $D_{E_{0}}$ such that $\Cent_{A}(E_{0}) \cong \M_{m_{E_{0}}}(D_{E_{0}})$.  

By our assumption, $\pi := \cInd_{G^{0}(F)_{[x]}}^{G^0(F)} \rho$ is an irreducible and supercuspidal representation of depth zero.  
Then there exists $y \in \Bscr^{E}(G^{0}, F)$ and an irreducible $G^{0}(F)_{y}$-representation $\sigma$ such that $[y]$ is a vertex and $(G^{0}(F)_{y}, \sigma)$ is a $[G^0(F), \pi]_{G^0(F)}$-type.  
Since vertices in $\Bscr^{R}(G^{0}, F)$ are permuted transitively by the action of $G^{0}(F)$, we may assume $G^{0}(F)_{y} \supset G^{0}(F)_{x}$.  

We show that $\Ind_{G^{0}(F)_{x}}^{G^{0}(F)_{y}} \Res_{G^{0}(F)_{x}}^{G^{0}(F)_{[x]}} \rho$ has a non-zero $G^{0}(F)_{y,0+}$-fixed part.  
Since $G^{0}(F)_{x} \cap G^{0}(F)_{y,0+} \subset G^{0}(F)_{x,0+}$, the representation $\rho$ is trivial on $G^{0}(F)_{x} \cap G^{0}(F)_{y,0+}$.  
Then $\Ind_{G^{0}(F)_{x} \cap G^{0}(F)_{y,0+}} ^{G^{0}(F)_{y,0+}} \Res_{G^{0}(F)_{x} \cap G^{0}(F)_{y,0+}} ^{G^{0}(F)_{[x]}} \rho$ has a non-zero $G^{0}(F)_{y,0+}$-fixed part by the Frobenius reciprocity.   
However, 
\[
	\Ind_{G^{0}(F)_{x} \cap G^{0}(F)_{y,0+}} ^{G^{0}(F)_{y,0+}} \Res_{G^{0}(F)_{x} \cap G^{0}(F)_{y,0+}} ^{G^{0}(F)_{[x]}} \rho \subset \Res_{G^{0}(F)_{y,0+}}^{G^{0}(F)_{y}} \Ind_{G^{0}(F)_{x}}^{G^{0}(F)_{y}} \Res_{G^{0}(F)_{x}}^{G^{0}(F)_{[x]}} \rho
\]
by the Mackey decomposition.  
Therefore $\Ind_{G^{0}(F)_{x}}^{G^{0}(F)_{y}} \Res_{G^{0}(F)_{x}}^{G^{0}(F)_{[x]}} \rho$ has a non-zero $G^{0}(F)_{y,0+}$-fixed part.  

Since $\Ind_{G^{0}(F)_{x}}^{G^{0}(F)_{y}} \Res_{G^{0}(F)_{x}}^{G^{0}(F)_{[x]}} \rho \subset \Res_{G^{0}(F)_{y}}^{G^0} \cInd_{G^{0}(F)_{[x]}}^{G^0} \rho = \Res_{G^{0}(F)_{y}}^{G^0} \pi$ by the Mackey decomposition, we also may assume $\sigma \subset \Ind_{G^{0}(F)_{x}}^{G^{0}(F)_{y}} \Res_{G^{0}(F)_{x}}^{G^{0}(F)_{[x]}} \rho$ by $G^{0}(F)_{[x]}$-conjugation if necessary by \cite[Theorem 5.5(ii)]{GSZ}.  
By the Frobenius reciprocity, $\Res_{G^{0}(F)_{x}}^{G^{0}(F)_{y}} \sigma$ is a subrepresentation of $\Res_{G^{0}(F)_{x}}^{G^{0}(F)_{[x]}} \rho$, which is trivial on $G^{0}(F)_{x,0+}$.  
Therefore, $\sigma$ has a non-zero $G^{0}(F)_{x,0}G^{0}(F)_{y,0+}$-fixed part.  
Since the image of $G^{0}(F)_{x,0}$ in $G^{0}(F)_{y}/G^{0}(F)_{y,0+}$ is a parabolic subgroup of $G^{0}(F)_{y}/G^{0}(F)_{y,0+}$ and $\sigma$ is cuspidal when we regard $\sigma$ as a $G^{0}(F)_{y}/G^{0}(F)_{y,0+}$-representation, we have $G^{0}(F)_{x,0}G^{0}(F)_{y,0+} = G^{0}(F)_{y}$, which implies $[x]=[y]$, that is, (1) holds.  

To show (2), let $(\tilde{J}, \tilde{\sigma})$ be the unique extension of $(G^{0}(F)_{x}, \sigma)$ such that $\pi \cong \cInd_{\tilde{J}}^{G} \tilde{\sigma}$.  

We show the $G^{0}(F)_{x,0+}$-fixed part in $\pi$ is contained in $\Ind_{\tilde{J}}^{G^{0}(F)_{[x]}} \tilde{\sigma}$.  
By the Mackey decomposition, we have 
\[
	\pi \cong \bigoplus_{g \in G^{0}(F)_{x,0+} \backslash G / \tilde{J} } \Ind_{G^{0}(F)_{x,0+} \cap {}^{g}\tilde{J}}^{G^{0}(F)_{x,0+}} \Res_{G^{0}(F)_{x,0+} \cap {}^{g}\tilde{J}}^{{}^{g}\tilde{J}} {}^{g} \tilde{\sigma}.  
\]
We put $\tau(g)=\Ind_{G^{0}(F)_{x,0+} \cap {}^{g}\tilde{J}}^{G^{0}(F)_{x,0+}} \Res_{G^{0}(F)_{x,0+} \cap {}^{g}\tilde{J}}^{{}^{g}\tilde{J}} {}^{g} \tilde{\sigma}$.  
Suppose $\tau(g)$ has a non-zero $G^{0}(F)_{x,0+}$-fixed part.  
Then $\Hom_{G^{0}(F)_{x,0+} \cap {}^{g}\tilde{J}}(\mathbf{1}, {}^{g}\tilde{\sigma}) \neq 0$ by the Frobenius reciprocity.  
Here, since $[x]$ is a vertex, $G^{0}(F)_{x}$ is a maximal compact open subgroup.  
Therefore we may assume $G^{0}(F)_{x}=\GL_{m_{E_{0}}}(\ofra_{D_{E_{0}}})$ by $G^{0}(F)$-conjugation if necessary.  
Then there exist $k,k' \in G^{0}(F)_{x}$ and a diagonal matrix $g'$ such that the $(i,i)$-coefficient of $g'$ is $\varpi_{D_{E_{0}}}^{a_i}$ with $a_1 \geq a_2 \geq \cdots \geq a_{m}$, and such that $g=kg'k'$.  
Since $G^{0}(F)_{x,0+}$ is normal in $G^{0}(F)_{x}$ and $G^{0}(F)_{x} \subset \tilde{J}$, the condition $\Hom_{G^{0}(F)_{x,0+} \cap {}^{g}\tilde{J}}(\mathbf{1}, {}^{g}\tilde{\sigma}) \neq 0$ holds if and only if $\Hom_{{}^{{(g')}^{-1}}G^{0}(F)_{x,0+} \cap \tilde{J}}(\mathbf{1}, \tilde{\sigma}) \neq 0$.  
Therefore $\sigma$ has a non-zero $G^{0}(F)_{x,0+}\left( {}^{{(g')}^{-1}}G^{0}(F)_{x,0+} \cap \tilde{J} \right)$-fixed part.  
If $a_i > a_{i+1}$ for some $i$, the image of ${}^{{(g')}^{-1}}G^{0}(F)_{x,0+} \cap \tilde{J}$ in $G^{0}(F)_{x}/G^{0}(F)_{x,0+}$ is a proper parabolic subgroup, which is a contradiction since $\sigma$ is cuspidal.  
Then $g' \in D_{E_{0}}^{\times} \subset G^{0}(F)_{[x]}$ and $g=kg'k' \in G^{0}(F)_{[x]}$.  
Therefore the $G^{0}_{x,0+}$-fixed part in $\pi$ is contained in $\bigoplus_{g \in G^{0}(F)_{x,0+} \backslash G^{0}(F)_{[x]} / \tilde{J} } \tau(g) = \cInd_{\tilde{J}}^{G^{0}(F)_{[x]}} \tilde{\sigma}$.  

Then we have $\rho \subset \Ind_{\tilde{J}}^{G^{0}(F)_{[x]}} \tilde{\sigma}$.  
Since these representations are irreducible, we obtain $\rho = \Ind_{\tilde{J}}^{G^{0}(F)_{[x]}} \tilde{\sigma}$.  
\end{prf}

\section{Factorization of tame simple characters}

\label{Factorization}
Let $[\Afra, n, 0, \beta]$ be a tame simple stratum of $A$.  
If $n=0$, suppose $\beta \in \ofra_{F}^{\times}$.  
By Proposition \ref{appfortame}, there exists a defining sequence $\left( [\Afra, n, r_i, \beta_i] \right)_{i=0, 1, \ldots, s}$ of $[\Afra, n, 0, \beta]$ such that
\begin{enumerate}
\item $F[\beta_{i}] \supsetneq F[\beta_{i+1}]$, 
\item $\beta_{i}-\beta_{i+1}$ is minimal over $F[\beta_{i+1}]$ and
\item $v_{\Afra}(\beta_{i}-\beta_{i+1})=k_{0}(\beta_{i}, \Afra)=-r_{i+1}$
\end{enumerate}
for $i=0, 1, \ldots, s-1$.  

We put $E_i=F[\beta_i]$.  
Let $B_{i}$ be the centralizer of $E_i$ in $A$.  
Let $c_{i}=\beta_{i}-\beta_{i+1}$ for $i=0, \ldots, s-1$ and let $c_{s}=\beta_{s}$.  

\begin{prop}
\label{factchar}
Let $0 \leq t < -k_{0}(\beta, \Afra)$.  
Let $\theta \in \mathscr{C}(\beta, t, \Afra)$.  
Then for $i=0, 1, \ldots, s$ there exists a smooth character $\phi_{i}$ of $E_{i}^{\times}$ such that we have $\theta = \prod_{i=0}^{s} \theta^{i}$, where the characters $\theta^{i}$ of $H^{t+1}(\beta, \Afra)$ are defined as in the following:  
\begin{enumerate}
\item $\theta^{i} | _{B_{i}^{\times} \cap H^{t+1}(\beta, \Afra)} = \phi_i \circ \Nrd_{B_i/E_i}$, and
\item $\theta^{i} | _{H^{t_{i}+1}(\beta, \Afra)} = \psi_{c_i}$, where $t_{i}=\max \left\{ t, \lfloor -v_{\Afra}(c_{i})/2 \rfloor \right\}$.  
\end{enumerate}
\end{prop}

\begin{prf}
We show this proposition by induction on the length $s$ of the defining sequence.  

First, suppose that $s=0$, that is, $\beta$ is minimal over $F$.  
We have $\theta = \theta^{0}$.  
Then it is enough to show that $\theta$ satisfies (1) and (2).  
Since $\theta$ is simple, $\theta | B_{0}^{\times} \cap H^{t+1}(\beta, \Afra)$ factors through $\Nrd_{B_{0}/E_{0}}$.  
Then there exists a character $\phi_{0}$ of $E_{0}^{\times}$ such that $\theta = \phi_{0} \circ \Nrd_{B_{0}/E_{0}}$, whence (1) holds.  
We have $v_{\Afra}(c_0)=v_{\Afra}(\beta)=-n$ and $t_{0}=\max \{ t, \lfloor -v_{\Afra}(c_{0})/2 \rfloor \} \geq \lfloor n/2 \rfloor$.  
Then we have $H^{t_{0}+1}(\beta, \Afra) \subset H^{\lfloor n/2 \rfloor +1}(\beta, \Afra)$.  
Since $\theta$ is simple, we have $\theta|_{H^{t_{0}+1}}(\beta, \Afra)=\psi_{\beta}=\psi_{c_{0}}$, whence (2) also holds.  

Next, suppose that $s>0$, that is, $\beta$ is not minimal over $F$.  
We put $t'=\max \{ t, \lfloor -k_{0}(\beta, \Afra)/2 \rfloor \}$.  
Since $k_{0}(\beta, \Afra)=v_{\Afra}(c_{0})$, we have $t' = \max \{ t, \lfloor -v_{\Afra}(c_{0})/2 \rfloor \}=t_{0}$.  
Since $\theta$ is simple, there exists $\theta' \in \mathscr{C}(\beta, t', \Afra)$ such that $\theta |_{H^{t'+1}(\beta, \Afra)} =  \psi_{c_0}\theta'$.  
By induction hypothesis, for $i=1, \ldots, s$, there exist a smooth character $\phi_{i}$ of $E_{i}^{\times}$ and a smooth character $\theta'^{i}$ of $H^{t'+1}(\beta_1, \Afra)=H^{t'+1}(\beta, \Afra)$ such that $\theta'^{i}|_{B_{i}^{\times} \cap H^{t'+1}(\beta, \Afra)}=\phi_{i} \circ \Nrd_{B_{i}/E_{i}}$ and $\theta'^{i}|_{H^{t'_{i}+1}(\beta, \Afra)}=\psi_{c_i}$, where $t'_{i}= \max \{ t', \lfloor -v_{\Afra}(c_{i})/2 \rfloor \}$.  
Here we have $r_1 < \ldots < r_{s} < n$, whence we obtain $-v_{\Afra}(c_{0}) < \ldots < -v_{\Afra}(c_{s-1}) < -v_{\Afra}(c_{s})$.  
Since $t'=\max \{ t, \lfloor -v_{\Afra}(c_{0})/2 \rfloor \}$ and $-v_{\Afra}(c_{0}) < -v_{\Afra}(c_{i})$, we have $t'_{i} = t_{i}$.  
We want to extend $\theta'^{i}$ to a character $\theta^{i}$ of $H^{t+1}(\beta, \Afra)$ as $\theta^{i} |_{B_{i}^{\times} \cap H^{t+1}(\beta, \Afra)} = \phi_{i} \circ \Nrd_{B_{i}/E_{i}}$.  
Suppose we obtain $\theta^{i}$ in such a way.  
Then $\theta^{i}$ satisfies (1) in Proposition by construction of $\theta^{i}$, and (2) in Proposition since $H^{t_{i}+1}(\beta, \Afra)=H^{t'_{i}+1}(\beta, \Afra)$ and $\theta^{i}|_{H^{t_{i}+1}(\beta, \Afra)}=\theta'^{i}|_{H^{t_{i}+1}(\beta, \Afra)}=\psi_{c_{i}}$.  

We have $B_{i}^{\times} \cap H^{t+1}(\beta, \Afra) \cap H^{t'+1}(\beta, \Afra) = B_{i}^{\times} \cap H^{t'+1}(\beta, \Afra)$, whence restrictions of $\theta'^{i}$ and $\phi_{i} \circ \Nrd_{B_{i}/E_{i}}$ to $B_{i}^{\times} \cap H^{t+1}(\beta, \Afra) \cap H^{t_{i}+1}(\beta, \Afra)$ are equal.  
Let $b_{1}, b_{2} \in B_{i}^{\times} \cap H^{t+1}(\beta, \Afra)$ and $h'_{1}, h'_{2} \in H^{t'+1}(\beta, \Afra)$ with $b_{1}h'_{1} = b_{2}h'_{2}$.  
Then $b_{1}^{-1}b_{2} = h'_{1}(h'_{2})^{-1} \in B_{i}^{\times} \cap H^{t+1}(\beta, \Afra) \cap H^{t_{i}+1}(\beta, \Afra)$ and $\phi_{i} \circ \Nrd_{B_{i}/E_{i}}(b_{1}^{-1}b_{2}) = \theta'^{i}(h'_{1}(h'_{2})^{-1})$.  
Therefore we also have 
\[
	\phi_{i} \circ \Nrd_{B_{i}/E_{i}}(b_{1})\theta'^{i}(h'_{1}) = \phi_{i} \circ \Nrd_{B_{i}/E_{i}}(b_{2}) \theta'^{i}(h'_{2}).  
\]
Then $\theta^{i}$ is well-defined as a map from $H^{t+1}(\beta, \Afra)$ to $\C^{\times}$.  

We show $\psi_{c_{i}}(b^{-1}hb)=\psi_{c_{i}}(h)$ for $b \in B_{i}^{\times} \cap H^{t+1}(\beta, \Afra)$ and $h \in H^{t_{i}+1}(\beta, \Afra)$.  
By definition of $\psi_{c_{i}}$, we have
\[
	\psi_{c_{i}}(b^{-1}hb) = \Trd_{A/F}(c_{i}(b^{-1}hb-1)).  
\]
Since $c_{i} \in E_{i}$ and $b \in B_{i} = \Cent_{A}(E_{i})$, we have 
\[
	c_{i}(b^{-1}hb-1)=c_{i}b^{-1}(h-1)b=b^{-1}c_{i}(h-1)b.  
\]
Therefore, we obtain
\[
	\psi_{c_{i}}(b^{-1}hb) = \Trd_{A/F}(b^{-1}c_{i}(h-1)b) = \Trd_{A/F}(c_{i}(h-1)) = \psi_{c_{i}}(h).  
\]

To show $\theta^{i}$ is a character, let $h_{1}, h_{2} \in H^{t+1}(\beta, \Afra)$.  
Then there exist $b_{1}, b_{2} \in B_{i}^{\times} \cap H^{t+1}(\beta, \Afra)$ and $h'_{1}, h'_{2} \in H^{t_{i}+1}(\beta, \Afra)$ such that $h_{1}=b_{1}h'_{1}$ and $h_{2}=b_{2}h'_{2}$.  
Therefore we have 
\begin{eqnarray*}
	\theta^{i}(h_{1}h_{2}) & = & \theta^{i}(b_{1}h'_{1} b_{2}h'_{2}) \\
	& = & \theta^{i}\left( (b_{1}b_{2}) (b_{2}^{-1}h'_{1}b_{2}h'_{2} \right) \\
	& = & (\phi_{i} \circ \Nrd_{B_{i}/E_{i}}) (b_{1}) (\phi_{i} \circ \Nrd_{B_{i}/E_{i}})(b_{2}) \psi_{c_{i}}(b_{2}^{-1}h'_{1}b_{2}) \psi_{c_{i}}(h'_{2}) \\
	& = &  (\phi_{i} \circ \Nrd_{B_{i}/E_{i}}) (b_{1}) \psi_{c_{i}}(h'_{1}) (\phi_{i} \circ \Nrd_{B_{i}/E_{i}})(b_{2}) \psi_{c_{i}}(h'_{2}) \\
	& = & \theta^{i}(b_{1}h'_{1}) \theta^{i}(b_{2}h'_{2}) = \theta^{i}(h_{1}) \theta^{i}(h_{2}).  
\end{eqnarray*}

We put $\theta^{0} = \theta \prod_{i=1}^{s}(\theta^{i})^{-1}$.  
To complete the proof, it is enough to show $\theta^{0}$ satisfies (1) and (2).  

To see (1), we show the restrictions of $\theta$ and $\theta^{i}$ ($i=1, \ldots, s$) to $B_{0}^{\times} \cap H^{t+1}(\beta, \Afra)$ factor through $\Nrd_{B_{0}/E_{0}}$.  
Since $\theta$ is simple, $\theta | _{B_{0}^{\times} \cap H^{t+1}(\beta, \Afra)}$ factors through $\Nrd_{B_{0}/E_{0}}$.  
We already have $\theta | _{B_{i}^{\times} \cap H^{t+1}(\beta, \Afra)} = \phi_{i} \circ \Nrd_{B_{i}/E_{i}}$.  
Since $B_{0}^{\times} \subset B_{i}^{\times}$, we have $\theta_{B_{0}^{\times} \cap H^{t+1}(\beta, \Afra)} = \phi_{i} \circ (\Nrd_{B_{i}/E_{i}} | _{B_{0}^{\times}})$.  
However, the equation $\Nrd_{B_{i}/E_{i}} | _{B_{0}^{\times}} = \mathrm{N}_{E_{0}/E_{i}} \circ \Nrd_{B_{0}/E_{0}}$ holds.  
Then $\theta^{i} |_{B_{0}^{\times} \cap H^{t+1}(\beta, \Afra)}$ factors through $\Nrd_{B_{0}/E_{0}}$.  
Therefore $\theta^{0} |_{B_{0}^{\times} \cap H^{t+1}(\beta, \Afra)}$ also factors through $\Nrd_{B_{0}/E_{0}}$, and there exists a character $\phi_{0}$ of $E_{0}^{\times}$ such that $\theta^{0} |_{B_{0}^{\times} \cap H^{t+1}(\beta, \Afra)} = \phi_{0} \circ \Nrd_{B_{0}/E_{0}}$.  

By restricting $\theta^{0} = \theta \prod_{i=1}^{s} (\theta^{i})^{-1}$ to $H^{t_{0}+1}(\beta, \Afra)=H^{t'+1}(\beta, \Afra)$, we have
\begin{eqnarray*}
	\theta^{0}|_{H^{t_{0}+1}(\beta, \Afra)} & = & (\theta|_{H^{t'+1}(\beta, \Afra)} \prod_{i=1}^{s} (\theta^{i}|_{H^{t'+1}(\beta, \Afra)})^{-1} \\
	& = & \psi_{c_{0}} \theta' \prod_{i=1}^{s} (\theta'^{i})^{-1} = \psi_{c_{0}} \theta' \theta'^{-1} = \psi_{c_{0}}.  
\end{eqnarray*}
Therefore, (2) also holds and complete the proof.  
\end{prf}

\section{Construction of a Yu datum from a S\'echerre datum}
\label{SStoYu}

Let $[\Afra, n, 0, \beta]$ be a tame simple stratum, $(J(\beta, \Afra), \lambda)$ be a maximal simple type with $[\Afra, n, 0, \beta]$, and let $(\tilde{J}(\lambda), \Lambda)$ be a maximal extension of $(J(\beta, \Afra), \lambda)$.  
We construct a Yu datum from the data of $[\Afra, n, 0, \beta]$ and $\lambda$.  

We put $\Pfra = \Pfra(\Afra)$.  
Let $([\Afra ,n, r_{i}, \beta_{i}])_{i=0}^{s}$, $E_{i}$, $B_{i}$ and $c_{i}$ be as in \S \ref{Factorization}.  
For $i=0, 1, \ldots, s$, we put $G^{i}=\Res_{E_{i}/F} \underline{\Aut}_{D \otimes E_{i}}(V)$ and $\rbf_{i}=-\ord(c_{i})$.  
If $\beta_s \in F$, we put $d=s$.  
If $\beta_s \notin F$, we put $d=s+1$, $G^{d}=G$ and $\rbf_{d}=\rbf_{s}$.  
Then $(G^{0}, \ldots, G^{d})$ is a tame twisted Levi sequence by Corollary \ref{findTTLS}.  
We also put $\rbf_{-1}=0$.  
For $i=-1, 0, 1, \ldots, d$, we put $\mathbf{s}_{i}=\rbf_{i}/2$.  

\begin{prop}
\label{compofTLS}
We fix a $G^{i-1}(F)$-equivalent and affine embedding 
\[
	\iota_{i}:\Bscr^{E}(G^{i-1}, F) \hookrightarrow \Bscr^{E}(G^{i}, F)
\]
for $i=1, \ldots, d$ and we put $\tilde{\iota}_{i}=\iota_{i} \circ \cdots \circ \iota_{1}$.  
We also put $\tilde{\iota}_{0}=\id_{\Bscr^{E}(G^{0}, F)}$.  
\begin{enumerate}
\item There exists $x \in \Bscr^{E}(G^{0}, F)$ such that $[x]$ is a vertex and
\begin{enumerate}
\item $G^{0}(F)_{[x]} = \Kfra(\Bfra_{0})$, 
\item $G^{0}(F)_{x} = B_{0}^{\times} \cap \U(\Afra) = \U(\Bfra_{0})$, 
\item $G^{0}(F)_{x,0+} = B_{0}^{\times} \cap \U^{1}(\Afra)$, 
\item $\gfra^{0}(F)_{x} = B_{0} \cap \Afra = \Bfra_{0}$, and
\item $\gfra^{0}(F)_{x,0+} = B_{0} \cap \Pfra$.  
\end{enumerate}
\item For $i=1, \ldots, d$, we have
\begin{enumerate}
\item $G^{i}(F)_{\tilde{\iota}_{i}(x), \mathbf{s}_{i-1} } = B_{i}^{\times} \cap \U^{\lfloor (-v_{\Afra}(c_{i-1})+1)/2 \rfloor}(\Afra)$, 
\item $G^{i}(F)_{\tilde{\iota}_{i}(x), \mathbf{s}_{i-1}+ } = B_{i}^{\times} \cap \U^{\lfloor -v_{\Afra}(c_{i-1})/2 \rfloor +1}(\Afra)$, 
\item $G^{i}(F)_{\tilde{\iota}_{i}(x), \rbf_{i-1} } = B_{i}^{\times} \cap \U^{-v_{\Afra}(c_{i-1})}(\Afra)$, 
\item $G^{i}(F)_{\tilde{\iota}_{i}(x), \rbf_{i-1}+ } = B_{i}^{\times} \cap \U^{-v_{\Afra}(c_{i-1})+1}(\Afra)$, 
\item $\gfra^{i}(F)_{\tilde{\iota}_{i}(x), \mathbf{s}_{i-1} } = B_{i} \cap \Pfra^{\lfloor (-v_{\Afra}(c_{i-1})+1)/2 \rfloor}$, 
\item $\gfra^{i}(F)_{\tilde{\iota}_{i}(x), \mathbf{s}_{i-1}+ } = B_{i} \cap \Pfra^{\lfloor -v_{\Afra}(c_{i-1})/2 \rfloor +1}$, 
\item $\gfra^{i}(F)_{\tilde{\iota}_{i}(x), \rbf_{i-1} } = B_{i} \cap \Pfra^{-v_{\Afra}(c_{i-1})}$, and
\item $\gfra^{i}(F)_{\tilde{\iota}_{i}(x), \rbf_{i-1}+ } = B_{i} \cap \Pfra^{-v_{\Afra}(c_{i-1})+1}$.  
\end{enumerate}
\item For $i=0, \ldots, s$, we have
\begin{enumerate}
\item $G^{i}(F)_{\tilde{\iota}_{i}(x), \mathbf{s}_{i}+ } = B_{i}^{\times} \cap \U^{ \lfloor -v_{\Afra}(c_{i})/2 \rfloor +1}(\Afra)$, 
\item $G^{i}(F)_{\tilde{\iota}_{i}(x), \rbf_{i} }  = B_{i}^{\times} \cap \U^{-v_{\Afra}(c_{i})}(\Afra)$, 
\item $G^{i}(F)_{\tilde{\iota}_{i}(x), \rbf_{i}+ } = B_{i}^{\times} \cap \U^{-v_{\Afra}(c_{i})+1} (\Afra)$, 
\item $\gfra^{i}(F)_{\tilde{\iota}_{i}(x), \rbf_{i} } = B_{i} \cap \Pfra^{-v_{\Afra}(c_{i})}(\Afra)$, and
\item $\gfra^{i}(F)_{\tilde{\iota}_{i}(x), \rbf_{i}+ } = B_{i} \cap \Pfra^{-v_{\Afra}(c_{i})+1}(\Afra)$.  
\end{enumerate}
\end{enumerate}
\end{prop}

\begin{prf}
We find $x \in \Bscr^{E}(G^{0}, F)$.  
Since $B$ is a central simple $E_{0}$-algebra, there exists a division $E_{0}$-algebra $D_{E_{0}}$ and a right $D_{E_{0}}$-module $W_{0}$ such that $B \cong \End_{D_{E_{0}}}(W_{0})$.  
Since $\Bfra_{0}$ is a maximal hereditary $\ofra_{E_{0}}$-order in $B_{0}$, there exists an $\ofra_{D_{E_{0}}}$-chain $(\Lcal_{i})_{i \in \Z}$ in $W_{0}$ of period 1 such that $\Bfra_{0}$ is the hereditary $\ofra_{E_{0}}$-order associated with $(\Lcal_{i})_{i \in \Z}$.  
Let $x \in \Bscr^{E}(G^{0}, F) \cong \Bscr^{E}(\underline{\Aut}_{D_{E_{0}}}(W_{0}), E_{0})$ be an element which corresponds to a lattice function constructed from $(\Lcal_{i})_{i \in \Z}$.  
Then by Proposition \ref{beingvertex} $[x]$ is a vertex in $\Bscr^{E}(G^{0}, F)$.  
Therefore by Proposition \ref{compoffiltG} (3) we have (1)-(a).  

To show the remainder assertion, we show $\Afra$ is the hereditary $\ofra_{F}$-order in $A$ associated with $\tilde{\iota}_{d}(x)$.  
Since $[\Afra, n, 0, \beta]$ is a stratum, $\Afra$ is $E=F[\beta]$-pure.  
Moreover, we have $\Afra \cap B_{0} = \Bfra_{0}$ by definition of $\Bfra_{0}$.  
Therefore by Proposition \ref{beingprincipal} (2) $\Afra$ is associated with $\tilde{\iota}_{d}(x)$.  
Since $v_{\Afra}(c_{i}) \in \Z_{\geq 0}$ and $v_{\Afra}(c_{i}) = \ord(c_{i})e(\Afra|\ofra_{F})$, the remaining assertions follow from Proposition \ref{compoffiltC}.  
\end{prf}

In the following, we regard $\Bscr^{E}(G^{0}, F), \ldots, \Bscr^{E}(G^{d-1}, F)$ are subsets in $\Bscr^{E}(G, F)$ via $\tilde{\iota}_{1}, \ldots, \tilde{\iota}_{d}$.  

\begin{prop}
\label{compofHJ}
\begin{enumerate}
\item $H^{1}(\beta, \Afra)=K_{+}^{d}$, 
\item $J(\beta, \Afra)={}^{\circ}K^{d}$, 
\item $\hat{J}(\beta, \Afra)=K^{d}$.  
\end{enumerate}
\end{prop}

\begin{prf}
We show (1).  
We have $r_{i}=-v_{\Afra}(c_{i-1})=-\ord(c_{i-1})e(\Afra|\ofra_{F})=-e(\Afra|\ofra_{F}) \rbf_{i-1}$ for $i=1, \ldots, s$ and $n = -v_{\Afra}(c_{s}) = -e(\Afra|\ofra_{F}) \rbf_{s}$.  
We have $G^{0}(F)_{x,0+} = B_{0}^{\times} \cap \U^{1}(\Afra)$ by Proposition \ref{compofTLS} (1)-(c).  
For $i=1, \ldots, s$ we have $B_{i}^{\times} \cap \U^{\lfloor r_{i}/2 \rfloor +1}(\Afra) = B_{i}^{\times} \cap \U^{\lfloor -v_{\Afra}(c_{i-1})/2 \rfloor +1}(\Afra) = G^{i}(F)_{x, \mathbf{s}_{i-1}+}$ by Proposition \ref{compofTLS} (2)-(b).  
We also have $B_{s}^{\times} \cap \U^{\lfloor n/2 \rfloor +1}(\Afra) = G^{s}(F)_{x, \mathbf{s}_{s}+}$.  
If $d=s+1$, by comparing Lemma \ref{presenofHJ} (1) and Definition \ref{defofKi} (1) of $K_{+}^{d}$ we have $H^{1}(\beta, \Afra) = K_+^{d}$.  
If $d=s$, we have $H^{1}(\beta, \Afra) = K^{d} \U^{\lfloor n/2 \rfloor +1}(\Afra)$ and it suffices to show $K^{d} \supset \U^{\lfloor n/2 \rfloor +1}(\Afra)$.  
However, since $\mathbf{s}_{s-1} < \mathbf{s}_{s}$ we have 
\[
	\U^{\lfloor n/2 \rfloor +1}(\Afra) = G^{s}(F)_{x, \mathbf{s}_{s}+} \subset G^{s}(F)_{x, \mathbf{s}_{s-1}+} \subset K^{d}.  
\]

(2) is similarly shown as (1), using Proposition \ref{compofTLS} (1)-(b), (2)-(a), Lemma \ref{presenofHJ} (2) and Definition \ref{defofKi} (2) instead of Proposition \ref{compofTLS} (1)-(c), (2)-(b), Lemma \ref{presenofHJ} (1) and Definition \ref{defofKi} (1), respectively.  
Since $J(\beta, \Afra) = {}^{\circ}K^{d}$ and $\Kfra(\Bfra_{0}) = G^{0}(F)_{[x]}$ by Proposition \ref{compofTLS} (1)-(a), we obtain 
\[
	\hat{J}(\beta, \Afra) = \Kfra(\Bfra)J(\beta, \Afra) = G^{0}(F)_{[x]} {}^{\circ}K^{d} = K^{d}, 
\]
whence (3) holds.  
\end{prf}

Let $\theta \in \mathscr{C}(\beta, 0, \Afra)$ be the unique character of $H^{1}(\beta, \Afra)$ in $\lambda$.  
Then we can take characters $\phi_{i}$ of $E_{i}^{\times}$ for $i=0, 1, \ldots, s$ and define characters $\theta^{i}$ as in Proposition \ref{factchar}.  
We put $\Phibf_{i} = \phi_{i} \circ \Nrd_{B_{i}/E_{i}}$.  
If $d=s+1$, we put $\Phibf_{d}=1$.  

\begin{prop}
For $i=0, 1, \ldots, d-1$, the character $\Phibf_{i}$ is $G^{i+1}$-generic relative to $x$ of depth $\rbf_{i}$.  
If $s=d$, then $\Phibf_{d}$ is of depth $\rbf_{d}$.  
\end{prop}

\begin{prf}
First we show that the restriction of $\Phibf_{i}$ to $G^{i}(F)_{x, \mathbf{s}_{i}+}$ is equal to $\psi_{c_{i}}$ for $i=0, \ldots, s$.  
We have 
\begin{eqnarray*}
	B_{i}^{\times} \cap H^{1}(\beta, \Afra) & = & G^{i}(F) \cap K_{+}^{d} \\
	& = & G^{0}(F)_{x, 0+} G^{1}(F)_{x,\mathbf{s}_{0}+} \cdots G^{i}(F)_{x, \mathbf{s}_{i-1}+}, 
\end{eqnarray*}
and $G^{i}(F)_{x, \mathbf{s}_{i}+} \subset B_{i}^{\times} \cap H^{1}(\beta, \Afra)$, as we have $s_{i} > s_{i-1}$ and then
\[
	G^{i}(F)_{x, \mathbf{s}_{i}+} \subset G^{i}(F)_{x, \mathbf{s}_{i-1}+} \subset G^{0}(F)_{x, 0+} G^{1}(F)_{x,\mathbf{s}_{0}+} \cdots G^{i}(F)_{x, \mathbf{s}_{i-1}+}.  
\]
To show $G^{i}(F)_{x,\mathbf{s}_{i}+} \subset H^{t_{i}+1}(\beta, \Afra)$, where $t_{i} = \max \{ 0, \lfloor -v_{\Afra}(c_{i})/2 \rfloor \} = \lfloor -v_{\Afra}(c_{i})/2 \rfloor$, we consider two cases.  
If $i<d$, we have 
\begin{eqnarray*}
	H^{t_{i}+1}(\beta, \Afra) & = & H^{1}(\beta, \Afra) \cap \U^{t_{i}+1}(\Afra) \\
	& = & K_{+}^{d} \cap G(F)_{x, \mathbf{s}_{i}+ } \\
	& = & G^{i+1}(F)_{x, \mathbf{s}_{i}+} \cdots G^{d}(F)_{x, \mathbf{s}_{d-1}+}, 
\end{eqnarray*}
and $G^{i}(F)_{x, \mathbf{s}_{i}+} \subset H^{t_{i}+1}(\beta, \Afra)$ since 
\[
	G^{i}(F)_{x, \mathbf{s}_{i}+} \subset G^{i+1}(F)_{x, \mathbf{s}_{i}+} \subset G^{i+1}(F)_{x, \mathbf{s}_{i}+} \cdots G^{d}(F)_{x, \mathbf{s}_{d-1}+}.  
\]
Otherwise, that is, if $i=s=d$, we also have
\[
	H^{t_{d}+1}(\beta, \Afra) = K_{+}^{d} \cap G(F)_{x, \mathbf{s}_{d}+ } = G^{d}(F)_{x, \mathbf{s}_{d}+}.  
\]
Therefore $G^{i}(F)_{x, \mathbf{s}_{i}+} \subset \left( B_{i}^{\times} \cap H^{1}(\beta, \Afra) \right) \cap H^{t_{i}+1}(\beta, \Afra)$, and we obtain 
\[
	\Phibf_{i} |_{G^{i}(F)_{x, \mathbf{s}_{i}+}} = \theta^{i} |_{G^{i}(F)_{x, \mathbf{s}_{i}+}} \\ = \psi_{c_{i}} |_{G^{i}(F)_{x, \mathbf{s}_{i}+}}.  
\]
In particular, $\Phibf_{i}$ is trivial on 
\[
\U^{-v_{\Afra}(c_{i})+1}(\Afra) \cap G^{i}(F)_{x, \mathbf{s}_{i}+} = G(F)_{x, \rbf_{i}+} \cap G^{i}(F)_{x, \mathbf{s}_{i}+} = G^{i}(F)_{x, \rbf_{i}+}.  
\]
Note that all $c_i$ have negative valuation.
Next, we show $\Phibf_{i}$ is not trivial on $G^{i}(F)_{x,\rbf_{i}}$.  
We have $G^{i}(F)_{x,\rbf_{i}} =\U(\Bfra_{i}) \cap \U^{-v_{\Afra}(c_{i})}(\Afra) = B_{i} \cap \left( 1+\Pfra^{-v_{\Afra}(c_{i})} \right) = 1 + \left( B_{i} \cap \Pfra^{-v_{\Afra}(c_{i})} \right)$.  
Then 
\begin{eqnarray*}
	\Phibf_{i}(G^{i}(F)_{x,\rbf_{i}}) & = & \psi_{c_{i}} \Bigl( 1 + \bigl( B_{i} \cap \Pfra^{-v_{\Afra}(c_{i})} \bigr) \Bigr) = \psi \circ \Trd_{A/F} \Bigl( c_{i} \bigl( B_{i} \cap \Pfra^{-v_{\Afra}(c_{i})} \bigr) \Bigr) \\
	& = & \psi \circ \Trd_{A/F} \left( B_{i} \cap \Afra \right) = \psi \circ \Tr_{E_{i}/F} \circ \Trd_{B_{i}/E_{i}} (\Bfra_{i}).  
\end{eqnarray*}
Since $\Bfra_{i}$ is a hereditary $\ofra_{E_{i}}$-order in $B_{i}$, we have $\Trd_{B_{i}/E_{i}} (\Bfra_{i}) = \ofra_{E_{i}}$.  
Moreover, since $E_{i}/F$ is tamely ramified, $\Tr_{E_{i}/F} (\ofra_{E_{i}}) = \ofra_{F}$.  
Therefore $\Phibf_{i}$ is not trivial on $G^{i}(F)_{x,\rbf_{i}}$, as $\psi$ is not trivial on $\ofra_{F}$.  
In particular, we completed the proof when $i=s=d$ and we may assume $i<d$ in the following.  

Finally, let $X_{c_{i}}^{*} \in (\underline{\Lie^{*}}(G^{i}))^{G^i}(F)$ as \S \ref{defofXc}.  
Since $c_{i}$ is minimal relative to $E_{i}/E_{i+1}$, the element $X_{c_{i}}^{*}$ is $G^{i+1}$-generic of depth $\rbf_{i}$ by Proposition \ref{genelem}.  
Then, to complete the proof it suffices to show that $\Phibf_{i} |_{G^{i}(F)_{x,\rbf_{i}:\rbf_{i}+}}$ is realized by $X_{c_{i}}^{*}$.  
The isomorphism $G^{i}(F)_{x,\rbf_{i}:\rbf_{i}+} \cong \gfra^{i}(F)_{x,\rbf_{i}:\rbf_{i}+}$ is induced from $1+y \mapsto y$.  
Therefore, when we regard $\psi \circ X_{c_{i}}^{*}$ as a character of $G^{i}(F)_{x,\rbf_{i}:\rbf_{i}+}$, for $1+y \in G^{i}(F)_{x, \rbf_{i}}$ we have
\[
	\left( \psi \circ X_{c_{i}}^{*} \right) (1+y) = \psi \circ X_{c_{i}}^{*}(y) = \psi \circ \Trd_{A/F}(c_{i}y) = \psi_{c_{i}}(1+y) = \Phibf_{i}(1+y).  
\]
\end{prf}

Then we have a 4-tuple $\left( x, (G^{i}), (\rbf_{i}), (\Phibf_{i}) \right)$.  
As in \S \ref{Yuconst}, we can define characters $\hat{\Phibf}_{i}$ of $K_{+}^{d}$.  

\begin{prop}
\label{phiistheta1}
For $i=0, 1, \ldots, s$, we have $\hat{\Phibf}_{i}=\theta^{i}$.  
\end{prop}

\begin{prf}
Recall the definition of $\hat{\Phibf}_{i}$.  The character $\hat{\Phibf}_{i}$ is defined as 
\begin{eqnarray*}
\hat{\Phibf}_{i}|_{K_{+}^{d} \cap G^{i}(F)} (g) & = & \Phibf_{i}(g), \\
\hat{\Phibf}_{i}|_{K_{+}^{d} \cap G(F)_{x, \mathbf{s}_{i}+ } } (1+y) & = & \Phibf_{i} (1+\pi_{i}(y) ).  
\end{eqnarray*}

Since $\left( K_{+}^{d} \cap G^{i}(F) \right) \left( K_{+}^{d} \cap G(F)_{x, \mathbf{s}_{i}+ } \right) =K^{d}_{+}$, it is enough to show that $\hat{\Phibf}_{i}$ is equal to $\theta^{i}$ on $K_{+}^{d} \cap G^{i}(F)$ and $K_{+}^{d} \cap G(F)_{x, \mathbf{s}_{i}+ }$.  

We have that $K^{d}_{+} \cap G^{i}(F) = B_{i}^{\times} \cap H^{1}(\beta, \Afra)$ and $K^{d}_{+} \cap G(F)_{r, \mathbf{s}_{i}+}=H^{t_{i}+1}(\beta, \Afra)$, where $t_{i} = \lfloor -v_{\Afra}(c_{i})/2 \rfloor$.  

If $g \in B_{i}^{\times} \cap H^{1}(\beta, \Afra)$, then $\hat{\Phibf}_{i}(g)=\Phibf_{i}(g)=\phi_{i} \circ \Nrd_{B_{i}/E_{i}}(g) = \theta^{i}(g)$.  

Suppose $1+y \in H^{t_{i}+1}(\beta, \Afra)$.  
Then $\pi_{i}(y) \in \mathfrak{g}^{i}(F)_{x, \mathbf{s}_{i}+}=B \cap \Pfra_{i}^{t_{i}+1}$ and $1+\pi_{i}(y) \in B_{i}^{\times} \cap H^{t_{i}+1}(\beta, \Afra)$.  
Therefore we have $\hat{\Phibf}_{i}(1+y)=\Phibf_{i}(1+\pi_{i}(y))=\theta^{i}(1+\pi_{i}(y))=\psi \circ \Trd_{A/F}(c_{i} \pi_{i}(y))$.  

Here, we show if $n \in \mathfrak{n}^{i}(F)$, then $\Trd_{A/F}(c_{i}n)=0$.  
Since $c_{i}$ is in the center of $B_{i}^{\times}$, the linear automorphism $z \mapsto c_{i}z$ of $A$ is also a $Z(G^{i})(F)$-automorphism.  
Then $c_{i}\gfra^{i}(F)$ is a trivial $Z(G^{i})(F)$-representation and $c_{i}\mathfrak{n}^{i}(F) \cong \mathfrak{n}^{i}(F)$ is a $Z(G^{i})(F)$-representation which does not contain any trivial subquotient.
Therefore we have $c_{i}\gfra^{i}(F) \subset \gfra^{i}(F)$ and $c_{i} \mathfrak{n}^{i}(F) \subset \mathfrak{n}^{i}(F)$.  
On the other hand, $\Trd_{A/F}$ is a $Z(G^{i})(F)$-homomorphism from $\gfra(F)$ to the trivial representation $F$.  
Since $\mathfrak{n}^{i}(F)$ does not have any trivial quotient, $\Trd_{A/F}(\mathfrak{n}^{i}(F))=0$.  
In particular, $\Trd_{A/F}(c_{i}n)=0$ as $c_{i}n \in c_{i} \mathfrak{n}^{i}(F) \subset \mathfrak{n}^{i}(F)$.  

Since $\pi_{i}: \gfra^{i}(F) \oplus \mathfrak{n}^{i}(F) \to \gfra^{i}(F)$ is the projection, $y-\pi_{i}(y) \in \mathfrak{n}^{i}(F)$.  
Therefore we have $\Trd_{A/F}( c_{i}y)=\Trd_{A/F} \left( c_{i} \left( y-\pi_{i}(y) \right) \right) + \Trd_{A/F} \left( c_{i}\pi_{i}(y) \right) = \Trd_{A/F} \left( c_{i}\pi_{i}(y) \right)$ and
\[
	\hat{\Phibf}_{i}(1+y) = \psi \circ \Trd_{A/F}(c_{i} \pi_{i}(y)) = \psi \circ \Trd_{A/F}(c_{i}y) = \psi_{c_i}(1+y) = \theta^{i}(1+y).  
\]
\end{prf}

\begin{prop}
\label{betaext}
The representation $\kappa_{0} \otimes \cdots \otimes \kappa_{d}$ is an extension of $\eta_\theta$ to $K^{d}$ (cf. Definition \ref{infinf} for the definition of $\kappa _i$).  
\end{prop}

\begin{prf}
We put $\hat{\kappa}'=\kappa_{0} \otimes \cdots \otimes \kappa_{d}$.  
By \cite[Lemma 3.27]{HM}, $\kappa_{i} |_{K_{+}^{d}}$ contains $\hat{\Phibf}_{i}$ for $i=0, \ldots, d$.  
If $d=s+1$, then $\hat{\Phibf}_{d}=\Phibf_{d}=1$ and 
\[
	\prod_{i=0}^{d} \hat{\Phibf}_{i} = \prod_{i=0}^{s} \hat{\Phibf}_{i} = \prod_{i=0}^{s} \theta^{i} = \theta
\]
by Proposition \ref{factchar} and Proposition \ref{phiistheta1}.  
Then $\hat{\kappa}'$ contains $\theta$ as a $K_{+}^{d}$-representation.  
Since $\eta_{\theta}$ is the unique irreducible $J^{1}(\beta, \Afra)$-representation which contains $\theta$, the $J^{1}(\beta, \Afra)$-representation $\hat{\kappa}'$ contains $\eta_{\theta}$.  

Then it suffices to show that the dimension of $\hat{\kappa}$ is equal to the dimension of $\eta_{\theta}$.  
The dimension of $\eta_{\theta}$ is $(J^{1}(\beta, \Afra) : H^{1}(\beta, \Afra))^{1/2}$.  
On the other hand, for $i=0, \ldots, d-1$ the dimension of $\kappa_{i}$ is $(J^{i+1} : J_{+}^{i+1})^{1/2}$ (cf. Proposition \ref{dimkap}), and the dimension of $\kappa_{d}$ is 1.  
Then the dimension of $\hat{\kappa}'$ is $\prod_{i=1}^{d} (J^{i} : J_{+}^{i})^{1/2}$, and it suffices to show that $(J^{1}(\beta, \Afra) : H^{1}(\beta, \Afra)) = \prod_{i=1}^{d} (J^{i} : J_{+}^{i})$.  
Here, $H^{1}(\beta, \Afra) = K_{+}^{d} = K_{+}^{0}J_{+}^{1} \cdots J_{+}^{d} = G^{0}(F)_{x,0+}J_{+}^{1} \cdots J_{+}^{d}$.  
Since $G^{i}(F)_{x,\mathbf{s}_{i}}J^{i+1}=G^{i+1}(F)_{x,\mathbf{s}_{i}}$ for $i=0, \ldots, d-1$, we also have 
\begin{eqnarray*}
	G^{0}(F)_{x,0+}J^{1} \cdots J^{d} & = & G^{0}(F)_{x,0+}G^{0}(F)_{x,\mathbf{s}_{0}}J^{1} \cdots J^{d} \\
	& = & G^{0}(F)_{x,0+}G^{1}(F)_{x,\mathbf{s}_{0}} J^{2} \cdots J^{d} = \cdots \\
	& = & G^{0}(F)_{x,0+}G^{1}(F)_{x, \mathbf{s}_{0}} \cdots G^{d}(F)_{x, \mathbf{s}_{d-1}} \\
	& = & G(F)_{x,0+} \cap K^{d}=\U^{1}(\Afra) \cap J(\beta, \Afra) = J^{1}(\beta, \Afra).  
\end{eqnarray*}
Since $G^{0}(F)_{x,0+} \cap (J^{1} \cdots J^{d}) = G^{0}(F)_{x,\rbf_{0}} \subset J_{+}^{1}$ (e.g. using \cite[Lemma 13.2]{Yu}), we have $\left( G^{0}(F)_{x,0+} J_{+}^{1} \cdots J_{+}^{d} \right) \cap ( J^{1} \cdots J^{d} ) = J_{+}^{1} \cdots J_{+}^{d}$, and
\begin{eqnarray*}
	J^{1}(\beta, \Afra)/H^{1}(\beta, \Afra) & = & \left( G^{0}(F)_{x,0+}J^{1} \cdots J^{d} \right) / \left( G^{0}(F)_{x,0+} J_{+}^{1} \cdots J_{+}^{d} \right) \\
	& \cong & ( J^{1} \cdots J^{d} ) / ( J_{+}^{1} \cdots J_{+}^{d} ).  
\end{eqnarray*}
Then it is enough to show $\left( ( J^{1} \cdots J^{d} ) : ( J_{+}^{1} \cdots J_{+}^{d} ) \right) = \prod_{i=1}^{d} (J^{i} : J_{+}^{i})$.  Let us prove this by induction on $d$. If $d=1$, this is trivial. Let us assume that this is true for $d-1$. It is now enough to show that $[J^d : J^d_+ ] = \frac{[J^1 \cdots J^d : J^1 _+ \cdots J^d _+ ]} {[J^1 \cdots J^{d-1} : J^1 _+ \cdots J^{d-1} _+ ]}$.
 The following fact will be useful.

\textit{Fact: Let $G' \subset G$ be groups and let $H $ be a normal subgroup of $G$. Let $\iota $ be the injective morphism of group $ G' / ( G' \cap H ) \hookrightarrow  G/ H $. As $G$-set, $G/HG'$ and $ (G/H) /\iota (G' / (G' \cap H )) $ are isomorphic.}
 \begin{sloppypar}Because $J^1 _+ \cdots J^d _+$ is a normal subgroup of $ J^1 \cdots J^d$, we can apply the previous fact to $G= J^1 \cdots J^d$, $G' = J^1 \cdots J^{d-1}$ , $H= J^1 _+ \cdots J^d _+$. Using the fact that  $H \cap G' = J^1 _+ \cdots J^{d-1} _+$, we deduce that, as $J^1 \cdots J^d $-sets, $J^1 \cdots J^d / J^1 \cdots J^{d-1} J^d_+$ and $ (J^1 \cdots J^d / J^1 _+ \cdots J^d _+ ) / \iota (J^1 \cdots J^{d-1} / J^1 _+ \cdots J^{d-1} _+ )$ are isomorphic. Let $X$ be this $J^1 \cdots J^d$-set. 
 The set $X$ is a fortiori a $J^d$-set. The group $J^d$ acts transitively on $X=J^1 \cdots J^d / J^1 \cdots J^{d-1} J^d_+$, and the stabiliser of $ (J^1 \cdots J^{d-1} J^d_+) \in J^1 \cdots J^d / J^1 \cdots J^{d-1} J^d_+$ is  $J^1 \cdots J^{d-1} J^d_+ \cap J^d $. The group $J^1 \cdots J^{d-1} J^d_+ \cap J^d $ is equal to $J^d_+$. Consequently, \end{sloppypar} \begin{center}$[J^d : J^d_+ ] = \#(X) = \frac{[J^1 \cdots J^d : J^1 _+ \cdots J^d _+ ]} {[J^1 \cdots J^{d-1} : J^1 _+ \cdots J^{d-1} _+ ]},$\end{center} as required.
Therefore we obtain $\hat{\kappa}' |_{J^{1}(\beta, \Afra)} = \eta_{\theta}$.  
\end{prf}

\begin{thm}
\label{Main1}
Let $(J, \lambda)$ be a maximal simple type associated to a tame simple stratum $[\Afra, n, 0, \beta]$.  
Let $(\tilde{J}, \Lambda)$ be a maximal extension of $(J, \lambda)$.  
Then there exists a Yu datum $\left( x, (G^{i})_{i=0}^{d}, (\rbf_{i})_{i=0}^{d}, (\Phibf_{i})_{i=0}^{d}, \rho \right)$ such that
\begin{enumerate}
\item $\hat{J}(\beta, \Afra)=K^{d}$, and
\item $\rho_{d} \left( x, (G^{i}), (\rbf_{i}), (\Phibf_{i}), \rho \right) \cong \cInd_{\tilde{J}}^{\hat{J}(\beta, \Afra)} \Lambda$.  
\end{enumerate}
\end{thm}

\begin{prf}
In the above argument, we can take a 4-tuple $\left( x, (G^{i}), (\rbf_{i}), (\Phibf_{i}) \right)$ from a S\'echerre datum.  
Therefore it is enough to show that we can take an irreducible $G^{0}(F)_{[x]}$-representation $\rho$ such that the Yu datum $\left( x, (G^{i}), (\rbf_{i}), (\Phibf_{i}), \rho \right)$ satisfies the desired conditions.  

Let $\eta$ be the unique $J^{1}(\beta, \Afra)$-subrepresentation in $\lambda|_{J^{1}(\beta, \Afra)}$.  
Then $\kappa_{0} \otimes \cdots \otimes \kappa_{d}$ is an extension of $\eta$ to $K^{d}=\hat{J}(\beta, \Afra)$ by Proposition \ref{betaext}.  
Therefore there exists an irreducible $\Kfra(\Bfra)$-representation $\rho$ such that $\rho$ is trivial on $\U^{1}(\Bfra)$ but not trivial on $\U(\Bfra)$, the representation $\cInd_{\Kfra(\Bfra)}^{B^{\times}}\rho$ is irreducible and supercuspidal, and $\cInd_{\tilde{J}}^{\hat{J}(\beta, \Afra)} \Lambda \cong \rho \otimes \kappa_{0} \otimes \cdots \otimes \kappa_{d}$ by Proposition \ref{findingdepth0}.  
Since we have equalities of groups $B^{\times}=G^{0}(F)$, $\Kfra(\Bfra)=G^{0}(F)_{[x]}$, $\U(\Bfra)=G^{0}(F)_{x}$ and $\U^{1}(\Bfra) = G^{0}(F)_{x+}$, then the 5-tuple $\left( x, (G^{i}), (\rbf_{i}), (\Phibf_{i}), \rho \right)$ is a Yu datum satisfying the condition in the theorem.  
\end{prf}

\section{Construction of a S\'echerre datum from a Yu datum}
\label{YutoSS}

Let $(x, (G^{i})_{i=0}^{d}, (\rbf_{i})_{i=0}^{d}, (\Phibf_{i})_{i=0}^{d}, \rho)$ be a Yu datum.  

First, since $G^{i}$ are tame twisted Levi subgroups in $G$ with $Z(G^{i})/Z(G)$ anisotropic, there exist tamely ramified field extensions $E_{i}/F$ in $A$ such that 
\[
	G^{i} \cong \Res_{E_{i}/F} \underline{\Aut}_{D \otimes_{F} E_{i}} (V)
\]
by Lemma \ref{detofTLG}.  
Since $G^{0} \subsetneq \ldots \subsetneq G^{d}$, we can choose $E_{0} \supsetneq \ldots \supsetneq E_{d} = F$.  
We put $B_{i} = \Cent_{A}(E_{i})$.  

Since $\cInd_{G^{0}(F)_{[x]}}^{G^{0}(F)} \rho$ is supercuspidal, $[x]$ is a vertex in $\Bscr^{E}(G^{0},F)$ by Proposition \ref{depth0ofYu}.  
Let $\Bfra_{0}$ be the hereditary $\ofra_{E_{0}}$-order in $B_{0}$ associated with $x$.  
Then the hereditary $\ofra_{F}$-order $\Afra$ associated with $x \in \Bscr^{E}(G,F)$ is $E_{0}$-pure and principal, and $\Afra \cap B_{0} = \Bfra_{0}$ by Proposition \ref{beingprincipal}.  
We also put $\Pfra = \Pfra(\Afra)$.  

To obtain a simple stratum, we need an element $\beta \in E_{0}$.  
We will take $\beta$ by using information from characters $(\Phibf_{i})_{i}$.  
For $c_{i} \in E_{i} = \Lie(Z(G^{i}))$, let $X_{c_{i}}^{*} \in \Lie^{*} (Z(G^{i}))$ be as in \S \ref{defofXc}.  
We put $s=\sup \{ i \mid \Phibf_i \neq 1 \}$.  

\begin{prop}
\label{takec}
Suppose $s \geq 0$.  
\begin{enumerate}
\item For $i=0, \ldots, d$, the hereditary $\ofra_{E_{i}}$-order in $B_{i}$ associated with $x \in \Bscr^{E}(G^{i}, F)$ is equal to $\Bfra_{i} = B_{i} \cap \Afra$.  
\item There exists $c_{i} \in \Lie(Z(G^{i}))_{-\rbf_{i}}$ such that $\Phibf_{i} | _{G^{i}(F)_{x,\rbf_{i}/2+:\rbf_{i}+}}$ is realized by $X_{c_i}^{*}$ for $i=0, \ldots, d-1$.  
\item If $s=d$, then there also exists $c_{s} \in \Lie(Z(G))_{-\rbf_{s}}$ such that $\Phibf_{s} | _{G(F)_{x,\rbf_{s}/2+:\rbf_{s}+}}$ is realized by $X_{c_s}^{*}$.  
\item For $i=0, \ldots, s$, we have $\rbf_{i} = -\ord(c_{i})$.  
\item For $i=0, \ldots, d-1$, the element $c_{i}$ is minimal relative to $E_{i}/E_{i+1}$.  
In particular, we have $E_{i} = E_{i+1}[c_{i}]$.  
\end{enumerate}
\end{prop}

\begin{prf}
We show (1).  
First, we have $\Bfra_{i} \cap B_{0} = \Afra \cap B_{i} \cap B_{0} = \Afra \cap B_{0} = \Bfra_{0}$.  
Moreover, for $g \in E_{i}^{\times}$ we also have 
\[
	g \Bfra_{i} g^{-1} = g (\Afra \cap B_{i}) g^{-1} = g \Afra g^{-1} \cap g B_{i} g^{-1} = \Afra \cap B_{i} = \Bfra_{i}, 
\]
as $\Afra$ is $E_{0}$-pure and $E_{0} \subset B_{i}$.  
Therefore (1) holds by Proposition \ref{beingprincipal} (2).  

Next, we show (2), and (3) is similarly shown.  
Since $\Phibf_{i}$ is trivial on $G^{i}(F)_{x, \rbf_{i}+}$ but not on $G^{i}(F)_{x, \rbf_{i}}$, we have $G^{i}(F)_{x, \rbf_{i}} \neq G^{i}(F)_{x, \rbf_{i}+}$ in particular.  
Then $n_{i} = \rbf_{i}e(\Bfra_{i}|\ofra_{E_{i}})e(E_{i}/F)$ is a non-negative integer and we have $G^{i}(F)_{x, \rbf_{i}} = \U^{n}(\Bfra)$ and $G^{i}(F)_{x, \rbf_{i}+} = \U^{n+1}(\Bfra)$, by Lemma \ref{lemforfindc} (3).  
On the other hand, a character $\psi \circ \Tr_{E_{i}/F}$ of $E_{i}$ is with conductor $\pfra_{E_{i}}$ since $E_{i}/F$ is tamely ramified.  
Therefore, we can apply Proposition \ref{propforgench} for $\Bfra_{i}, n$ and $\psi \circ \Tr_{E_{i}/F}$ as $\Bfra_{i}$ is principal by (1) and Proposition \ref{beingprincipal} (1).  
Thus there exists $c_{i} \in E_{i}$ such that 
\[
	\Phibf_{i} (1+y) = (\psi \circ \Tr_{E_{i}/F}) \circ \Trd_{B_{i}/E_{i}}(c_{i}y) = \psi \circ (\Tr_{E_{i}/F} \circ \Trd_{B_{i}/E_{i}}) (c_{i}y) = \psi \circ X_{c_{i}}^{*}(y)
\]
for $1+y \in \U^{\lfloor n_{i}/2 \rfloor +1}(\Bfra_{i}) = G^{i}(F)_{x, \mathbf{r}_{i}/2+}$.  
Then (2) holds.  

We have $v_{E_{i}}(c_{i})=-n_{i}/e(\Bfra_{i}|\ofra_{E_{i}})=-\rbf_{i}e(E_{i}/F)$ by Proposition \ref{propforgench}, and
\[
\ord(c_{i})=v_{E_{i}}(c_{i})/e(E_{i}/F)=-\rbf_{i}, 
\]
whence (4) holds.  

To show (5), let $c'_{i} \in E_{i}^{\times}$ such that $X_{c'_{i}}^{*}$ is $G^{i+1}$-generic of depth $\rbf_{i}$ and  the restriction of $\Phibf_i $ to  $G^{i}(F)_{x, \rbf_{i}:\rbf_{i}+}$ is realized by $X^*_{c'_{i}}$.  
In particular, we have
\[
	(\psi \circ \Tr_{E_{i}/F}) \circ \Trd_{B_{i}/E_{i}}(c_{i}y) = \Phibf_{i}(1+y) = \psi \circ X_{c'_{i}}(y) = (\psi \circ \Tr_{E_{i}/F}) \circ \Trd_{B_{i}/E_{i}}(c'_{i}y)
\]
for $y \in \Qfra_{i}^{n_{i}}$, where $\Qfra_{i}$ is the radical of $\Bfra_{i}$.  
Then we have $c_{i} - c'_{i} \in \Qfra_{i}^{-n_{i}+1} \cap E_{i} \subset c_{i}(\Qfra_{i} \cap E_{i}) = c_{i}\pfra_{E_{i}}$ and $c_{i}^{-1}c'_{i} \in 1+\pfra_{E_{i}}$.  
Thus $(c'_{i})^{-1}c_{i} \in 1+\pfra_{E_{i}}$.  
On the other hand, $c'_{i}$ is minimal relative to $E_{i}/E_{i+1}$ by Proposition \ref{genelem}.  
Therefore, by Lemma \ref{preservemin} $c_{i}$ is also minimal relative to $E_{i}/E_{i+1}$.  
\end{prf}

Therefore if $s \geq 0$, we can take $c_i$ for $i=0,1, \ldots, s$.  
We put $\beta_{i} = \sum_{j=i}^{s} c_{j}$ for $i=0,1, \ldots, s$, $\beta = \beta_{0}$ and $n = -v_{\Afra}(\beta)$.  
Since 
\[
v_{\Afra}(c_i) = -e(\Afra|\ofra_{F}) \ord(c_{i}) = -e(\Afra|\ofra_{F})\rbf_{i} < -e(\Afra|\ofra_{F})\rbf_{j} = -e(\Afra|\ofra_{F}) \ord(c_{j}) = v_{\Afra}(c_j)
\]
for $i,j=0, 1, \ldots, s$ with $i>j$, we have $n=-v_{\Afra}(\beta_{i})$ for any $i=0, 1, \ldots, s$.  
We also put $r_{i} = -v_{\Afra}(c_{i-1})$ for $i=1, \ldots, s$ and $r_{0}=0$.  

\begin{prop}
Suppose $s \geq 0$.  
\begin{enumerate}
\item $E_{i}=F[\beta_{i}]$ for $i=0, 1, \ldots, s$.  
	In particular, $[\Afra, n, 0, \beta]$ is a simple stratum.  
\item $\left( [\Afra, n, r_{i}, \beta_{i}] \right) _{i=0}^{s}$ is a defining sequence of $[\Afra, n, 0, \beta]$.  
\end{enumerate}
\end{prop}

\begin{prf}
First, suppose $\Afra = \Afra(E_0)$.  
We will show this proposition by downward induction on $i$.  

If $i=s$, then $\beta_{s}=c_{s}$ is minimal over $F$.  
Therefore for any $r' \in \{ 0, 1, \ldots, n-1 \}$, the stratum $[\Afra, n, r', \beta_{s}]$ is simple.  
The equation $E_{s}=F[\beta_{s}]$ trivially holds.  
If $s=0$, then $\left( [\Afra, n, r_{i}, \beta_{i}] \right) _{i=0}^{0}$ is a defining sequence of $[\Afra, n, 0, \beta]$ and this proposition holds.  
If $s>0$, we have $r_{s} = -v_{\Afra}(c_{s-1}) < -v_{\Afra}(c_{s})$. We prove by downward induction on $i_0$ that $\left( [\Afra, n, r_{j+i_{0}}, \beta_{j+i_{0}}] \right) _{j=0}^{s-i_{0}}$ is a defining sequence of a simple stratum $[\Afra, n, r_{i_{0}}, \beta_{i_{0}}]$. 
For $i_0=s$, the stratum $[\Afra, n, r_{s}, \beta_{s}]$ is simple and $\left( [\Afra, n, r_{i+s}, \beta_{i+s}] \right)_{i=0}^{0}$ is a defining sequence of $[\Afra, n, r_{s}, \beta_{s}]$.  

Let $i_{0} \in \{ 0, 1, \ldots, s-1 \} $ and suppose that $E_{i}=F[\beta_{i}]$ and that $\left( [\Afra, n, r_{j+i}, \beta_{j+i}] \right)_{j=0}^{s-i}$ is a defining sequence of a simple stratum $[\Afra, n, r_{i}, \beta_{i}]$ for any integer $i$ with $i_0 < i  \leq s$.  
The element $c_{i_{0}}$ is minimal over $E_{i_{0}+1}$.  
Since $r_{i_{0}+1}=-v_{\Afra}(c_{i_0})$, a 4-tuple $[\Bfra_{\beta_{i_0}+1}, r_{i_{0}+1}, r_{i_{0}+1}-1, c_{i_0}]$ is a simple stratum, where $\Bfra_{\beta_{i_{0}+1}} = \Afra \cap \Cent_{A(E_0)}(\beta_{i_{0}+1})$.  
Moreover, $c_{i_0} \notin E_{i_{0}+1}=F[\beta_{i_{0}+1}]$.  
Therefore, by Proposition \ref{Mforcons}, we have $F[\beta_{i_0}] = F[\beta_{i_{0}+1}, c_{i_0}] = E_{i_{0}+1}[c_{i_{0}}]$ and $[\Afra, n, r_{i_{0}+1}, \beta_{i_0}]$ is a pure stratum with $k_{0}(\beta_{i_0}, \Afra)=-r_{i_{0}+1}$, where $F[\beta_{i_{0}+1}, c_{i_0}] = E_{i_{0}+1}[c_{i_{0}}]$ follows from our induction hypothesis.  
If $i_0>0$, we have $r_{i_0}=-v_{\Afra}(c_{i_{0}-1}) < -v_{\Afra}(c_{i_{0}})=r_{i_{0}+1}$ and $[\Afra, n, r_{i_{0}}, \beta_{i_{0}}]$ is a simple stratum.  
Since $\left( [\Afra, n, r_{j+i_{0}+1}, \beta_{j+i_{0}+1}] \right)_{j=0}^{s-i_{0}-1}$ is a defining sequence of a simple stratum $[\Afra, n, r_{i_{0}+1}, \beta_{i_{0}+1}]$ by our induction hypothesis, $\left( [\Afra, n, r_{j+i_{0}}, \beta_{j+i_{0}}] \right) _{j=0}^{s-i_{0}}$ is also a defining sequence of a simple stratum $[\Afra, n, r_{i_{0}}, \beta_{i_{0}}]$.  
If $i_{0}=0$, then $[\Afra, n, 0, \beta]$ is simple and we can show $\left( [\Afra, n, r_{i}, \beta_{i}] \right) _{i=0}^{s}$ is also a defining sequence of a simple stratum $[\Afra, n, 0, \beta]$ in the same way as above.  
Then the proposition for $\Afra=\Afra(E_{0})$ case holds.  

We will show the proposition in general case.  
Since $\beta_{i} \in E_{i} \subset E_{0}$ for $i=0, \ldots, s$, we can regard $\beta_{i}$ as in $A(E_0)$.  
Then (1) follows from the proposition for $\Afra=\Afra(E_0)$ case.  
Moreover, if we put $n'=-v_{\Afra(E_0)}(\beta)$, $r'_{0}=0$ and $r'_{i}=-v_{\Afra(E_0)}(c_{i-1})$ for $i=1, \ldots, s$, then $\left( [\Afra(E_0), n', r'_{i}, \beta_{i}] \right) _{i=0}^{s}$ is a defining sequence of a simple type $[\Afra(E_0), n', 0, \beta]$ by the proposition for $\Afra=\Afra(E_0)$ case.  
Since for $c \in E_0$ we have $v_{\Afra}(c)=e(\Afra|\ofra_{F})e(E_{0}/F)^{-1}v_{\Afra(E_0)}(c)$, we also have
\[
	n=-v_{\Afra}(\beta)=-e(\Afra|\ofra_{F})e(E_{0}/F)^{-1}v_{\Afra(E_0)}(\beta)=e(\Afra|\ofra_{F})e(E_{0}/F)^{-1}n'
\]
and
\[
r_{i}=-v_{\Afra}(c_{i-1})=-e(\Afra|\ofra_{F})e(E_{0}/F)^{-1}v_{\Afra(E_0)}(c_{i-1})=e(\Afra|\ofra_{F})e(E_{0}/F)^{-1}r'_{i}
\]
for $i=1, \ldots, s$.  
Since $\left( [\Afra(E_0), n', r'_{i}, \beta_{i}] \right) _{i=0}^{s}$ is a defining sequence of a simple type $[\Afra(E_0), n', 0, \beta]$, we have $r'_{i}=-k_{0}(\beta_{i-1}, \Afra(E_0))$ for $i=1, \ldots, s$.  
We also have $k_{0}(c, \Afra)=e(\Afra|\ofra_{F})e(E_{0}/F)^{-1}k_{0}(c, \Afra(E_0))$ by Lemma \ref{compofk0}, whence
\[
r_{i}=e(\Afra|\ofra_{F})e(E_{0}/F)^{-1}r'_{i}=-e(\Afra|\ofra_{F})e(E_{0}/F)^{-1}k_{0}(\beta_{i-1}, \Afra(E_0))=-k_{0}(\beta_{i-1}, \Afra)
\]
for $i=1, \ldots, s$.  Then by Proposition \ref{BHforapp} strata $[\Afra, n, r_{i}, \beta_{i}]$ are simple and equivalent to $[\Afra, n, r_{i}, \beta_{i-1}]$ for $i=1, \ldots, s$.  
Therefore (2) holds.  
\end{prf}

Then we have a simple stratum $[\Afra, n, 0, \beta]$ with a defining sequence $([\Afra, n, r_{i}, \beta_{i}])_{i=0}^{s}$ if $s \geq 0$.  
If $s = - \infty$, we take a simple stratum $[\Afra, 0, 0, \beta]$ with $\Afra$ maximal and $c_{0} = \beta_{0} = \beta \in \ofra_{F}^{\times}$, and then we can define subgroups $H^{1}(\beta, \Afra)$ and $J(\beta, \Afra)$ in $G$ for any case.  
Moreover, since $\Bfra_{0}$ is maximal, we also can define $\hat{J}(\beta, \Afra) = \Kfra(\Bfra_{0})J(\beta, \Afra)$.  

\begin{prop}
\begin{enumerate}
\item We have
\begin{enumerate}
\item $G^{0}(F)_{[x]} = \Kfra(\Bfra_{0})$, 
\item $G^{0}(F)_{x} = B_{0}^{\times} \cap \U(\Afra) = \U(\Bfra_{0})$, 
\item $G^{0}(F)_{x,0+} = B_{0}^{\times} \cap \U^{1}(\Afra)$, 
\item $\gfra^{0}(F)_{x} = B_{0} \cap \Afra = \Bfra_{0}$, and
\item $\gfra^{0}(F)_{x,0+} = B_{0} \cap \Pfra$.  
\end{enumerate}
\item For $i=1, \ldots, d$, we have
\begin{enumerate}
\item $G^{i}(F)_{x, \mathbf{s}_{i-1} } = B_{i}^{\times} \cap \U^{\lfloor (-v_{\Afra}(c_{i-1})+1)/2 \rfloor}(\Afra)$, 
\item $G^{i}(F)_{x, \mathbf{s}_{i-1}+ } = B_{i}^{\times} \cap \U^{\lfloor -v_{\Afra}(c_{i-1})/2 \rfloor +1}(\Afra)$, 
\item $G^{i}(F)_{x, \rbf_{i-1} } = B_{i}^{\times} \cap \U^{-v_{\Afra}(c_{i-1})}(\Afra)$, 
\item $G^{i}(F)_{x, \rbf_{i-1}+ } = B_{i}^{\times} \cap \U^{-v_{\Afra}(c_{i-1})+1}(\Afra)$, 
\item $\gfra^{i}(F)_{x, \mathbf{s}_{i-1} } = B_{i} \cap \Pfra^{\lfloor (-v_{\Afra}(c_{i-1})+1)/2 \rfloor}$, 
\item $\gfra^{i}(F)_{x, \mathbf{s}_{i-1}+ } = B_{i} \cap \Pfra^{\lfloor -v_{\Afra}(c_{i-1})/2 \rfloor +1}$, 
\item $\gfra^{i}(F)_{x, \rbf_{i-1} } = B_{i} \cap \Pfra^{-v_{\Afra}(c_{i-1})}$, and
\item $\gfra^{i}(F)_{x, \rbf_{i-1}+ } = B_{i} \cap \Pfra^{-v_{\Afra}(c_{i-1})+1}$.  
\end{enumerate}
\item For $i=0, \ldots, s$, we have
\begin{enumerate}
\item $G^{i}(F)_{x, \mathbf{s}_{i}+ } = B_{i}^{\times} \cap \U^{\lfloor -v_{\Afra}(c_{i})/2 \rfloor +1}(\Afra)$, 
\item $G^{i}(F)_{x, \rbf_{i} }  = B_{i}^{\times} \cap \U^{-v_{\Afra}(c_{i})}(\Afra)$,
\item $G^{i}(F)_{x, \rbf_{i}+ } = B_{i}^{\times} \cap \U^{-v_{\Afra}(c_{i})+1} (\Afra)$, 
\item $\gfra^{i}(F)_{x, \rbf_{i} } = B \cap \Pfra^{-v_{\Afra}(c_{i})}$, and
\item $\gfra^{i}(F)_{x, \rbf_{i}+ } = B \cap \Pfra^{-v_{\Afra}(c_{i})}$.  
\end{enumerate}
\end{enumerate}
\end{prop}

\begin{prf}
Similar to the proof of Proposition \ref{compofTLS}.  
\end{prf}

\begin{prop}
\begin{enumerate}
\item $K^{d}_{+} = H^{1}(\beta, \Afra)$, 
\item ${}^{\circ}K^{d} = J(\beta, \Afra)$, and
\item $K^{d} = \hat{J}(\beta, \Afra)$.  
\end{enumerate}
\end{prop}

\begin{prf}
Similar to the proof of Proposition \ref{compofHJ}.  
\end{prf}

Next, we construct a simple character in $\mathscr{C}(\beta, 0, \Afra)$ from $(\Phibf_{i})_{i}$.  

\begin{lem}
\label{lemforbeingsimp}
Suppose $s \geq 0$.  
For $i=0, 1, \ldots, s$, the following assertions hold.  
\begin{enumerate}
\item $\hat{\Phibf}_{i} |_{B_{i}^{\times} \cap H^{1}(\beta, \Afra)}$ factors through $\Nrd_{B_{i}/E_{i}}$.  
\item $\hat{\Phibf}_{i} |_{H^{t_{i}+1}(\beta, \Afra)} = \psi_{c_{i}}$, where $t_{i} = \lfloor -v_{\Afra}(c_{i})/2 \rfloor$.  
\item $H^{t_{i}+1}(\beta, \Afra) = H^{t_{i}+1}(\beta_{i}, \Afra)$ is normalized by $B_{i}^{\times} \cap \Kfra(\Afra)$.  
\item For any $g \in B_{i}^{\times} \cap \Kfra(\Afra)$ and $h \in H^{1}(\beta, \Afra) \cap {}^{g}H^{1}(\beta, \Afra)$, we have $\hat{\Phibf}_{i}(g^{-1}hg) = \hat{\Phibf}_{i}(h)$.  
\end{enumerate}
\end{lem}

\begin{prf}
We have $B_{i}^{\times} \cap H^{1}(\beta, \Afra) = G^{i}(F) \cap K^{d}$.  
By construction of $\hat{\Phibf}_{i}$ we have $\hat{\Phibf}_{i}|_{B_{i}^{\times} \cap H^{1}(\beta, \Afra)} = \hat{\Phibf}_{i}|_{G^{i}(F) \cap K^{d}} = \Phibf_{i}$.  
The map $\Phibf_{i}$ is a character of $G^{i}(F)$, and then $\Phibf_{i}$ factors through $\Nrd_{B_{i}/E_{i}}$ and (1) holds.  

We also have $H^{t_{i}+1}(\beta, \Afra) = K^{d} \cap G(F)_{x, \mathbf{s}_{i}+}$.  
Since $\Phibf_{i}|_{G^{i}(F)_{x,\mathbf{s}_{i}+:\rbf_{i}+}}$ is realized by $X_{c_{i}}^{*}$ by Proposition \ref{takec} (2) or (3), we have 
\[
	\Phibf_{i}(1+y) = \psi \circ \Tr_{E_{i}/F} \circ \Trd_{B_{i}/E_{i}} (c_{i}y) = \psi \circ \Trd_{A/F}(c_{i}y)
\]
for $y \in B_{i} \cap \Pfra^{t_{i}+1}=\gfra^{i}(F)_{x, \mathbf{s}_{i}+}$.  
We recall that $\pi_{i}:\gfra(F) = \gfra^{i}(F) \oplus \mathfrak{n}^{i}(F) \to \gfra^{i}(F)$ is the projection and
\[
	\hat{\Phibf}_{i}(1+y) = \Phibf_{i}(1+\pi_{i}(y)) = \psi \circ \Trd_{A/F}(c_{i}\pi_{i}(y))
\]
for $1+y \in K^{d} \cap G(F)_{x, \mathbf{s}_{i}+} = H^{t_{i}+1}(\beta, \Afra)$.  
However, we also can show $\Trd_{A/F}(c_{i}\pi_{i}(y)) = \Trd_{A/F}(c_{i}y)$ as in the proof of Proposition \ref{phiistheta1}.  
In conclusion, for $1+y \in H^{t_{i}+1}(\beta, \Afra)$ we obtain $\hat{\Phibf}_{i}(1+y) = \psi \circ \Trd_{A/F}(c_{i}y) = \psi_{c_{i}}(y)$ and (2) holds.  

Let $g \in B_{i}^{\times} \cap \Kfra(\Afra)$.  
We check that $g$ normalizes $H^{t_{i}+1}(\beta, \Afra)$.  
We consider two cases.  
First, suppose $i < d$.  
Then we have $H^{t_{i}+1}(\beta, \Afra) = G^{i+1}(F)_{x, \mathbf{s}_{i}+} \cdots G^{d}(F)_{x, \mathbf{s}_{d-1}+}$.  
Thus it suffices to show $g$ normalizes $G^{j}(F)_{x, \mathbf{s}_{j-1}+}$ for $j=i+1, \ldots, d$.  
However, we have
\[
	gG^{j}(F)_{x, \mathbf{s}_{j-1}+}g^{-1} = g \left( B_{j}^{\times} \cap \U^{t_{j-1}+1}(\Afra) \right) g^{-1} = ( gB_{j}^{\times} g^{-1} ) \cap (g \U^{t_{j-1}+1}(\Afra) g^{-1}).  
\]
Since $g \in B_{i}^{\times} \subset B_{j}^{\times}$ we have $gB_{j}^{\times} g^{-1} = B_{j}^{\times}$.  
Moreover, we also have $g \U^{t_{j-1}+1}(\Afra) g^{-1} = \U^{t_{j-1}+1}(\Afra)$ as $g \in \Kfra(\Afra)$.  
Therefore we obtain $gG^{j}(F)_{x, \mathbf{s}_{j}+}g^{-1} = B_{j}^{\times} \cap \U^{t_{j-1}+1}(\Afra) = G^{i}(F)_{x, \mathbf{s}_{j}+}$.  
Next, suppose $i = d = s$.  
Then we have $H^{t_{s}+1}(\beta, \Afra) = G^{d}(F)_{x, \mathbf{s}_{s}+} = \U^{t_{s}+1}(\Afra)$.  
Since $g \in \Kfra(\Afra)$, we obtain 
\[
	g H^{t_{s}+1}(\beta, \Afra) g^{-1} = g \U^{t_{s}+1}(\Afra) g^{-1} = \U^{t_{s}+1}(\Afra) = H^{t_{s}+1}(\beta, \Afra).  
\]
Therefore we obtain (3).  

Here, let $g$ be as above and $h \in H^{1}(\beta, \Afra)$.  
Since 
\[
H^{1}(\beta, \Afra) = \left(B_{i}^{\times} \cap H^{1}(\beta, \Afra) \right) H^{t_{i}+1}(\beta, \Afra), 
\]
we have $h=bh'$ for some $b \in B_{i}^{\times} \cap H^{1}(\beta, \Afra)$ and $h' \in H^{t_{i}+1}(\beta, \Afra)$.  
By the above argument, we have $h' \in H^{t_{i}+1}(\beta, \Afra) = g H^{t_{i}+1}(\beta, \Afra) g^{-1}$ and $h'$ is an element in $H^{1}(\beta, \Afra) \cap gH^{1}(\beta, \Afra) g^{-1}$.  
Then, $h \in H^{1}(\beta, \Afra) \cap g H^{1}(\beta, \Afra) g^{-1}$ if and only if $b \in H^{1}(\beta, \Afra) \cap g H^{1}(\beta, \Afra) g^{-1}$.  
Suppose $h \in H^{1}(\beta, \Afra) \cap g H^{1}(\beta, \Afra) g^{-1}$.  
Therefore we obtain
\begin{eqnarray*}
	\hat{\Phibf}_{i}(g^{-1}hg) & = & \hat{\Phibf}_{i} \left( (g^{-1}bg)(g^{-1}h'g) \right) 	\\
	& = & \hat{\Phibf}_{i}(g^{-1}bg) \hat{\Phibf}_{i}(g^{-1}h'g) = \Phibf_{i}(g^{-1}bg) \psi_{c_{i}}(g^{-1}h'g).  
\end{eqnarray*}
Here, since $\Phibf_{i}$ is a character of $G^{i}(F)=B_{i}^{\times}$ and $g \in B_{i}^{\times}$, we have $\Phibf_{i}(g^{-1}bg) = \Phibf_{i}(b)$.  
Moreover, since $c_{i}$ is an element in $E_{i}$, which is the center of $B_{i}$, we also have
\begin{eqnarray*}
	\psi_{c_{i}}(g^{-1}h'g) & = & \psi \circ \Trd_{A/F}(c_{i}g^{-1}h'g) = \psi \circ \Trd_{A/F}(g^{-1}c_{i}h'g) \\
	& = & \psi \circ \Trd_{A/F}(c_{i}h') = \psi_{c_{i}}(h').  
\end{eqnarray*}
Therefore we obtain $\hat{\Phibf}_{i}(g^{-1}hg) = \Phibf_{i}(b)\psi_{c_{i}}(h') = \hat{\Phibf}_{i}(bh') = \hat{\Phibf}_{i}(h)$, which implies (4).  
\end{prf}

\begin{prop}
We have $\prod_{i=0}^{d} \hat{\Phibf}_{i} \in \mathscr{C}(\beta, 0, \Afra)$.  
\end{prop}

\begin{prf}
If $s=-\infty$, then $\Phibf_{d}=1$ and $\hat{\Phibf}_{d}=1$, and then $\prod_{i=0}^{d} \hat{\Phibf}_{i} = 1 \in \mathscr{C}(\beta, 0, \Afra)$.  
Therefore we assume $s \in \Z$.  
If $d=s+1$, then $\Phibf_{d}=1$ and $\hat{\Phibf}_{i} = 1$ and we have $\prod_{j=i}^{d} \hat{\Phibf}_{j} = \prod_{j=i}^{s} \hat{\Phibf}_{j}$ for $i=0, \ldots, s$.  
Thus we show $\bar{\theta}_{i} := \prod_{j=i}^{s} \hat{\Phibf}_{j}|_{H^{t_{j}+1}(\beta, \Afra)} \in \mathscr{C}(\beta_{i}, \lfloor r_{i}/2 \rfloor, \Afra)$ by downward induction on $i=0, \ldots, s$.  

First, suppose $i=s$.  
Since $\beta_{s} = c_{s}$ is minimal over $F$, we need to check (1), (2) and (3) in Definition \ref{defofsimpch}.  
(2) is already shown as Lemma \ref{lemforbeingsimp} (1).  
Since $-v_{\Afra}(c_{s}) = -v_{\Afra}(\beta_{s}) = n$, we have $H^{t_{s}+1}(\beta, \Afra) = \U^{\lfloor n/2 \rfloor +1}(\Afra)$ and (3) is also shown as Lemma \ref{lemforbeingsimp} (2).  
Let $g \in B_{i}^{\times} \cap \Kfra(\Afra)$ and $h \in H^{t_{s}+1}(\beta, \Afra)$.  
Then $g^{-1}hg \in H^{t_{s}+1}(\beta, \Afra)$ by Lemma \ref{lemforbeingsimp} (3), and $\hat{\Phibf}_{i}(g^{-1}hg) = \hat{\Phibf}_{i}(h)$ by Lemma \ref{lemforbeingsimp} (4), which implies (1).  
Therefore $\hat{\Phibf}_{s} \in \mathscr{C}(\beta_{s}, t_{s}, \Afra)$.  

Next, suppose $0<i<s$.  
Since $k_{0}(\beta_{i-1}, \Afra) = v_{\Afra}(c_{i-1}) = -r_{i} > -n = v_{\Afra}(\beta_{i-1})$, the element $\beta_{i-1}$ is not minimal over $F$, and then we need to check (1), (2) and (4) in Definition \ref{defofsimpch}.  

To show (1), let $g \in B_{i-1}^{\times} \cap \Kfra(\Afra)$ and $h \in H^{t_{i-1}+1}(\beta, \Afra)$.  
Then $g^{-1}hg \in H^{t_{i-1}+1}(\beta, \Afra)$ by Lemma \ref{lemforbeingsimp} (3).  
For $j=i-1, \ldots, s$, we have $g \in B_{i-1}^{\times} \cap \Kfra(\Afra) \subset B_{j}^{\times} \cap \Kfra(\Afra)$.  
Therefore by Lemma \ref{lemforbeingsimp} (4) we have $\hat{\Phibf}_{j}(g^{-1}hg) = \hat{\Phibf}_{j}(h)$ and $\bar{\theta}_{i-1}(g^{-1}hg) = \prod_{j=i-1}^{s}\hat{\Phibf}_{j}(g^{-1}hg) = \prod_{j=i-1}^{s} \hat{\Phibf}_{j}(h) = \bar{\theta}_{i-1}(h)$, whence (1) holds.  

For $j=i-1, \ldots, s$, the restriction of $\hat{\Phibf}_{j}$ to $B_{j}^{\times} \cap H^{t_{i-1}+1}(\beta, \Afra)$ factors through $\Nrd_{B_{j}/E_{j}}$.  
Since $\Nrd_{B_{j}/E_{j}} |_{B_{i-1}^{\times}} = \Nrm_{E_{i-1}/E_{j}} \circ \Nrd_{B_{i-1}/E_{i-1}}$, the restriction of $\hat{\Phibf}_{j}$ to $B_{i-1}^{\times} \cap H^{t_{i-1}+1}(\beta, \Afra)$ factors through $\Nrd_{B_{i-1}/E_{i-1}}$.  
Then the character $\bar{\theta}_{i-1}=\prod_{j=i-1}^{s} \hat{\Phibf}_{j}|_{B_{i-1}^{\times} \cap H^{t_{i-1}+1}(\beta, \Afra)}$ also factors through $\Nrd_{B_{i-1}/E_{i-1}}$ and (2) holds.  

We show (4).  
We put $r'_{i-1}=0$ and $r'_{j}=r_{i}$ for $j=i, \ldots, s$.  
Then the sequence $([\Afra, n, r'_{(i-1)+i'}, \beta_{(i-1)+i'}])_{i'=0}^{s-i+1}$ is a defining sequence of $[\Afra, n, 0, \beta_{i-1}]$.  
Since $-k_{0}(\beta_{i-1}, \Afra) = r_{i}$, we have $\max \{ \lfloor r_{i-1}/2 \rfloor , \lfloor - k_{0}(\beta_{i-1}, \Afra)/2 \rfloor \} = \lfloor r_{i}/2 \rfloor = t_{i-1}$.  
Then $\bar{\theta}_{i-1}|_{H^{t_{i-1}+1}(\beta, \Afra)} = \bar{\theta}_{i} \hat{\Phibf}_{i-1}|_{H^{t_{i-1}+1}(\beta, \Afra)}$.  
The character $\bar{\theta}_{i}$ is an element in $\mathscr{C}(\beta_{i}, \lfloor r_{i}/2 \rfloor, \Afra)$ by induction hypothesis.  
On the other hand, $\hat{\Phibf}_{i-1}|_{H^{t_{i-1}+1}(\beta, \Afra)} = \psi_{c_{i-1}}$ by Lemma \ref{lemforbeingsimp} (2).  
Therefore (4) is shown and we complete the proof.  
\end{prf}

We put $\theta = \prod_{i=0}^{d} \hat{\Phi}_{i}$, and let $\eta_{\theta}$ be the Heisenberg representation of $\theta$.  

\begin{prop}
\label{betaext'}
$\kappa_{0} \otimes \cdots \otimes \kappa_{d}$ is an extension of $\eta_{\theta}$ to $K^{d}$.  
\end{prop}

\begin{prf}
Similar to the proof of Proposition \ref{betaext}.  
\end{prf}

\begin{thm}
\label{Main2}
Let $\left(x, (G^{i})_{i=0}^{d}, (\rbf_{i})_{i=0}^{d}, (\Phibf_{i})_{i=0}^{d}, \rho \right)$ be a Yu datum.  
Then there exists a maximal, tame simple type $(J, \lambda)$ associated with $[\Afra, n, 0, \beta]$ and a maximal extension $(\tilde{J}, \Lambda)$ of $(J, \lambda)$ such that
\begin{enumerate}
\item $\hat{J} :=\hat{J}(\beta, \Afra) = K^{d}$, and
\item $\rho_{d} = \cInd_{\tilde{J}}^{\hat{J}} \Lambda$.  
\end{enumerate}
\end{thm}

\begin{prf}
We can construct a tame simple stratum $[\Afra, n, 0, \beta]$ and a simple character $\theta \in \mathscr{C}(\beta, \Afra)$ as above.  
We take a $\beta$-extension $\kappa$ of $\eta_{\theta}$ and an extension $\hat{\kappa}$ of $\kappa$ to $\hat{J}$ by Lemma \ref{exist_ext_of_b-ext} (1).  
On the other hand, let $\kappa_{i}$ be the representation of $K^{d}$ as in Section 3 for $i=-1, 0, \ldots, d$.  
By Proposition \ref{betaext'}, the representation $\hat{\kappa}'=\kappa_{0} \otimes \cdots \otimes \kappa_{d}$ is an extension of a $\beta$-extension ${}^{\circ} \lambda$ of $\eta_\theta$ to $K^{d}$.  
Then by Lemma \ref{exist_ext_of_b-ext} (2), there exists a character $\chi$ of $\hat{J}/J^{1}(\beta, \Afra)$ such that $\hat{\kappa}' \cong \hat{\kappa} \otimes \chi$.  
The representation $\kappa_{-1}$ is the extension of $\rho$ to $K^{d}$, trivial on $K^{d} \cap G(F)_{x, 0+} = J^{1}(\beta, \Afra)$.  

We construct ``depth-zero part" $\sigma$ of a simple type from $\rho$.  
By Lemma \ref{depth0ofYu}, there exists a depth-zero simple type $(G^{0}(F)_{x}, \sigma^{0})$ of $G^{0}(F)$ and a maximal extension $(\tilde{J}^{0}, \tilde{\sigma}^{0})$ such that $\rho \cong \Ind_{\tilde{J}^{0}}^{G^{0}(F)_{[x]}} \tilde{\sigma}^{0}$.  
We put $\tilde{J} = \tilde{J}^{0}J=\tilde{J}^{0}J^{1}(\beta, \Afra)$.  
Since $J^{1}(\beta, \Afra) \cap G^{0}(F) = G^{0}(F)_{x, 0+}$, we have $\tilde{J}^{0}/G^{0}(F)_{x,0+} \cong \tilde{J}/J^{1}(\beta, \Afra)$ and we can extend $\tilde{\sigma}^{0}$ to $\tilde{J}$ as $\tilde{\sigma}$, which is trivial on $J^{1}(\beta, \Afra)$. We put $\sigma = \Res_{J}^{\tilde{J}} \tilde{\sigma}$.  
The representation $\sigma$ is an extension of $\sigma^{0}$ to $J$, trivial on $J^{1}(\beta, \Afra)$.  
Since $(G^{0}(F)_{x}, \sigma^{0})$ is a maximal simple type of depth zero and $\chi$ is a character of $\hat{J}$ trivial on $J^{1}(\beta, \Afra)$, the $J(\beta, \Afra)/J^{1}(\beta, \Afra)$-representation $\sigma \otimes \chi$ is cuspidal, and then $(J, \sigma \otimes \chi \otimes \kappa)$ is a simple type.  
By construction of $\tilde{J}$ and $\tilde{\sigma}$, the pair $(\tilde{J}, \tilde{\sigma} \otimes \Res_{\tilde{J}}^{\hat{J}}(\chi \otimes \hat{\kappa}))$ is a maximal extension of $(J, \sigma \otimes \chi \otimes \kappa)$.  
We put $\Lambda = \tilde{\sigma} \otimes \Res_{\tilde{J}}^{\hat{J}}( \chi \otimes \hat{\kappa})$.  

The representation $\kappa_{-1}$ is the extension of $\rho$ as $\kappa_{-1}$ is trivial on $K_{+}^{0}J^{1} \cdots J^{d} = J^{1}(\beta, \Afra)$, that is, the representation $\kappa_{-1}$ is $\rho$ regarded as a representation of $K^{d}=\hat{J}$ via $K^{0}/K_{+}^{0} = \Kfra(\Bfra)/\U^{1}(\Bfra) \cong K^{d}/(K_{+}^{0}J^{1} \cdots J^{d}) = \hat{J}/J^{1}(\beta, \Afra)$.  
Then we have $\kappa_{-1} \cong \cInd_{\tilde{J}}^{\hat{J}} \tilde{\sigma}$ by Lemma \ref{inf_and_ind} and
\[
	\cInd_{\tilde{J}}^{\hat{J}} \Lambda \cong (\cInd_{\tilde{J}}^{\hat{J}} \tilde{\sigma}) \otimes \chi \otimes \hat{\kappa} \cong \kappa_{-1} \otimes \kappa_{0} \otimes \cdots \otimes \kappa_{d} = \rho_{d}, 
\]
which finishes the proof.  
\end{prf}

\begin{cor}
\label{tame_and_esstame}
The set of essentially tame supercuspidal representations of $G$ is equal to the set of tame supercuspidal representations of $G$.  
\end{cor}

\begin{prf}
Let $\pi$ be an irreducible supercuspidal representation of $G$.  
Since $\cInd_{K^{d}(\Psi)}^{G} \rho^{d}(\Psi)$ is irreducible for any Yu's datum $\Psi$, $\pi$ is tame supercuspidal if and only if $\pi \supset \rho^{d}(\Psi)$ for some $\Psi$.  
However, by Theorem \ref{Main1} and \ref{Main2} it holds if and only if $\pi$ contains some compact induction of a maximal extension $(\tilde{J}, \Lambda)$ of a tame, maximal simple type, which is equivalent to $\pi$ is essentially tame by \ref{esstame}.  
\end{prf}

\begin{rem}
In the condition of Theorem \ref{Main1} or \ref{Main2}, suppose $G=\GL_{N}(F)$.  
Then we have $\tilde{J} = \tilde{J}(\lambda) = \hat{J}(\beta, \Afra)$ by Remark \ref{rem_for_Jhat} (3).  
Therefore $\tilde{J} = K^{d}(\Psi)$ and $\cInd_{\tilde{J}}^{K^{d}(\Psi)} \Lambda = \Lambda$, which leads to Theorem \ref{main_for_split_case}.  
\end{rem}
\section{Wild case} \label{secwild}

Let $[\Afra, n, 0, \beta]$ be a Bushnell-Kutzko simple stratum. For the purpose of the paper, we assumed that $F[\beta ]/F$ is tamely ramified. 
\begin{rem}(cf. also \cite{FGT})
\begin{enumerate}
\item If we remove the assumption that $F[\beta]/F$ is tame in our fixed simple stratum $[\Afra, n, 0, \beta]$, then the sequence of fields $E_{0}, \ldots, E_{s}$ attached to a defining sequence can not be chosen decreasing for $\subset$ in general. 
\label{omit_tame_assump}
It always decreases for $[ \bullet : F]$.  
\item In a certain sense, we have explained that Bushnell--Kutzko and S\'echerre's constructions are compatible with Yu's construction as they essentially are the same on their common domain of definition.  Does there exist a construction generalizing both of them in a single formalism?
\item If one tries to obtain generalization of these approaches, one has to remove (among other things) the axiom of inclusions in the twisted Levi sequence by $(\ref{omit_tame_assump})$ of this remark and by definition of $\overrightarrow{G}$.  
This implies that one can not expect a factorable construction $\rho_{d}=\otimes \kappa^{i}$ as Yu's one.  
\end{enumerate}
\end{rem}

\bigbreak\bigbreak
\noindent Arnaud Mayeux\par
\noindent Einstein Institute of Mathematics, The Hebrew University of Jerusalem, Givat Ram. Jerusalem,
9190401, Israel \par
\noindent E-mail address: \texttt{arnaud.mayeux@mail.huji.ac.il}

\bigbreak\bigbreak
\noindent Yuki Yamamoto\par
\noindent National Institute of Technology, Niihama College, 
7-1 Yakumo-cho, Niihama City, 
Ehime, 792-8580, Japan\par
\noindent E-mail address: \texttt{y.yamamoto@niihama-nct.ac.jp}


\begin{thebibliography}{99}
\bibitem{Ad} J. D. Adler, \emph{Refined anisotropic $K$-types and supercuspidal representations}, Pacific J. Math., (1) \textbf{7} (1998), 1-32.
\bibitem{BH}  C. J. Bushnell and G. Henniart. \emph{The essentially tame Jacquet--Langlands correspondence for inner forms of $\GL(n)$}, Pure Appl. Math. Q., (3, Special Issue: In honor of Jacques Tits) \textbf{7} (2011), 469-538.  
\bibitem{BK1} C. J. Bushnell and P. C. Kutzko, \emph{The admissible dual of $\GL(N)$ via compact open subgroups}, Annals of Mathematics Studies 129 (Princeton University Press, 1993).  
\bibitem{BK94} C. J. Bushnell and P. C. Kutzko, \emph{Simples types in GL(N): Computing conjugacy classes}, Contemporary Mathematics Volume 177 (1994).  

\bibitem{Br} P. Broussous, \emph{Hereditary orders and embeddings of local fields in simple algebras}, J. Algebra, (1) \textbf{204} (1998), 324-336.  

\bibitem{BL} P. Broussous and B. Lemaire, \emph{Building of $\GL(m, D)$ and centralizers}, Transform. Groups, (1) \textbf{7} (2002), 15-50.  

\bibitem{BT} F. Bruhat and J. Tits, \emph{Groupes r\'eductifs sur un corps local}, Inst. Hautes \'Etudes Sci. Publ. Math.,  \textbf{41} (1972), 5-251.  
\bibitem{BT2} F. Bruhat and J. Tits, \emph{Groupes r\'eductifs sur un corps local. II. Sch\'emas en groupes.  Existence d'une donn\'ee radicielle valu\'ee}, Inst. Hautes \'Etudes Sci. Publ. Math., \textbf{60} (1984), 197-376. 


\bibitem{SGA3} M.\,Demazure, A.\,Grothendieck: {\it S\'eminaire de G\'eom\'etrie Alg\'ebrique du Bois Marie - 1962-64 - Sch\'emas en groupes - (SGA 3)}.  Lecture Notes in Mathematics 151-152-153, Springer-Verlag, Berlin-Heidelberg-New York, 1970.
\bibitem{DMdS} A.\,Dubouloz, A.\,Mayeux and J.\,P.\,dos\,Santos, \emph{A survey on algebraic dilatations}, 2023


\bibitem{Fi2} J. Fintzen, \emph{On the construction of tame supercuspidal representations}, Compositio Mathematica 157 (2021), no. 12, pp. 2733–2746. 
 
 \bibitem{Fi} J. Fintzen, \emph{Types for tame p-adic groups.} Annals of Mathematics 193 no. 1 (2021), pp. 303-346. 
 
 \bibitem{Fide21} J. Fintzen, \emph{Tame cuspidal representations in non-defining characteristics}, Michigan Math.
J. 72 (2022), 331–342. 


\bibitem{GSZ} M. Grabitz, A. J. Silberger and E.-W. Zink, \emph{Level zero types and Hecke algebras for local central simple algebras}, J. Number Theory (1) \textbf{77} (1999), 1-26.  
\bibitem{HM} J. Hakim and F. Murnaghan, \emph{Distinguished tame supercuspidal representations}, Int. Math. Res. Pap. ImRP, \textbf{2}(2008), Art. ID rpn005, 166. 



\bibitem{Ho} R. E. Howe, \emph{Tamely ramified supercuspidal representations of $\GL_N$}, Pacific J. Math, \textbf{2}(1977),73(2):437-460. 

\bibitem{FKS} T. Kaletha, J. Fintzen and L. Spice, \emph{A twisted Yu construction, Harish-Chandra characters, and endoscopy}, 
Duke Math. J. 172 (2023), no. 12, 2241–2301.


\bibitem{KP} T. Kaletha and G. Prasad, \emph{Bruhat-Tits Theory: A New Approach}, New Mathematical Monographs, Series Number 44, 2022.
  
\bibitem{Ki} J.-L. Kim, \emph{Supercuspidal representations: an exhaustion theorem,} J. Amer. Math. Soc. 20 (2007),
no. 2, 273–320.

  

\bibitem{FGT} A. Mayeux, \emph{Comparison of Bushnell-Kutzko and Yu’s constructions of
supercuspidal representations}, p. 2758 in \emph{J. Fintzen, W.-T. Gan, S. Takeda, New Developments in Representation Theory of p-adic Groups}, Oberwolfach Rep. 16 (2019), no. 4, pp. 2739–2819.

\bibitem{Ma23d} A.\,Mayeux, \emph{Multi-centered dilatations, congruent isomorphisms and Rost double deformation spaces}, arXiv:2303.07712, 2023.

\bibitem{MRR}  A.\,Mayeux,  T.\,Richarz and  M.\,Romagny, \emph{N\'eron blowups and low-degree cohomological applications}, Algebraic Geometry 10 (6) (2023) 729–753.




 
\bibitem{MP1} A. Moy and G. Prasad, \emph{Unrefined minimal $K$-types for $p$-adic groups}, Invent. Math., (1-3) \textbf{116} (1994), 393-408.  
\bibitem{MP2} A. Moy and G. Prasad, \emph{Jacquet functors and unrefined minimal $K$-types}, Comment. Math. Helv., (1) \textbf{71} (1996), 98-121.  

\bibitem{Neuk} J. Neukirch, \emph{Algebraic Number Theory}, Springer-Verlag, Translated from the 1992 German original, (1999). 

\bibitem{S1} V. S\'echerre, \emph{Repr\'esentations lisses de $\GL(m, D)$. I. caract\`eres simples}, Bull. Soc. math. France \textbf{132} (2004), no. 3, 327-396. 
\bibitem{S2} V. S\'echerre, \emph{Repr\'esentations lisses de $\GL(m, D)$. II. $\beta$-extensions}, Compos. Math. \textbf{141} (2005), no. 6, 1531-1550.  
\bibitem{S3} V. S\'echerre, \emph{Repr\'esentations lisses de $\GL(m, D)$. III. Types simples}, Ann. Sci. \'Ecole Norm. Sup. (4) \textbf{38} (2005), no. 6, 951-977.  
\bibitem{SS} V. S\'echerre and S. Stevens, \emph{Repr\'esentations lisses de $\GL(m, D)$. IV. Repr\'esentations supercuspidales}, J. Inst. Math. Jussieu \textbf{7} (2008), no. 3, 527-574.  
\bibitem{Ro} R. Steinberg, \emph{Torsion in reductive groups}, Advances in Math., \textbf{15} (1975), 63-92.  
\bibitem{Yu} J.-K. Yu, \emph{Construction of tame supercuspidal representations}, J. Amer. Math. Soc. \textbf{14}
(2001), no. 3, 579-622 (electronic).
\bibitem{Yu2} J.-K.\,Yu, \emph{Smooth models associated to concave functions in Bruhat-
Tits theory}, Autour des sch\'emas en groupes. Vol. III, 227--258, Panor. Synth\`eses \textbf{47}, Soc. Math. France, Paris, 2015. 
\end{thebibliography}
\end{document}